\documentclass[11pt, reqno]{amsart}
\usepackage{amsthm,amssymb,mathrsfs,mathtools}
\usepackage{amsmath,amssymb,amsfonts,graphicx,color}
\usepackage{amssymb,geometry}
\usepackage[shortlabels]{enumitem}
\usepackage{empheq}
\usepackage[T1]{fontenc}

\usepackage[notref,notcite,final]{showkeys}
\mathtoolsset{showonlyrefs}
\usepackage{fullpage,color}

\newtheorem{mainthm}{Theorem}
\newtheorem{theorem}{Theorem}[section]
\newtheorem*{theorem*}{Theorem}

\newtheorem{lemma}[theorem]{Lemma}
\newtheorem{proposition}[theorem]{Proposition}
\newtheorem*{proposition*}{Proposition}

\newtheorem*{conjecture*}{Conjecture}

\theoremstyle{definition}
\newtheorem{definition}[theorem]{Definition}
\newtheorem{remark}[theorem]{Remark}

\numberwithin{equation}{section}

\def\bR {\mathbb{R}}

\def\cE {\mathcal{E}}

\def\scrL{\mathscr{L}}

\def\la {\langle}
\def\ra {\rangle}

\newcommand{\tx}[1]{\mathrm{#1}}
\newcommand{\wto}{\rightharpoonup}

\newcommand{\wt}[1]{\widetilde{#1}}
\newcommand{\bs}[1]{\boldsymbol{#1}}

\newcommand{\sh}[1]{#1^\sharp}

\newcommand{\nt}[1]{(#1^\sharp)^\perp}

\newcommand{\spn}{\operatorname{span}}

\renewcommand{\ker}{\operatorname{ker}}

\newcommand{\eee}{\mathrm e}

\newcommand{\ud}{\mathrm{\,d}}
\newcommand{\vd}{\mathrm{d}}

\newcommand{\vD}{\mathrm{D}}
\newcommand{\dd}[1]{{\frac{\vd}{\vd{#1}}}}

\definecolor{deepgreen}{cmyk}{1,0,1,0.5}

\newcommand{\E}{\mathcal{E}}

\newcommand{\R}{\mathbb{R}}
\newcommand{\Sp}{\mathbb{S}}

\newcommand{\p}{\partial}

\makeatletter

\newcommand{\Rmnum}[1]{\expandafter\@slowromancap\romannumeral #1@}
\makeatother

\newcommand{\ang}[1]{\left\langle{#1}\right\rangle}
\newcommand{\abs}[1]{\left\lvert{#1}\right\rvert}

\newcommand{\EQ}[1]{\begin{equation}\begin{split} #1 \end{split}\end{equation}}

\newcommand{\Del}[1]{}

\numberwithin{equation}{section}

\newcommand{\mand}{{\ \ \text{and} \ \  }}

\definecolor{green}{rgb}{0,0.8,0} 

\newcommand{\eps}{\epsilon}

\title[Dynamics kink-antikink pairs] {Dynamics of strongly interacting kink-antikink pairs \\
for scalar fields on a line}

\author{Jacek Jendrej}
\address{CNRS \& LAGA (UMR 7539 CNRS), Universit\'e Paris 13 -- Sorbonne Paris Cit\'e, Universit\'e Paris~8,
99 avenue Jean-Baptiste Cl\'ement, 93430 Villetaneuse, France}
\email{jendrej@math.univ-paris13.fr}

\author{Micha{\l} Kowalczyk}
\address{Departamento de Ingenier\'{\i}a Matem\'atica and Centro
de Modelamiento Matem\'atico (UMI 2807 CNRS), Universidad de Chile, Casilla
170 Correo 3, Santiago, Chile.}
\email {kowalczy@dim.uchile.cl}

\author{Andrew Lawrie}
\address{Department of Mathematics, Massachusetts Institute of Technology, 77 Massachusetts Ave, 2-267, Cambridge, MA 02139, U.S.A.}
\email{alawrie@mit.edu}



\thanks{J. Jendrej is supported by ANR-18-CE40-0028 project ESSED and Chilean projects FONDECYT 1170164 and France-Chile ECOS-Sud C18E06 project. M. Kowalczyk was partially funded by Chilean research grants FONDECYT 1170164,  France-Chile ECOS-Sud C18E06 and CMM Conicyt PIA AFB170001. A. Lawrie is supported by NSF grant DMS-1700127 and a Sloan Research Fellowship. The authors thank Y. Martel for a helpful observation that was used in the proof of Lemma~\ref{lem:phi4-lin-dt}.  }

\begin{document}

\begin{abstract}
This paper concerns classical nonlinear scalar field models on the real line.
If the potential is a symmetric double-well, such a model admits static solutions called kinks and antikinks,
which are perhaps the simplest examples of topological solitons.
We study pure kink-antikink pairs, which are solutions that converge in one infinite time direction
to a superposition of one kink and one antikink, without radiation.
Our main result is a complete classification of all kink-antikink pairs in the strongly interacting regime,
which means the speeds of the kinks tend asymptotically to zero.
We show that up to translation there is exactly one such solution,
and we give a precise description of the dynamics of the kink separation.
\end{abstract}

\maketitle

\section{Introduction}
\label{sec:intro}
\subsection{Setting of the problem}
\label{ssec:setting}

We study scalar field equations on the real line. 
Let $U: \bR \to [0, +\infty)$ be a function of class $C^\infty$ and consider the Lagrangian action, 
\begin{equation}
\label{eq:lagrange}
\scrL(\phi) = \int_{-\infty}^{\infty}\int_{-\infty}^{\infty} \Big(\frac 12(\partial_t \phi)^2 - \frac 12 (\partial_x\phi)^2 - U(\phi)\Big)\, \ud  x \ud t,
\end{equation}  
for real valued functions $\phi = \phi(t, x)$. The Euler-Lagrange equation associated to $\scrL$ is the nonlinear wave equation, 
\begin{equation}
\label{eq:phi4}
\partial_t^2 \phi(t, x) - \partial_x^2 \phi(t, x) + U'(\phi(t, x)) = 0, \qquad (t, x) \in \bR\times \bR,\ \phi(t, x) \in \bR, 
\end{equation}
We will study~\eqref{eq:phi4} for potentials $U$ that are even functions taking the global minimal value $U_{\min} = 0$,  and such that there are distinct real numbers $\phi_+ > 0$ and $\phi_- = -\phi_+$ so that 
\EQ{ \label{eq:U} 
&U(\phi_{-}) = U(\phi_+) = U_{\min} =  0, \\
&U(\phi) > 0 \text{ for }\phi \in (\phi_-, \phi_+), \\
&U''(\phi_-)  = U''(\phi_+) > 0.
}
Two classically studied examples of~\eqref{eq:phi4} with potentials as in~\eqref{eq:U} are the \emph{sine-Gordon equation}, 
\EQ{ \label{eq:sg} 
\partial_t^2 \phi(t, x) - \partial_x^2 \phi(t, x) -\sin \phi(t, x) = 0,
}
where we have taken $U( \phi) = 1+ \cos \phi$ and $\phi_+ = \pi$ above, and the \emph{$\phi^4$ model}, 
\EQ{ \label{eq:phi4-m} 
\partial_t^2 \phi(t, x) - \partial_x^2 \phi(t, x) - \phi(t, x) + \phi(t, x)^3 = 0
}
where $U( \phi) = \frac{1}{4}(1- \phi^2)^2$ and $\phi_+ = 1$. 

The potential energy $E_p$, the kinetic energy $E_k$, and the total energy $E$ associated with the equation \eqref{eq:phi4} are given by 
\begin{align}
E_p( \phi)  &= \int_{-\infty}^{+\infty}\Big(\frac 12 (\partial_x\phi)^2 + U(\phi)\Big)\ud x, \\
E_k(\phi) &= \int_{-\infty}^{+\infty}\frac 12(\partial_t \phi)^2\ud x, \\
E(\phi)  &= \int_{-\infty}^{+\infty}\Big(\frac 12(\partial_t \phi)^2+\frac 12 (\partial_x\phi)^2 + U(\phi)\Big)\ud x.
\end{align}
We say that a solution to \eqref{eq:phi4} is in the energy space if $E(\phi)$ is finite. For such a  solution the energy is conserved, i.e.,  $E(\phi(t,\cdot))=constant$. By a solution $\phi(t, x)$ of~\eqref{eq:phi4}, we always mean a strong solution in the energy space.
By standard arguments, the Cauchy problem for \eqref{eq:phi4} is locally well-posed for initial data $( \phi_0, \phi_1) \in H^1(\bR) \times L^2(\bR)$,
and globally well-posed under additional assumptions on $U$,
for instance if $U$ is globally Lipschitz or if $\lim_{\phi \to \pm\infty}U(\phi) = \infty$. 

Stationary solutions of  \eqref{eq:phi4}  are the critical points of the potential energy. The trivial ones include the vacuum fields $\phi(t, x) = \phi_{\pm}$, which are global minima of $E_p$. Importantly, there are also non-constant static solutions $\phi(t, x)$ called \textit{kinks} 
connecting the two vacua, that is for instance 
\EQ{ \label{eq:connect} 
\lim_{x \to -\infty}\phi(t, x) = \phi_- \mand \lim_{x \to \infty}\phi(t, x) = \phi_+ \quad \forall \, t \in \R.
}
All of these solutions are given by
\begin{equation}
\label{eq:kink}
\phi(t, x) = H(x - a),
\end{equation}
where $H(x)$ is an increasing, smooth, odd function that minimizes the potential energy restricted to the functions $\phi(x)$ satisfying~\eqref{eq:connect}, and $a \in \R$ is a parameter.
  For the sine-Gordon equation~\eqref{eq:sg} the kink is given by $H(x) = 4 \arctan(\eee^x) - \pi$ and for the $\phi^4$ model we have $H(x) = \tanh(x/ \sqrt{2})$. We will study the function $H$ for general $U$ as in~\eqref{eq:U} in detail in Section \ref{ssec:stationary}. 

In this paper we agree that solutions of the form \eqref{eq:kink} that are increasing will be called  {\it kinks} and those that are decreasing (i.e., that connect from $\phi_+$ at $- \infty$ to $\phi_-$ at $+ \infty$)  will be called   {\it antikinks}. The latter are all given by $\phi(t, x) = H(-x + a) = -H(x - a)$, where the last equality follows from the symmetry of $U$. 
ODE analysis shows that besides the vacuum fields, the kinks, and the antikinks, no other finite potential energy stationary solutions such that $\phi_-\leq \phi\leq \phi_+$ exist. 
We note that equation \eqref{eq:phi4} is invariant by Lorentz transformations and 
applying a  Lorentz boost we obtain moving kinks and antikinks:
\begin{equation}
\phi(t, x) = H(\gamma(x - vt - a)),\qquad \phi(t, x) = H(\gamma({-}x + vt + a))
\end{equation}
where $v \in (-1, 1)$ and $\gamma = (1 - v^2)^{-\frac 12}$.

Kinks and  antikinks  are the simplest examples of topological solitons (they are one-dimensional) and this perhaps explains why the wave equation \eqref{eq:phi4} is widely studied both as a model problem in physics and due to its own merit as an interesting and challenging mathematical problem. For example, the question of nonlinear stability of the kink for the $\phi^4$-model~\eqref{eq:phi4-m} is classical, but still open for general smooth perturbations; see the recent work of the second author with Martel and Mu\~noz~\cite{KMM} where stability of the $\phi^4$ kink was proved under odd perturbations. For some other, special potentials this problem was studied in \cite{MR2770013}, \cite{MR2835867}. 
On the mathematical physics side, we refer the reader to~\cite{PhysRevD.20.3120}, \cite{MS}
and the references therein for specific examples and their motivations.

\subsection{Main results}
\label{ssec:results}
In this paper we consider the question of kink-antikink solutions to~\eqref{eq:phi4} in what we call the {\it strongly interacting} regime. Multi-kinks are informally defined as solutions that converge to a superposition of a finite number of kinks and antikinks, without radiation, as $t \to \infty$. We will define ``strongly interacting'' precisely below, but informally this means the special class of kink-antikink pairs for which the speeds of both the kink and the antikink tend to zero as $t \to \infty$. An interesting aspect of this regime is that the dynamics are driven solely by nonlinear interactions between the kink and antikink. This is in contrast to a multi-kink configuration consisting of boosted kinks and antikinks (i.e., the kinks have a nontrivial asymptotic velocities),  where the nonlinear interactions between the kinks are negligible as compared to the internal dynamics of each kink determined by the Lorentz boost. 

One motivation for considering kink-antikink pairs in the strongly interacting regime is that they exhibit the following threshold behavior. The Bogomolny structure of the energy (see Section~\ref{ssec:coer-kink}) implies that the kink and antikink are the minimal energy configurations connecting two distinct vacuua. 
It readily follows that any topologically trivial solution (i.e., one that tends to the same vacuum, say $\phi_+$, as $x \to \pm \infty$) with energy strictly less than \emph{twice} the potential energy of the kink must remain uniformly bounded away from the other vacuum point, $\phi_-$.  It is natural to ask if there exist topologically trivial solutions with the least possible energy, $2E_p(H)$, that reach (or rather come arbitrarily close to) two distinct vacuum points. As a first candidate for such a threshold solution one can consider a superposition of a well separated kink and antikink. That is, we consider the function 
%
\[
w(x;a)=\phi_+-H(x+a)+H(x-a),\qquad a \gg 1.
\]
Note that $w$ satisfies, with some $c>0$,
\[
-\partial_{xx}w+U'(w)=O(e^{\,-ca}), \quad 
E_p(w(\cdot;a))=2 E_p(H) + O(e^{\,-ca}).
\]
In other words  $w$ 
is ``almost'' a stationary solution of \eqref{eq:phi4} when $a\gg 1$, with energy nearly equal to $2 E_p(H)$.
We have $\lim_{x \to -\infty}w(x; a) = \phi_+ = \lim_{x \to \infty}w(x; a)$,
while $w(0; a) = \phi_+ - 2H(a) = \phi_- + 2(\phi_+ - H(a))$ is close to the other vacuum $\phi_-$ when $a \gg 1$.
This motivates the study of the existence of an exact threshold solution of the form $\phi(t, x) = w(x; a(t)) + \eps(t, x)$ where $a(t) \to \infty, \abs{a'(t)} \to 0$, and $\| (\eps, \p_t \eps)(t) \|_{H^1 \times L^2} \to 0$ as $t \to \infty$. We make the  following more general definition. 





\begin{definition} \label{d:kak} 
We say that a solution $\phi(t, x)$ of \eqref{eq:phi4} is a \emph{strongly interacting kink-antikink pair}
if there exist real-valued functions $x_1(t)$ and $x_2(t)$ such that
\begin{gather}
\label{eq:norm-to-zero}
\lim_{t \to \infty} \Big( \|\partial_t \phi(t)\|_{L^2} + \|\phi(t) - (\phi_+ - H(\cdot - x_1(t)) + H(\cdot - x_2(t)))\|_{H^1}\Big) = 0, \\
\label{eq:dist-to-infty-0}
\lim_{t \to \infty} \big(x_2(t) - x_1(t)\big) = \infty.
\end{gather}

\end{definition}
We remark that if $\phi(t, x)$ is a strongly interacting kink-antikink pair, then
\EQ{
E( \phi) = 2 E_p(H), 
}
and would thus be a  topologically trivial solution with the minimal energy needed to (asymptotically) connect two distinct vacua, i.e., a threshold solution.
\begin{remark}
It is not difficult to see, using the analysis from Section~\ref{ssec:coer-kink} below, that one can equivalently
define a strongly interacting kink-antikink pair as a solution $\phi(t, x)$ of \eqref{eq:phi4}
such that $\lim_{x\to -\infty}\phi(t, x) = \lim_{x \to \infty}\phi(t, x) = \phi_+$, $E(\phi) = 2E_p(H)$
and $\lim_{t \to \infty}\phi(t, x_0(t)) = \phi_-$ for some real-valued function $x_0(t)$.\end{remark}
Our goal is to find and classify all strongly interacting kink-antikink pairs.   
 
 Before stating the main theorems we introduce the following explicit constants. Given $U$ as in~\eqref{eq:U}, we define, 
\begin{equation}
\label{eq:k-def}
\kappa := \exp\bigg(\int_0^{\phi_+} \bigg( \frac{\sqrt{U''(\phi_+)}}{\sqrt{2U(y)}} - \frac{1}{\phi_+ -y} \bigg)\ud y\bigg),
\end{equation}
and
\begin{equation}
\label{def AAA}
A:= \phi_+ \sqrt[4]{U''(\phi_+)}\Big(\int_{0}^{\phi_+}\sqrt{2U(y)}\ud y\Big)^{-\frac 12}\kappa.
\end{equation}
With this notation in hand we can state our main result. 
\begin{mainthm}[Existence and uniqueness of the strongly interacting kink-antikink pair] 
\label{thm:main}
There exist a $C^1$ function $x(t)$ and a solution $\bs \phi_{(2)}(t, x)$ of \eqref{eq:phi4}
such that for all $\epsilon>0$ and all $t \geq T_0 = T_0(\epsilon)$,   
\begin{equation}
\label{eq:x-asympt}
\big|x(t) - (U''(\phi_+))^{-\frac 12}\log(At)\big| \leq t^{-2 + \epsilon}, \qquad \big|x'(t) - (U''(\phi_+))^{-\frac 12}t^{-1}\big| \leq t^{-3 + \epsilon}
\end{equation}
and 
\begin{equation}
\label{eq:error-est}
\begin{aligned}
&\big\|\bs \phi_{(2)}(t) - \big(\phi_+ - H(\cdot + x(t)) + H(\cdot - x(t))\big)\big\|_{H^1} \\
&+ \big\|\partial_t \bs\phi_{(2)}(t) + x'(t)\big(\partial_x H(\cdot + x(t)) + \partial_x H(\cdot - x(t))\big) \big\|_{L^2} \leq t^{-2 + \epsilon}.
\end{aligned}
\end{equation}
Moreover, $\bs \phi_{ (2)}$ is the {\it unique} strongly interacting kink-antikink pair up to translation, i.e.,  if $\phi(t, x)$ is \textit{any} 
strongly interacting kink-antikink pair, then there exist $t_0, x_0 \in \R$ so that 
\EQ{
\phi(t, x) = \bs \phi_{(2)}(t- t_0, x-x_0). 
}
\end{mainthm}

\begin{remark} 
We expect that the subset of the energy space given by $\mathcal M=\{\bs \phi_{(2)}(t-t_0, x-x_0)\mid (t_0, x_0)\in \bR\times\bR\}$ is in fact a {\it smooth} two dimensional manifold,  but we chose not to pursue this issue here.
\end{remark} 
\begin{remark} 
One can observe in \eqref{eq:x-asympt} that the main order term of $x'(t)$, namely $(U''(\phi_+))^{-\frac 12}t^{-1}$,
is the time derivative of $(U''(\phi_+))^{-\frac 12}\log(A t)$, which is the main order term of $x(t)$.
Similarly, in the estimate \eqref{eq:error-est} the term $x'(t)\big(\partial_x H(\cdot + x(t)) + \partial_x H(\cdot - x(t))\big)$
in the second line is the time derivative of the term ${-}\big(\phi_+ - H(\cdot + x(t)) + H(\cdot - x(t))\big)$ from the first line. Thus $\bs \phi_{(2)}(t, x)$ is in fact a strongly interacting kink-antikink pair in the sense of Definition~\ref{d:kak}. 
Such solutions are discussed in the mathematical physics literature.
For instance, \cite[Chapter 5.2]{MS} contains formal and numerical predictions about the evolution of an initial configuration
composed of a stationary kink and antikink placed at a large distance.
As we make the initial separation tend to infinity, the corresponding solutions converge to strongly interacting kink-antikink pairs.

We highlight the uniqueness statement in Theorem~\ref{thm:main}, which is new even for the completely integrable sine-Gordon equation; see the further discussion of this case below. 
\end{remark} 

\begin{remark} 
The sine-Gordon equation~\eqref{eq:sg} is a very special case of~\eqref{eq:phi4} as it is a canonical example of a completely integrable equation and one can write down explicit  solutions. 
An example of a strongly interacting kink-antikink pair is furnished by 
\[
\bs \phi_{SG, (2)}(t,x)= \pi - 4\arctan\left(\frac{t}{\cosh x}\right),
\]
and the family $\mathcal M$  of such pairs is given by time and space translations of $\phi_{SG,(2)}$.
Although in this case $\mathcal M$ is explicit, the uniqueness part of our theorem is novel and does not seem to follow directly from the fact that the sine-Gordon  equation is completely  integrable.  

Note that for $t\gg 1$  we have (uniformly in $x$)
\[
\bs \phi_{SG, (2)}(t,x) \approx  \pi -  4 \arctan \big(e^{\,x+\log 2t}\big)+ 4\arctan\big(e^{\,x-\log 2t}\big) .
\]
As expected $\bs \phi_{SG, (2)}$ is  for large positive  times approximated by the superposition of the sine-Gordon antikink $-H(x)=\pi - 4\arctan(e^{\,x})$ and the kink $H(x)=4\arctan(e^{\,x}) - \pi$, shifted respectively to  $x_1(t)=-\log 2t$ and $x_2(t)=\log 2t$. 

\end{remark} 

\begin{remark}
Kink-antikink pairs in the strongly interacting regime considered in Theorem~\ref{thm:main} are threshold solutions in the sense that they have the minimal energy $E = K E_p(H)$ needed to contain $K$ distinct kink structures.  Alternatively, one could consider solutions that are approximately the superposition of Lorentz boosted kinks and antikinks with nontrivial velocities, which we dub the weakly interacting regime. Any $K$-kink solution of the latter type would have nontrivial asymptotic kinetic energy, and thus total energy strictly above $K E_p(H)$. 
The weakly interacting regime should be  accessible given the existing literature (or via the techniques introduced in this paper), in particular given the landmark works of Merle~\cite{Merle90}, Martel~\cite{Martel05}, and Martel, Merle~\cite{MM06}, who proved the existence of $N$-soliton solutions to g-KdV and NLS with \textit{distinct, nontrivial velocities};   see also Martel, Merle, Tsai~\cite{MaMeTs} and C\^ote, Martel, Merle~\cite{CMM11}.  Note that in~\cite{Martel05}, Martel also established uniqueness of weakly interacting  $N$-soliton for g-KdV for each given set of distinct velocities.   In the context of nonlinear waves, see the work of C\^{o}te, Mu\~noz~\cite{CM}, who constructed $N$-solitons solutions with distinct velocities for nonlinear Klein-Gordon equations.  We emphasize a key distinction:  in the strongly interacting regime considered here, the dynamics  are driven solely by nonlinear interactions between the kinks, whereas in the weakly interacting regime the soliton interactions are negligible to main order.
Probably the first rigorous construction of a strongly interacting multi-soliton was obtained by Martel and Rapha\"el \cite{MaRa18}.
\end{remark}


\begin{remark}
The solution $\bs \phi_{(2)}$ in Theorem~\ref{thm:main} contains an antikink moving to the left and a kink moving to the right in forward time. Since~\eqref{eq:phi4} is time-reversible, one may ask what happens when time is run backwards and the kink and antikink structures move towards each other and eventually collide (i.e., the distance $x_2(t) - x_1(t)$ becomes $\simeq 1$). The folklore conjecture is that whereas soliton collisions are known to be {\it elastic} for the integrable sine-Gordon equation, collisions should  be {\it inelastic} for equations that are not completely integrable, i.e. for the $\phi^4$-model~\eqref{eq:phi4-m} and for the general equation~\eqref{eq:phi4}. Here {\it inelastic} means that the collision results in some quantum of energy radiating away freely as $ t \to - \infty$. 

The threshold solution $\bs \phi_{(2)}$ is an interesting solution for which to consider the collision problem. Indeed, if any part of the solution breaks off as free radiation after the collision, the fact that it is the minimal energy topologically trivial kink-anitkink structure suggests that  the entire solution should disperse as $t \to - \infty$. Such a phenomenon was established by the first and third authors for the minimal energy $2$-bubble configuration for the $k$-equivariant $\R^2  \to \Sp^2$ wave maps problem in~\cite{JL1}. A key ingredient in~\cite{JL1} is  a so-called threshold theorem (proved earlier in~\cite{CKLS1}), which says that any topologically trivial $k$-equivariant wave map with energy less than twice the energy of the $k$-equivariant harmonic map $Q$ must disperse freely in both time directions.
However, an analogous threshold theorem for~\eqref{eq:phi4-m} does not seem within reach. 
Even the small energy problem is extremely challenging given the slow dispersive decay of the $1d$ Klein-Gordon waves (which appear after linearization about the vacua $\phi_{\pm}$); see Delort~\cite{Delort} and Hayashi-Naumkin~\cite{HN1, HN2} on the modified scattering procedure for NLKG solutions with cubic and quadratic nonlinearities and small, decaying initial data.
In general, there is very little known about the collision problem.
We refer to Martel and Merle~\cite{MM11, MM11-2} for rigorous results in the case of the gKdV equation.
\end{remark}

\subsection{A summary of the proof} 

In this section we give a brief outline of the paper, focusing on the proof of Theorem~\ref{thm:main}. 

Section~\ref{sec:coer} gives a detailed study of the kink solution $H(x)$ and the coercivity properties of the operator obtained by linearization. We establish several technical lemmas, including a computation of the formal attraction force between a well separated kink-antikink pair. This section is technical in nature and can be skimmed on a first reading.  

The argument used to prove Theorem~\ref{thm:main} is then divided in two parts. First, in Section~\ref{sec:mod} we give a preliminary dynamical classification of {\it all} finite energy strongly interacting kink-antikink pairs. Then, in Section~\ref{sec:nonlin-anal} we prove the existence of a kink-antikink pair while also establishing its uniqueness in a certain $t$-weighted function space. The dynamical classification result of Section~\ref{sec:mod} is then used to show that {\it every} strongly interacting kink-antikink pair lies in the function space in which uniqueness was established, thus giving uniqueness in the energy space and finishing the proof of Theorem~\ref{thm:main}.  The structure of this argument, which establishes {\it uniqueness} of the multi-kink in addition to its existence, is novel and should be of independent interest.  We give a rough sketch of how this works below. 

\textbf{Part 1:} To establish the preliminary classification we use a scheme similar to the one introduced by the first and third authors to  classify all two bubble wave maps in~\cite{JL1}, and by the first author to classify  strongly interacting two-solitons for gKdV in~\cite{jendrej2018dynamics}. We assume that $\phi(t, x)$ is a strongly interacting kink-antikink pair, and without loss of generality that $\phi_{\pm} = \pm1$. This means that for large enough times, $\phi$ admits a decomposition of the form 
\EQ{
\phi(t, x) = 1 - H( x - x_1(t)) + H( x- x_2(t)) + g(t, x) 
}
satisfying conditions \eqref{eq:norm-to-zero} and \eqref{eq:dist-to-infty-0}, or equivalently
\begin{align} \label{eq:g-dec}
\lim_{t \to \infty}\|(g(t), \partial_t \phi(t))\|_{H^1 \times L^2 (\R)} = 0 , \quad  
\lim_{t \to \infty} (x_2(t) - x_1(t)) = \infty.
\end{align}
The goal is to turn the qualitative assumptions above into quantitative information on the dynamics and decay of $(g(t), \p_t g(t), x_1(t), x_2(t))$. 

By standard modulation theoretic arguments, we fix the unique choice of $x_1(t)$ and $x_2(t)$ for which $g(t)$ satisfies the orthogonality conditions
\begin{equation} \label{eq:ortho} 
\la \partial_x H(\cdot - x_1(t)), g(t)\ra = 0,\qquad \la \partial_x H(\cdot - x_2(t)), g(t)\ra = 0.
\end{equation}
Differentiation of the orthogonality conditions,  use of the equation satisfied by $g(t, x)$, and an argument based on the Taylor expansion of the energy are enough to give preliminary estimates on the size of $|x_j'(t)|, |x_j''(t)|$ and $\|g(t), \p_t g(t) \|_{H^1 \times L^2}$  in terms of the distance $x_2(t) - x_1(t)$ between the kinks. However, as one might expect, these standard arguments are not sufficient to understand the dynamics in a useful way. At this point, we perform an ad hoc change of unknowns, replacing $x_j'(t)$ with corrected variables $p_j(t)$. The point is that while the $p_j(t)$ are small perturbations of $x_j'(t)$,  the correction, which is built using a localized momentum functional, cancels terms of indeterminate sign in the equations for $x_j''(t)$. We reveal that the dynamics of $p_j(t)$, and hence of $x_j'(t)$, are determined, up to negligible error, by the nonlinear interaction force $F( x_2(t) - x_1(t))$ between the two kinks; see Lemma~\ref{lem:norm-form}. A study of the ODE satisfied by the $p_j(t)$ yields bounds on the distance between the kinks, $ \simeq 2\log t$, as well as decay rates for $x_j'(t), x_j''(t)$.
We remark that  the technique of modifying a modulation parameter with a localized functional based on an underlying symmetry was used
in a similar context by the first author in~\cite{JJ-APDE}. 

At the conclusion of Section~\ref{sec:mod}, one could rather easily construct a strongly interacting kink-antikink pair. For example, see the construction performed in the recent work of the first and third authors with Rodriguez on singular wave maps in ~\cite[Section 5]{JLR1}, which used an analogous preliminary classification of the dynamics to pass to a weak limit of a sequence of well chosen approximations to the desired solution; see also previous work of Rodriguez~\cite{R19}. However, such constructions fail to establish {\it uniqueness}, which is a main goal of this work. To this end, we introduce a new version of Liapunov-Schmidt reduction in the setting of dispersive equations, inspired in part by work of the second author on the $2d$ elliptic Allen-Cahn problem in~\cite{MR2557944}.   That we can use this philosophy not just to construct but to prove unconditional uniqueness is novel, and relies crucially on the preliminary classification in Section~\ref{sec:mod}. 



\textbf{Part 2:} By Liapunov-Schmidt reduction, we simply mean that the process of finding the desired solution will be carried out in two steps described below. The implementation of these steps is of course quite different from the elliptic case, as we are here dealing with a nonlinear wave equation. 

We assume {\it a priori} that 
\EQ{\label{eq:ansatz} 
\phi(t) =1-H( \cdot - x_1(t))+H(\cdot - x_2(t))+g(t)
}
 and that \eqref{eq:g-dec} and \eqref{eq:ortho} hold.  We project the equation \eqref{eq:phi4} onto the space spanned by $\partial_x H(  \cdot - x_j(t))$, $j=1,2$ and onto its orthogonal complement. This way we are lead to solving the {\it projected equation}
\begin{equation}
\label{eq:proj}
\partial_{t}^2 \phi - \partial_{x}^2\phi + U'(\phi) = \lambda_1(t)\partial_x H( \cdot - x_1(t)) +\lambda_2(t)\partial_x H( \cdot - x_2(t)) 
\end{equation}
and what is referred to as the  {\it bifurcation equation}
\begin{equation}
\label{eq:bif}
\lambda_1(t)=0, \qquad \lambda_2(t)=0, 
\end{equation}
see for example~\cite[Section 2.4]{chow-hale}. Any $(g(t, x), x_1(t), x_2(t))$ that solves  both equations is the desired kink-antikink pair. 

{\it Step 1:} The first step is to solve~\eqref{eq:proj} by finding unique $(g(t,x), \lambda_j(t))$, for {\it given fixed} $x_j(t)$'s, within function spaces motivated by the classification result. The core ingredients in this step are energy-type  estimates for the linearized equation followed by a contraction mapping argument. Of course the linearized potential is time dependent (the kinks are moving), so a naive definition of the energy functional is not sufficient. We design a modified energy, namely a mixed energy/localized momentum functional, where a local momentum term is added to remove terms of critical size but indeterminate sign after differentiation. The addition of the localized momentum correction term is analogous to the mixed energy-localized virial functional used by the first author to study $2$-bubble energy critical waves in~\cite{JJ-AJM}, which drew its inspiration from Rapha\"el, Szeftel~\cite{RaSz11}. Here the underlying symmetry yielding the modulation parameters is translation and hence the correction is built from the generator of momentum, where in~\cite{JJ-AJM} the symmetry is scaling, which necessitates a localized virial correction. Our argument   requires $g(t, x)$ to exhibit  a quantitative improvement in time decay over what is given by the preliminary classification. One way of showing improved decay for the error $g$ would be to further refine the ansatz, i.e., extract the next order profiles from $g$ before imposing orthogonality conditions. Here we pursue an alternative method to obtain the improvement,  which consists of a further modification of the energy functional designed to exploit additional decay of the time derivative of the forcing; see Lemma~\ref{lem:phi4-lin-dt}. 

{\it Step 2:} The second step is to solve the bifurcation equation~\eqref{eq:bif}. In other words we seek  the unique pair of trajectories $(x_1(t), x_2(t))$ such that the corresponding triplet $(g(t), \lambda_1(t), \lambda_2(t))$ found in the first step satisfies $\lambda_1(t) = \lambda _2(t) \equiv 0$. We find that~\eqref{eq:bif} is a non local and nonlinear system of second order ODEs for $(x_1(t), x_2(t))$. To set up a contraction mapping, we must compare the solutions found in Step 1 arising from distinct pairs of trajectories $(x_1(t), x_2(t))$. This leads to a main difficulty in the method, as the chosen orthogonality conditions {\it depend on the choice of the trajectory}; see Lemma~\ref{lem:path-dep}.
Crucial to the entire argument of course, is the design of the function spaces in which the contraction mapping arguments are performed. 

By combining Parts 1 and 2 outlined above, the proof of Theorem 1 is completed at the end of Section~\ref{sec:nonlin-anal}. 

\subsection{Notation}
\label{ssec:notation}
Even if $v(x)$ is a function of one variable $x$, we often write $\partial_x v(x)$
instead of $v'(x)$ to denote the derivative. The prime notation is only used
for the time derivative of a function of one variable $t$
and for the derivative of the potential $U$.

We now define some  function spaces frequently used in the paper.
Let $\gamma, \beta, \alpha \in \bR$, $T_0 > 0$ and $z: [T_0, \infty) \to \bR$ a continuous function. We set
\begin{align}
\|z\|_{N_\gamma} &:= \sup_{t \geq T_0} t^\gamma |z(t)|, \\
\label{eq:Wab-def} \|z\|_{W_{\alpha, \beta}} &:= \sup_{\tau \geq t \geq T_0} t^{\beta - \alpha} \Big|\int_t^\tau s^\alpha z(s) \ud s\Big|.
\end{align}
If $z$ is twice continuously differentiable, we set
\begin{equation}
\|z\|_{S_\gamma} := \|z\|_{N_\gamma} + \|z'\|_{N_{\gamma+1}} + \|z''\|_{N_{\gamma + 1}}.
\end{equation}
Note that we are using the same time weight for $z'$ and $z''$.

If $z$ is a continuous function from  $[T_0, \infty)$ to some Banach space $E$, we denote
\begin{equation}
\|z\|_{N_\gamma(E)} := \big\| t \mapsto \|z(t)\|_E \big\|_{N_\gamma}.
\end{equation}
If $z$ is twice continuously differentiable function from  $[T_0, \infty)$ to $E$, we denote
\begin{equation}
\|z\|_{S_\gamma(E)} := \|z\|_{N_\gamma(E)} + \|z'\|_{N_{\gamma+1}(E)} + \|z''\|_{N_{\gamma + 1}(E)}.
\end{equation}
If the space $E$ is clear from the context, we write $N_\gamma$
instead of $N_\gamma(E)$ and $S_\gamma$ instead of $S_\gamma(E)$.
We define in the usual way the Banach spaces $N_\gamma(E)$ and $S_\gamma(E)$ as the completion of the space
of smooth compactly supported functions $[T_0, \infty) \to E$ for the corresponding norm.
Note that if $z \in N_\gamma(E)$, then $z$ is a continuous function from $[T_0, \infty)$ to $E$,
and if $z \in S_\gamma(E)$, then $z$ is a twice continuously differentiable function from $[T_0, \infty)$ to $E$.

\begin{remark}
We should keep in mind that all these norms depend on $T_0$. Often we can make some constants small by taking $T_0$ large enough.
For example, if $\gamma_1 < \gamma_2$ and $c_0 > 0$ is a small constant, then $\|\cdot\|_{N_{\gamma_1}} \leq c_0 \|\cdot \|_{N_{\gamma_2}}$
if $T_0$ is large enough (depending only on $\gamma_1$, $\gamma_2$ and $c_0$). We will use this fact frequently.
\end{remark}

We conclude this subsection with some additional notational conventions. 
\begin{itemize} 
\item If $\|\cdot\|_A$ and $\|\cdot\|_B$ are two norms, we denote $\|\cdot\|_{A \cap B} := \max(\|\cdot\|_A, \|\cdot\|_B)$.

\item We denote $\cE := H^1(\bR) \times L^2(\bR)$ (the energy space).

\item For $u, v: \bR \to \bR$ we write $\la u, v\ra := \int_{\bR} uv\ud x$, whenever this expression makes sense.

\item We denote $\vD$ and $\vD^2$ the first and second Fr\'echet derivatives of a functional.

\item We denote $x_+$ the positive part of $x$, in other words $x_+ = x$ if $x \geq 0$ and $x_+ = 0$ otherwise.

\item We take $\chi: \bR \to [0, 1]$ to be a decreasing $C^\infty$ function
such that $\chi(x) = 1$ for $x \leq \frac 13$
and $\chi(x) = 0$ for $x \geq \frac 23$.
\end{itemize} 


\section{Potential energy and interaction of a kink-antikink pair}
\label{sec:coer}
In this section, we analyse configurations close to a superposition of a well-separated kink and antikink at a fixed time.
We prove coercivity of the potential energy and prove bounds on various interaction terms, which will be used in later sections.

We note that by changing $U(\phi)$ to $\wt U(\phi) := \frac{1}{(\phi_+)^2U''(\phi_+)}U(\phi_+ \phi)$,
without loss of generality we can assume that $\phi_+=-\phi_-=1$ and $U''(-1) = U''(1) = 1$.
Indeed, $\phi(t, x)$ solves \eqref{eq:phi4} if and only if
\begin{equation}
\wt \phi(t, x) := \frac{1}{\phi_+}\phi\bigg(\frac{t}{\sqrt{U''(\phi_+)}}, \frac{x}{\sqrt{U''(\phi_+)}}\bigg)
\end{equation}
solves the same equation, but with the potential $\wt U(\phi)$ instead of $U(\phi)$.
For the kink $H$ of the original problem  this amounts to
\[
\tilde H(x)=\frac{1}{\phi_+} H\bigg(\frac{x}{\sqrt{U''(\phi_+)}}\bigg).
\]
Thus in the rest of this paper  we always assume that $\phi_+=1$, $\phi_-=-1$ and $U''(-1) = U''(1) = 1$.
\subsection{Stationary solutions}
\label{ssec:stationary}
A stationary field $\phi(t, x) = \psi(x)$ is a solution of \eqref{eq:phi4} if and only if
\begin{equation}
\label{eq:psi4}
\partial_x^2\psi(x) = U'(\psi(x)),\qquad\text{for all }x\in \bR.
\end{equation}
We seek solutions of \eqref{eq:psi4} having finite potential energy $E_p(\psi)$.
Since $U(\psi) \geq 0$ for $\psi \in \bR$, the condition $E_p(\psi) < \infty$ implies
\begin{align}
\label{eq:psi4-H1}
&\int_{-\infty}^{+\infty}\frac 12 (\partial_x \psi(x))^2 \ud x < \infty, \\
\label{eq:psi4-U}
&\int_{-\infty}^{+\infty}U(\psi(x)) \ud x < \infty.
\end{align}
From \eqref{eq:psi4-H1} we have $\psi \in C(\bR)$,
so \eqref{eq:psi4} and $U \in C^\infty(\bR)$ yield $\psi \in C^\infty(\bR)$.
Multiplying \eqref{eq:psi4} by $\partial_x \psi$ we get
\begin{equation}
\partial_x\Big(\frac 12 (\partial_x \psi)^2 - U(\psi)\Big)
= \partial_x \psi\big(\partial_x^2 \psi - U'(\psi)\big) = 0,
\end{equation}
so $\frac 12 (\partial_x \psi(x))^2 - U(\psi(x)) = k$ is a constant.
But then \eqref{eq:psi4-H1} and \eqref{eq:psi4-U} imply $k = 0$.
We obtain the first order Bogomolny equations:
\begin{equation}
\label{eq:bogom}
\partial_x\psi(x) = \sqrt{2U(\psi(x))}\quad\text{or}\quad \partial_x\psi(x) = -\sqrt{2U(\psi(x))}.
\end{equation}
We consider the first case, since the second is obtained by changing $x$ to $-x$.
If $\psi$ connects the two vacua $-1$ and $1$, then there exists $a \in \bR$
such that $\psi(a) = 0$. The solution of \eqref{eq:bogom}
with this initial condition is
$\psi(x) = H(x - a)$,
where $H(x)$ is defined by
\begin{equation}
\label{eq:H-def}
H(x) := G^{-1}(x),\quad\text{with}\ G(\psi) := \int_0^\psi\frac{\ud y}{\sqrt{2U(y)}}.
\end{equation}
\begin{proposition}
\label{prop:prop-H}
The function $H(x)$ defined by \eqref{eq:H-def} is of class $C^\infty(\bR)$
and there exist constants $\kappa > 0$ and $C > 0$ such that for all $x \in \bR$
\begin{gather}
\label{eq:H-asym-m}
\big|H(x)+1- \kappa \eee^{ x}\big| + \big|\partial_x H(x) -  \kappa \eee^{ x}\big| + \big|\partial_x^2 H(x) -  \kappa \eee^{ x}\big| \leq C\eee^{2 x}, \\
\label{eq:H-asym-p}
\big|H(x)-1+ \kappa \eee^{- x}\big| + \big|\partial_x H(x) -  \kappa \eee^{- x}\big| + \big|\partial_x^2 H(x) + \kappa \eee^{- x}\big| \leq C\eee^{-2 x}.
\end{gather}
\end{proposition}
\begin{proof}
We only prove \eqref{eq:H-asym-p}, which provides the asymptotic behavior of $H(x)$
for $x \to \infty$. The arguments for \eqref{eq:H-asym-m} are very similar.

Using the third order Taylor expansion of $U(y)$ around $y = 1$ one obtains $\frac{1}{\sqrt{2U(y)}} - \frac{1}{1-y} = O(1)$
as $y \to 1$, thus
\begin{equation}
G(\psi) = -\log(1-\psi) + \int_0^1 \bigg( \frac{1}{\sqrt{2U(y)}} - \frac{1}{1-y} \bigg)\ud y + O(1-\psi)
=-\log\Big(\frac{1-\psi}{\kappa}\Big) + O(1-\psi),
\end{equation}
where $\kappa$ is defined in \eqref{eq:k-def}.
Let $1 - H(x) = \kappa\eee^{-z}$. We obtain
\begin{equation}
z  - C\eee^{-z} \leq x \leq z + C\eee^{-z}.
\end{equation}
This implies in particular $|z - x| \lesssim 1$, and once we know this we get
\begin{equation}
|z - x| \lesssim \eee^{-z} \lesssim \eee^{- x},
\end{equation}
which implies
\begin{equation}
|1 - H(x) - \kappa\eee^{-x}| = \kappa \eee^{-x}|\eee^{x-z} - 1|
\lesssim \eee^{-x}\eee^{-x} = \eee^{-2x}.
\end{equation}
The bound for $\partial_x H(x)$ is obtained from \eqref{eq:bogom} and the fact that $\sqrt{2U(\psi)}
= (1-\psi) + O((1-\psi)^2)$. The bound for $\partial_x^2 H(x)$
is obtained from \eqref{eq:psi4} and the fact that $U'(\psi) = (1-\psi) + O((1-\psi)^2)$.
\end{proof}
We now compute two constants which appear in the proof. We claim that
\begin{align}
\|\partial_x H\|_{L^2}^2 = 2\int_0^1\sqrt{2U(y)}\ud y, \qquad
\label{eq:reduced-force}
\int_\bR \partial_x H(x)\left(U''(H(x)) - U''(1)\right)\eee^{x} \ud x = -2 \kappa.
\end{align}
The first formula follows from \eqref{eq:bogom} and a change of variable $y = H(x)$.
The second formula follows from
\begin{equation}
\begin{aligned}
\int_{-\infty}^R (\partial_x H) U''(H(x))\eee^{x} \ud x &= \int_{-\infty}^R \partial_x^3H(x)\eee^{x}\ud x \\
&= \eee^{R}(\partial_x^2 H(R) - \partial_x H(R)) + \int_{-\infty}^R \partial_x H(x) \eee^{x}\ud x,
\end{aligned}
\end{equation}
thus
\begin{equation}
\int_\bR \partial_x H(U''(H(x)) - U''(1))\eee^{x} \ud x = \lim_{R\to \infty} \eee^{R}(\partial_x^2 H(R) - \partial_x H(R)).
\end{equation}
\subsection{Coercivity}
\label{ssec:coer-kink}
Let $\phi: \bR \to \bR$ be a state such that $E_p(\phi) < \infty$, $\lim_{x \to -\infty}\phi(x) = -1$
and $\lim_{x \to +\infty}\phi(x) = 1$.
We have the classical Bogomolny coercivity:
\begin{equation}
\label{eq:bogom-coer}
\begin{aligned}
E_p(\phi) = \int_\bR\Big(\frac 12 (\partial_x \phi)^2 + U(\phi)\Big)\ud x
&= \int_\bR \sqrt{2U(\phi)}\partial_x \phi\ud x + \frac 12 \int_\bR (\partial_x \phi - \sqrt{2U(\phi)})^2\ud x \\
&= \int_{-1}^1 \sqrt{2U(y)}\ud y + \frac 12 \int_\bR (\partial_x \phi - \sqrt{2U(\phi)})^2\ud x.
\end{aligned}
\end{equation}
In particular,
\begin{equation}
\label{eq:V-of-H}
E_p(H) = \int_{-1}^1 \sqrt{2U(y)}\ud y,\qquad E_p(\phi) \geq E_p(H).
\end{equation}

We define
\begin{equation}
L := \vD^2 E_p(H) = -\partial_x^2 + U''(H) = -\partial_x^2 + 1 + (U''(H) - 1).
\end{equation}
Observe that $U''(H) - 1$ is an exponentially decaying $C^\infty$ function.
Differentiating $\partial_x^2 H(x - a) = U'(H(x-a))$ with respect to $a$ we obtain
\begin{equation}
\label{eq:Lc-ker}
\big({-}\partial_x^2 + U''(H(\cdot - a))\big)\partial_x H(\cdot - a) = 0,
\end{equation}
in particular for $a = 0$ we have $L(\partial_x H) = 0$.
Differentiating \eqref{eq:Lc-ker} with respect to $a$ at $a = 0$ we obtain
\begin{equation}
\label{eq:U'''-identity}
L(\partial_x^2 H) = -U'''(H)(\partial_x H)^2.
\end{equation}
\begin{proposition}
\label{prop:L-index}
The operator $L$ is self-adjoint with domain $H^2(\bR)$,
$\tx{spec}(L) \subset \{0\} \cup [\lambda, +\infty)$ for some $\lambda > 0$ and
$\ker L = \spn(\partial_x H)$.
\end{proposition}
\begin{proof}
This is a standard consequence of the Sturm-Liouville theory and the fact
that $\partial_x H(x) > 0$ for all $x \in \bR$.
\end{proof}

\begin{lemma}
There exists $c > 0$ such that for all $v \in H^1(\bR)$ the following inequality holds:
\begin{equation}
\label{eq:vLv-coer}
\la v, Lv\ra \geq c \|v\|_{H^1}^2 -\lambda \|\partial_x H\|^{-2}_{L^2} \la \partial_x H, v\ra^2.
\end{equation}
\end{lemma}
\begin{proof}
By the definition of $L$ we have
\begin{equation}
\label{eq:vLv-1}
\la v, Lv\ra = \|v\|_{H^1}^2 - \int_\bR (1-U''(H))v^2\ud x,
\end{equation}
and by Proposition~\ref{prop:L-index} we have
\[
\la v, Lv\ra\geq \lambda \|v\|^2_{L^2}-\lambda \|\partial_x H\|^{-2}_{L^2}\la \partial_x H, v\ra^2.
\]
This implies
\begin{equation}
\label{eq:vLv-2}
\la v, Lv\ra -c \|v\|^2_{H^1}\geq (1-c)\lambda \|v\|_{L^2}^2-c\int_\bR (1 - U''(H))v^2\ud x - (1-c)\lambda \|\partial_x H\|^{-2}_{L^2} \la \partial_x H, v\ra^2.
\end{equation}
Since $1 - U''(H)$ is a bounded function, \eqref{eq:vLv-coer}  follows by taking $c$ small.
\end{proof}

%
Given $X = (x_1, x_2)$ with $x_1 < x_2$, we denote
\begin{equation}
H_j(x) := H(x - x_j), \qquad L_X := \vD^2 E_p(1 - H_1 + H_2) = -\partial_x^2 + U''(1 - H_1 + H_2).
\end{equation}
\begin{lemma}
\label{lem:D2H}
There exist $\lambda_0, z_0 > 0$ such that for all
$X=(x_1, x_2)$ with $x_2 - x_1 \geq z_0$ and $v \in H^1(\bR)$
\begin{equation}
\label{eq:D2H-coer}
\begin{aligned}
\la v, L_X v\ra \geq \lambda_0 \|v\|_{H^1}^2 -{ \lambda \|\partial_x H\|^{-2}_{L^2}} \big(\la\partial_x H_1, v\ra^2 + \la\partial_x H_2, v\ra^2\big).
\end{aligned}
\end{equation}
\end{lemma}
\begin{proof}

We set
\begin{align}
\chi_1(x) &:= \chi\Big(\frac{x-x_1}{x_2 - x_1}\Big), \\
\chi_2(x) &:= 1 - \chi_1(x), \\
v_1 &:= \chi_1 v, \\
v_2 &:= \chi_2 v.
\end{align}
We have
\begin{equation}
\la v, L_X v\ra = \la v_1, L_X v_1\ra + \la v_2, L_X v_2\ra + 2\la v_1, L_X v_2\ra,
\end{equation}
so it suffices to prove that
\begin{gather}
\label{eq:coer-1}
\la v_j, L_X v_j\ra \geq {c}\|v_j\|_{H^1}^2 - { \lambda \|\partial_x H\|^{-2}_{L^2}}\la\partial_x H_j, v_j\ra^2
-{o(1)\|v\|_{H^1}^2}, \\
\label{eq:coer-2}
\la v_1, L_X v_2\ra {\geq o(1)\|v\|_{H^1}^2}, \\
\label{eq:coer-3}
\big|\la\partial_x H_j, v_j\ra^2 - \la \partial_x H_j, v\ra^2\big| {\leq o(1)\|v\|_{H^1}^2},
\end{gather}
{where $c>0$ is the constant in \eqref{eq:vLv-coer} and $o(1)\to 0$ as $z_0\to \infty$.}

%
We  prove \eqref{eq:coer-1} for $j = 1$ (the proof for $j = 2$ is similar).
Without loss of generality we can assume $x_1 = 0$, so that $x_2 \geq z_0$. We then have
\begin{equation}
L_X = L + V, \qquad V := U''(1 - H + H_2) - U''(H).
\end{equation}
We thus obtain
\begin{equation}
\la v_1, L_X v_1\ra = \la v_1, L v_1\ra + \la v_1, Vv_1\ra \geq c \|v_1\|_{H^1}^2 - \lambda \|\partial_x H\|^{-2}_{L^2}\la \partial_x H, v_1\ra^2
+ \int_{\bR} \chi_1^2 V v^2\ud x.
\end{equation}
Note that $\|\chi_1^2 V\|_{L^\infty} \ll 1$. Indeed, if $x \geq \frac 23 x_2$ then $\chi_1(x) = 0$.
If $x \leq \frac 23 x_2$, then $|1 + H_2| \ll 1$ which implies $|V(x)| \ll 1$. This proves \eqref{eq:coer-1}.

Next, we show \eqref{eq:coer-2}. Using the fact that $\|\partial_x \chi_j\|_{L^\infty} \ll 1$ we obtain
\[
\begin{aligned}
\la v_1, L_X v_2\ra &= \int_{\bR} \partial_x(\chi_1 v)\partial_x(\chi_2 v)\ud x +V \chi_1\chi_2 v^2\\
&= \int_{\bR}\chi_1\chi_2 (\partial_x v)^2\ud x + o(1)\|v\|_{H^1}^2\geq o(1)\|v\|_{H^1}^2,
\end{aligned}
\]
as the first term in the second line is positive.

Finally, the bound \eqref{eq:coer-3} follows from
\begin{equation}
\begin{aligned}
\big|\la\partial_x H_1, v_1\ra^2 - \la \partial_x H_1, v\ra^2\big| &\lesssim (\|v_1\|_{L^2} + \|v\|_{L^2})
\big|\la\partial_x H_1, v_1 - v\ra\big| \\ &\lesssim \|v\|_{L^2}^2\|\chi_2\partial_x H_1\|_{L^2} \leq o(1)\|v\|_{L^2}^2,
\end{aligned}
\end{equation}
and similarly for $j = 2$.

\end{proof}

\subsection{Interaction of the kinks}
\label{ssec:interactions}
Note that $|1 - H_1| \lesssim \eee^{-(x - x_1)_+}$ and $|1 + H_2| \lesssim \eee^{-(x_2 - x)_+}$.
The following lemma is often useful while estimating interactions.
\begin{lemma}
\label{lem:exp-cross-term}
For any $x_1 < x_2$ and $\alpha, \beta > 0$ with $\alpha \neq \beta$ the following bound holds:
\begin{equation}
\int_{\bR}\eee^{-\alpha(x - x_1)_+}\eee^{-\beta(x_2 - x)_+}\ud x \lesssim_{\alpha, \beta} \eee^{-\min(\alpha, \beta)(x_2 - x_1)}, \qquad \forall x_1, x_2 \in \bR.
\end{equation}
For any $\alpha > 0$, the following bound holds:
\begin{equation}
\int_{\bR}\eee^{-\alpha(x - x_1)_+}\eee^{-\alpha(x_2 - x)_+}\ud x \lesssim_{\alpha} (x_2 - x_1)\eee^{-\alpha(x_2 - x_1)}, \qquad \forall x_1, x_2 \in \bR.
\end{equation}
\end{lemma}
\begin{proof}
Straightforward computation.
\end{proof}
To measure the interaction between the kinks located at  $x_1, x_2 \in \bR$, we introduce a function $\Phi\in C^\infty(\bR^3)$ by 
\begin{equation} \label{eq:Phi-def} 
 \Phi(x_1, x_2, x) := -U'(1 - H_1(x) + H_2(x)) - U'(H_1(x)) + U'(H_2(x)).
\end{equation}
Observe that
\begin{equation}
\label{eq:DEp-ansatz}
\vD E_p(1 - H_1 + H_2) = -\partial_x^2(1 - H_1 + H_2) + U'(1 - H_1 + H_2) = -\Phi(x_1, x_2, \cdot).
\end{equation}
\begin{lemma}
\label{lem:inter-bound}
There exists $C > 0$ (depending only on $U$) such that for all $x_1, x_2, x \in \bR$ with $x_2 - x_1 \geq 1$
the following inequalities are true for all $j \in \{1, 2\}$:
\begin{align}
\label{eq:inter-bound-1} |\Phi(x_1, x_2, x)| + |\partial_{x_j}\Phi(x_1, x_2, x)| 
&\leq C\eee^{-(x - x_1)_+}\eee^{-(x_2 - x)_+}, \\
\label{eq:inter-bound-2} \big|\big(U''(1-H_1(x) + H_2(x)) - U''(H_j(x))\big)\partial_x H_j(x)\big| &\leq C\eee^{-(x - x_1)_+}\eee^{-(x_2 - x)_+}, \\
\label{eq:inter-bound-3} \big|\partial_x H_1(x)\partial_x H_2(x)| &\leq C\eee^{-(x - x_1)_+}\eee^{-(x_2 - x)_+}.
\end{align}
\end{lemma}
\begin{proof}
For $w_1, w_2 \in \bR$ set
\begin{equation}
f(w_1, w_2) := -U'(1-w_1 + w_2) - U'(w_1) + U'(w_2).
\end{equation}
Then $\Phi(x_1, x_2, x) = f(H_1(x), H_2(x))$ so, by the Chain Rule and the fact that $\partial_{x_j}H_j = -\partial_x H_j$,
\begin{align}
\label{eq:deriv-xj-Phi}
\partial_{x_j}\Phi(x_1, x_2, x) &= -\partial_x H_j(x)\partial_{w_j}f(H_1(x), H_2(x)), \\
\partial_{x_j}^2\Phi(x_1, x_2, x) &= \partial_x^2 H_j(x)\partial_{w_j}f(H_1(x), H_2(x)) + (\partial_x H_j(x))^2\partial_{w_j}^2 f(H_1(x), H_2(x)), \\
\partial_{x_1}\partial_{x_2}\Phi(x_1, x_2, x) &= \partial_x H_1(x)\partial_x H_2(x)\partial_{w_1}\partial_{w_2}f(H_1(x), H_2(x)).
\end{align}
Since $|1 - H_1| + |\partial_x H_1| + |\partial_x^2 H_1| \lesssim \eee^{-(x - x_1)_+}$ and $|1 + H_2| + |\partial_x H_2| + |\partial_x^2 H_2| \lesssim \eee^{-(x_2 - x)_+}$,
in order to prove \eqref{eq:inter-bound-1} it suffices to check that for $-1 \leq w_1, w_2 \leq 1$
\begin{align}
\label{eq:fw1w2-1} |f(w_1, w_2)| &\lesssim |(1-w_1)(1+w_2)|, \\
\label{eq:fw1w2-2} |\partial_{w_1}f(w_1, w_2)| &\lesssim |1 + w_2|, \\
\label{eq:fw1w2-3} |\partial_{w_2}f(w_1, w_2)| &\lesssim |1 - w_1|, \\
\label{eq:fw1w2-5} |\partial_{w_2}^2f(w_1, w_2)| &\lesssim |1 - w_1|.
\end{align}
We have $\partial_{w_1}f(w_1, w_2) = U''(1 - w_1 + w_2) - U''(w_1) = U''(1 - w_1 + w_2) - U''(w_1)$.
Since $U''$ is locally Lipschitz, \eqref{eq:fw1w2-2} follows. The bound \eqref{eq:fw1w2-3} is similar,
and \eqref{eq:fw1w2-5} are proved similarly, using that fact that $U'''$ is locally Lipschitz.
Finally, in order to prove \eqref{eq:fw1w2-1} we notice that
\begin{equation}
f(w_1, w_2) = f(1, w_2) - \int_{w_1}^1 \partial_{w_1}f(w, w_2)\ud w = - \int_{w_1}^1 \partial_{w_1}f(w, w_2)\ud w
\end{equation}
and we conclude using \eqref{eq:fw1w2-2}.

The left hand side in \eqref{eq:inter-bound-2} is $|\partial_{x_j}\Phi(x_1, x_2, x)|$,
so the estimate follows from \eqref{eq:inter-bound-1}. The bound \eqref{eq:inter-bound-3} is clear.
\end{proof}

We will often denote $z = x_2 - x_1$ the distance between the kinks. We introduce the following function,
which is the (renormalised) formally computed attraction force between a kink and an~antikink at distance $z$:
\begin{equation}
\label{eq:force-def}
F(z) := \|\partial_x H\|_{L^2}^{-2}\la \partial_x H, \Phi(0, z, \cdot)\ra.
\end{equation}
For future reference we note that 
\[
F(x_2-x_1)= \|\partial_x H\|_{L^2}^{-2}\la \partial_x H_1(\cdot), \Phi(x_1, x_2, \cdot)\ra=\|\partial_x H\|_{L^2}^{-2}\la \partial_x H_2(\cdot), \Phi(x_1, x_2, \cdot)\ra.
\]

\begin{lemma}
\label{lem:Fz}
There exists $C > 0$ such that for all $z \geq 1$ the function $F(z)$ satisfies
\begin{align}
\label{eq:Fz}
\big|F(z) - 2\|\partial_x H\|_{L^2}^{-2}\kappa^2 \eee^{-z}\big| \leq C z \eee^{-2z}, \\
\label{eq:dFz}
\big|F'(z) + 2\|\partial_x H\|_{L^2}^{-2}\kappa^2 \eee^{-z}\big| \leq Cz \eee^{-2z},
\end{align}
where $\kappa$ is defined by \eqref{eq:k-def}.
\end{lemma}
\begin{proof}
Lemma~\ref{lem:inter-bound} implies $\lim_{z \to \infty} F(z) = 0$,
hence \eqref{eq:Fz} follows by integrating \eqref{eq:dFz}.

We now prove \eqref{eq:dFz}. Denote $\wt H(x) := H(x - z)$. Using the notations from Lemma~\ref{lem:inter-bound} and \eqref{eq:deriv-xj-Phi}, we have
\begin{equation}
F'(z) = \|\partial_x H\|_{L^2}^{-2}\big\la \partial_x H, \partial_{x_2}\Phi(0, z, \cdot)\big\ra = -\|\partial_x H\|_{L^2}^{-2}\int_{\bR} \partial_x H(x)\,\partial_x \wt H(x)\,\partial_{w_2}f(H(x), \wt H(x))\ud x.
\end{equation}
In the computation which follows the symbol ``$\simeq$'' means ``up to terms of order $z \eee^{-2z}$''.

The fundamental theorem of calculus together with \eqref{eq:fw1w2-5} yields
\begin{equation}
\big|{-}\partial_{w_2}f(H, \wt H) - (U''(H) - U''(1))\big| = \big|\partial_{w_2}f(H, \wt H) - \partial_{w_2}f(H, -1)\big| \lesssim |(1 + \wt H)(1 - H)|.
\end{equation}
Using $|\partial_x H| + |1 - H| \lesssim \eee^{-x_+}$ and $|\partial_x \wt H| + |1 + \wt H| \lesssim \eee^{-(z - x)_+}$,
we obtain
\begin{equation}
\bigg|F'(z) - \|\partial_x H\|_{L^2}^{-2}\int_\bR (\partial_x H)(\partial_x \wt H)(U''(H) - U''(1))\ud x\bigg| \lesssim
\eee^{-2x_+}\eee^{-2(z-x)_+}.
\end{equation}
Lemma~\ref{lem:exp-cross-term} yields
\begin{equation}
\label{eq:deriv-of-F}
F'(z) \simeq \|\partial_x H\|_{L^2}^{-2}\int_{\bR} (\partial_x H)(\partial_x \wt H)(U''(H) - U''(1))\ud x.
\end{equation}
The function $U''$ is locally Lipschitz, thus $|U''(H) - U''(1)| \lesssim |1-H|$.
We also have, by Proposition~\ref{prop:prop-H},
\begin{equation}
\big|\partial_x \wt H(x) - \kappa\eee^{x - z}\big| \lesssim \min(\eee^{2(x - z)}, \eee^{x-z}).
\end{equation}
Since
\begin{equation}
\int_\bR |(\partial_x H)(1-H)|\min(\eee^{2(x - z)}, \eee^{x-z})\ud x \lesssim \int_{-\infty}^z \eee^{-2x_+}\eee^{-2(z-x)_+}\ud x
+ \int_z^\infty \eee^{-2x}\eee^{x - z}\ud x \lesssim z\eee^{-2z},
\end{equation}
we conclude that
\begin{equation}
F'(z) \simeq \kappa\|\partial_x H\|_{L^2}^{-2}\eee^{-z}\int_\bR (\partial_x H)(U''(H) - U''(1))\eee^x \ud x \simeq -2\kappa^2 \|\partial_x H\|^{-2}\eee^{-z},
\end{equation}
where in the last step we use \eqref{eq:reduced-force}.
\end{proof}
\begin{remark}
Let $x_2 - x_1 \gg 1$, $H_1 := H(x - x_1)$ and $H_2 := H(x - x_2)$. Then, by translation invariance and symmetry,
it follows from the last lemma that
\begin{equation}
\la \partial_x H_1, \Phi(x_1, x_2, \cdot)\ra = \la \partial_x H_2, \Phi(x_1, x_2, \cdot)\ra = 2\kappa^2 \eee^{-(x_2 - x_1)} + O\big((x_2 - x_1)\eee^{-2(x_2 - x_1)}\big).
\end{equation}
\end{remark}

In the next lemma, we compute the potential energy
of a kink-antikink configuration $\phi(x) = 1 - H(x - x_1) + H(x-x_2)$.
\begin{lemma}
\label{lem:V-two-kink}
There exists $C > 0$ such that if $x_2 - x_1 \gg 1$ and $H_j(x) := H(x - x_j)$, then
\begin{equation}
\label{eq:V-two-kink}
\big|E_p(1 - H_1 + H_2) - \big(2E_p(H) - 2\kappa^2 \eee^{-(x_2 - x_1)}\big)\big| \leq C(x_2 - x_1)\eee^{-2(x_2 - x_1)},
\end{equation}
where $\kappa$ is defined by \eqref{eq:k-def}.
\end{lemma}
\begin{proof}
Without loss of generality assume $x_1 = 0$ and $x_2 = z \gg 1$. Observe that
\begin{gather}
\dd z\int_\bR \partial_x H\,\partial_x H(\cdot - z)\ud x = -\int_\bR \partial_x H \partial_x^2 H(\cdot - z) = \int_\bR U'(H)\partial_x H(\cdot- z)\ud x, \\
\dd z \int_\bR U(1 - H + H(\cdot - z))\ud x = -\int_\bR U'(1 - H + H(\cdot - z))\partial_x H(\cdot - z)\ud x,
\end{gather}
thus
\begin{equation}
\dd z E_p(1 - H + H(\cdot - z)) = -\int_\bR\big(U'(1 - H + H(\cdot - z))-U'(H)\big)\partial_x H(\cdot - z)\ud x.
\end{equation}
By symmetry, we obtain
\begin{equation}
\begin{aligned}
\dd z E_p(1 - H + H(\cdot - z)) &= \int_\bR\big(U'(1 - H + H(\cdot - z))-U'(H(\cdot - z))\big)\partial_x H\ud x \\
&= \int_\bR\big(U'(1 - H + H(\cdot - z)) + U'(H)-U'(H(\cdot - z))\big)\partial_x H\ud x\\ &= -\|\partial_x H\|_{L^2}^2 F(z),
\end{aligned}
\end{equation}
where $F(z)$ is defined by \eqref{eq:force-def}.
It remains to check that
\begin{equation}
\label{eq:Ep-limit}
\lim_{z\to\infty} E_p(1 - H + H(\cdot - z)) = 2E_p(H),
\end{equation}
and \eqref{eq:V-two-kink} will follow by integrating \eqref{eq:Fz} in $z$.

Let $-1 \leq w_1, w_2 \leq 1$. Integrating the inequality $|U'(1-w + w_2) + U'(w)| \lesssim |1 + w_2|$
for $w_1 \leq w \leq 1$, we get
$|U(1 - w_1 + w_2) - U(w_1) - U(w_2)| \lesssim |(1-w_1)(1+w_2)|$, thus
\begin{equation}
\label{eq:Ep-limit-1}
\big|U(1 - H(x) + H(x - z)) - U(H(x)) - U(H(x - z))\big| \lesssim \eee^{-x_+}\eee^{-(z-x)_+}.
\end{equation}
We also have
\begin{equation}
\label{eq:Ep-limit-2}
\begin{aligned}
\big|\big(\partial_x(1 - H(x) + H(x - z))\big)^2 - (\partial_x H(x))^2 - (\partial_x H(x - z))^2\big| \\
 \lesssim |\partial_x H(x)\partial_x H(z - x)| \lesssim \eee^{-x_+}\eee^{-(z-x)_+},
\end{aligned}
\end{equation}
and \eqref{eq:Ep-limit} follows by integrating \eqref{eq:Ep-limit-1} and \eqref{eq:Ep-limit-2} in $x$.
\end{proof}
\subsection{Taylor expansions}
To finish this section, we prove a few estimates based on the Taylor expansion of the nonlinearity.
\begin{lemma} There exists $C > 0$ such that for all $-1 \leq w, \sh w \leq 1$ and $|g| + |\sh g| \ll 1$
the following bounds hold:
\begin{gather}
\label{eq:U-taylor}
\big|U(w + g) - U(w) - U'(w)g\big| \leq Cg^2, \\
\label{eq:U-taylor-2}
\big|U(w + g) - U(w) - U'(w)g - U''(w)g^2/2\big| \leq Cg^3, \\
\label{eq:U'-taylor}
\big|U'(w + g) - U'(w) - U''(w)g\big| \leq Cg^2, \\
\label{eq:U''-taylor}
\big|U''(w + g) - U''(w) - U'''(w)g\big| \leq Cg^2,
\end{gather}
\begin{equation}
\label{eq:U'-taylor-diff}
\begin{aligned}
\big|\big(U'(\sh w + \sh g) - U'(\sh w) - U''(\sh w)\sh g\big) -
\big(U'(w + g) - U'(w) - U''(w)g\big)\big| \\
\leq C\big(|\sh g| + |g|\big)\big(\big|\sh g - g\big| + |\sh w - w|\big(|\sh g| + |g|)\big).
\end{aligned}
\end{equation}
\end{lemma}
\begin{proof}
Bounds \eqref{eq:U-taylor}, \eqref{eq:U-taylor-2}, \eqref{eq:U'-taylor} and \eqref{eq:U''-taylor} are clear,
so we are left with \eqref{eq:U'-taylor-diff}.
The Taylor formula yields
\begin{align}
U'(w + g) - U'(w) - U''(w)g &= \int_0^g s U'''(w+g-s)\ud s, \\
U'(\sh w + \sh g) - U'(\sh w) - U''(\sh w)\sh g &= \int_0^{\sh g} s U'''(\sh w+\sh g-s)\ud s.
\end{align}
We observe that, since $U'''$ is locally Lipschitz,
\begin{equation}
\begin{aligned}
\bigg|\int_0^{\sh g} s U'''(\sh w+\sh g-s)\ud s - \int_0^{\sh g} s U'''(w+g-s)\ud s\bigg| &\lesssim \bigg|\int_0^{\sh g} s|\sh w + \sh g - w - g|\ud s\bigg| \\
&\lesssim (\sh g)^2(|\sh w - w| + |\sh g - g|).
\end{aligned}
\end{equation}
We also have
\begin{equation}
\bigg|\int_0^{\sh g} s U'''(w+g-s)\ud s - \int_0^{g} s U'''(w+g-s)\ud s\bigg| = \bigg|\int_g^{\sh g}s U'''(w + g - s)\ud s\bigg|
\lesssim \bigg|\int_g^{\sh g}|s|\ud s\bigg|.
\end{equation}
If $g$ and $\sh g$ have the same sign, then the last integral equals $\frac 12 |(\sh g)^2 - g^2| = \frac 12 (|\sh g| + |g|)|\sh g - g|$.
If $g$ and $\sh g$ have opposite signs, then we obtain $\frac 12 ((\sh g)^2 + g^2) \leq \frac 12 (|\sh g| + |g|)^2 = \frac 12 (|\sh g| + |g|)|\sh g - g|$. This proves \eqref{eq:U'-taylor-diff}.
\end{proof}

\section{Main order asymptotics of any kink-antikink pair}
\label{sec:mod}

We consider any solution $\phi$ of \eqref{eq:phi4} of the form
\begin{equation}
\label{eq:lyap-schmidt-phi}
\phi(t, x) = 1-H(x - x_1(t)) + H(x - x_2(t)) + g(t, x),
\end{equation}
satisfying conditions \eqref{eq:norm-to-zero} and \eqref{eq:dist-to-infty-0}, equivalently
\begin{align}
\lim_{t \to \infty}\|(g(t), \partial_t \phi(t))\|_{\cE} = 0, \label{eq:g-to-0} \\
\lim_{t \to \infty} (x_2(t) - x_1(t)) = \infty. \label{eq:dist-to-infty}
\end{align}
The first step is to specify the choice of $(x_1, x_2)$.
\begin{lemma}
Under the conditions \eqref{eq:g-to-0} and \eqref{eq:dist-to-infty}, there exists a pair of
twice continuously differentiable functions
$(\wt x_1, \wt x_2): [T_0, \infty) \to \bR $ such that the following holds.
Define $\wt g: [T_0, \infty) \to H^1(\bR)$ by the relation $\phi(t, x) = 1-H(x - \wt x_1(t)) + H(x - \wt x_2(t)) + \wt g(t, x)$.
Then \eqref{eq:g-to-0} and \eqref{eq:dist-to-infty} hold with $(x_1, x_2, g)$ replaced by $(\wt x_1, \wt x_2, \wt g)$ and,
moreover, $\wt g(t)$ satisfies the orthogonality conditions
\begin{equation}
\label{eq:g-tilde-orth}
\la \partial_x H(\cdot - \wt x_1(t)), \wt g(t)\ra = 0,\qquad \la \partial_x H(\cdot - \wt x_2(t)), \wt g(t)\ra = 0.
\end{equation}
\end{lemma}
\begin{proof}
The proof follows a well-known scheme based on a quantitative version of the Implicit Function Theorem,
see for instance \cite[Lemma~3.3]{JJ-Pisa}.

\textbf{Step 1.} (Choice of parameters for a fixed time.)
Fix fix $t$ and write $(x_1, x_2, g, \phi)$ instead of $(x_1(t), x_2(t), g(t), \phi(t))$.
We prove that there exists $C_0, \eta_0 > 0$ having the following property.
If $\phi(x) = 1 - H(x - x_1) + H(x - x_2) + g(x)$ with $x_2 - x_1 \geq \eta^{-1}$, $\|g\|_{H^1} \leq \eta$ and $\eta \leq \eta_0$,
then there exists a unique pair $(\wt x_1, \wt x_2)$ such that $\wt g(x) := \phi(x) - (1 - H(x - \wt x_1) + H(x - \wt x_2))$
satisfies $\wt x_2 - \wt x_1 \geq (C\eta)^{-1}$, $\|\wt g\|_{H^1} \leq C\eta$ and the orthogonality conditions
\begin{equation}
\la \partial_x H(\cdot - \wt x_1), \wt g\ra = 0,\qquad \la \partial_x H(\cdot - \wt x_2), \wt g\ra = 0.
\end{equation}

We define $\Gamma: \bR^2 \times H^1(\bR) \to \bR^2$ by
\begin{equation}
\begin{aligned}
\Gamma(x_1, x_2, \phi) := (&\la \partial_x H(\cdot - x_1), \phi - (1 - H(\cdot - x_1) + H(\cdot - x_2))\ra,\\
& \la \partial_x H(\cdot - x_2), \phi - (1 - H(\cdot - x_1) + H(\cdot - x_2))\ra).
\end{aligned}
\end{equation}
It is easy to check that $\vD_{x_1, x_2}\Gamma$ is a uniformly non-degenerate matrix, which implies the claim.

\textbf{Step 2.} (Time differentiability of the modulation parameters.)
Thus, if $T_0$ is large enough, by Step 1. there exist $\wt x_1(T_0)$ and $\wt x_2(T_0)$ such that \eqref{eq:g-tilde-orth}
holds for $t = T_0$. We now define $(\wt x_1, \wt x_2)$ as the solution of the system of differential equations
(with initial conditions at $t = T_0$)
\begin{equation}
\begin{aligned}
\wt x_1'(t)\big({-}\|\partial_x H\|_{L^2}^2 - \la \partial_x^2 \wt H_1(t), \wt g(t)\ra\big) + \wt x_2'(t)\la \partial_x \wt H_1(t), \partial_x \wt H_2(t)\ra + \la \partial_x \wt H_1(t), \partial_t \phi(t)\ra &= 0, \\
\wt x_2'(t)\big(\|\partial_x H\|_{L^2}^2 - \la \partial_x^2 \wt H_2(t), \wt g(t)\ra\big) - \wt x_1'(t)\la \partial_x \wt H_1(t), \partial_x \wt H_2(t)\ra
+ \la \partial_x \wt H_2(t), \partial_t \phi(t)\ra &= 0,
\end{aligned}
\end{equation}
where we abbreviated $\wt H_j(t, x) := H(x - \wt x_j(t))$ and $\wt g(t) := \phi(t) - (1 - \wt H_1(t) + \wt H_2(t))$.
The computation at the beginning of Lemma~\ref{lem:en-bound} below shows that \eqref{eq:g-tilde-orth} then holds for all $t \geq T_0$.
By a straightforward bootstrap argument and using the uniqueness part of Step 1, $(\wt x_1, \wt x_2)$ satisfy $\wt x_2(t) - \wt x_1(t) \to \infty$ and $\|\wt g(t)\|_{H^1} \to 0$.
Also, we deduce from the differential equations that $\wt x_1$ and $\wt x_2$ are twice continuously differentiable.
\end{proof}
In the sequel, we write $(x_1, x_2, g)$ instead of $(\wt x_1, \wt x_2, \wt g)$.
In other words, we have \eqref{eq:lyap-schmidt-phi}, \eqref{eq:g-to-0}, \eqref{eq:dist-to-infty} and,
additionally, $x_1, x_2$ are twice continuously differentiable and satisfy
\begin{equation}
\label{eq:g-orth}
\la \partial_x H(\cdot - x_1(t)), g(t)\ra = 0,\qquad \la \partial_x H(\cdot - x_2(t)), g(t)\ra = 0.
\end{equation}

\begin{remark}
Passing from $\phi$ to the triple $(x_1, x_2, g)$ defines a diffeomorphism of a neighborhood of the kink-antikink pairs
and a codimension two submanifold of $\bR^2 \times H^1(\bR)$ determined by the conditions \eqref{eq:g-orth}.
Note that $(x_1, x_2, g)$ is not a system of coordinates;
informally speaking, $g(t,x)$ and $(x_1(t), x_2(t))$ are ``not independent''.
This causes some trouble when one wants to compare two solutions corresponding to two different pairs of trajectories
$X = (x_1, x_2)$ and $\sh X = (\sh x_1, \sh x_2)$, see Lemma~\ref{lem:path-dep} below.
\end{remark}

%
Writing $H_j(t, x) := H(x - x_j(t))$, we have
\begin{align}
\label{eq:lyap-schmidt-dtphi}
\partial_t \phi(t) &= x_1'(t)\partial_x H_1(t) - x_2'(t)\partial_x H_2(t) + \partial_t g(t), \\
\partial_t^2 \phi(t) &= x_1''(t)\partial_x H_1(t) -(x_1'(t))^2\partial_x^2 H_1(t) - x_2''(t)\partial_x H_2(t) + (x_2'(t))^2 \partial_x^2 H_2(t) + \partial_t^2 g(t), \\
\label{eq:lyap-schmidt-dxphi}
\partial_x \phi(t) &= {-}\partial_x H_1(t) + \partial_x H_2(t) + \partial_x g(t), \\
\partial_x^2 \phi(t) &= -\partial_x^2 H_1(t) + \partial_x^2 H_2(t) + \partial_x^2 g(t),
\end{align}
thus \eqref{eq:phi4} rewrites as
\begin{equation}
\begin{aligned}
\partial_t^2 g &+ x_1''\partial_x H_1 - (x_1')^2\partial_x^2 H_1 - x_2''\partial_x H_2 + (x_2')^2 \partial_x^2 H_2 \\
&+ \partial_x^2 H_1 - \partial_x^2 H_2 - \partial_x^2 g + U'(1 - H_1 + H_2 + g) = 0,
\end{aligned}
\end{equation}
or, using $\partial_x^2 H_j = U'(H_j)$,
\begin{equation}
\label{eq:dt2g}
\begin{aligned}
\partial_t^2 g &+ x_1''\partial_x H_1 - (x_1')^2\partial_x^2 H_1 - x_2''\partial_x H_2 + (x_2')^2 \partial_x^2 H_2 \\
 &- \partial_x^2 g + U'(1 - H_1 + H_2 + g) + U'(H_1) - U'(H_2) = 0.
\end{aligned}
\end{equation}

\begin{lemma}
\label{lem:en-bound}
If $\phi$ is a solution of \eqref{eq:phi4} such that \eqref{eq:g-to-0} and \eqref{eq:dist-to-infty} hold, then
there exist $C_0$ and $T_0$ such that for all $t \geq T_0$ the following bounds hold:
\begin{gather}
\label{eq:en-bound}
\|g(t)\|_{H^1}^2 + \|\partial_t \phi(t)\|_{L^2}^2 \leq C_0\eee^{-(x_2(t) - x_1(t))}, \\
|x_1'(t)| + |x_2'(t)| \leq C_0 \eee^{-\frac 12(x_2(t) - x_1(t))}, \label{eq:x'bound} \\
|x_1''(t)| + |x_2''(t)| \leq C_0 \eee^{-(x_2(t) - x_1(t))}. \label{eq:x''bound}
\end{gather}
\end{lemma}
\begin{proof}
Differentiating in $t$ the first relation in \eqref{eq:g-orth}, we obtain
\begin{equation}
\label{eq:x'bound1}
\begin{aligned}
0 &= \dd t\la \partial_x H_1(t), g(t)\ra = -x_1'(t)\la \partial_x^2 H_1(t), g(t)\ra + \la\partial_x H_1(t), \partial_t g(t)\ra \\
&= -x_1'(t)\la \partial_x^2 H_1(t), g(t)\ra + \la \partial_x H_1(t), -x_1'(t)\partial_x H_1(t) + x_2'(t) \partial_x H_2(t) + \partial_t \phi(t)\ra \\
&= x_1'(t)\big({-}\|\partial_x H\|_{L^2}^2 - \la \partial_x^2 H_1(t), g(t)\ra\big) + x_2'(t)\la \partial_x H_1(t), \partial_x H_2(t)\ra + \la \partial_x H_1(t), \partial_t \phi(t)\ra.
\end{aligned}
\end{equation}
Similarly, differentiating in $t$ the second relation in \eqref{eq:g-orth} yields
\begin{equation} \label{eq:x'bound2} 
0 = x_2'(t)\big(\|\partial_x H\|_{L^2}^2 - \la \partial_x^2 H_2(t), g(t)\ra\big) - x_1'(t)\la \partial_x H_1(t), \partial_x H_2(t)\ra
+ \la \partial_x H_2(t), \partial_t \phi(t)\ra.
\end{equation}
This can be viewed as a linear system for $x_1'(t)$ and $x_2'(t)$.
Note that $\lim_{t \to \infty}|\la \partial_x H_1(t), \partial_x H_2(t)\ra| = 0$ by \eqref{eq:inter-bound-3},
so the matrix of the system is diagonally dominant. In particular, we obtain
\begin{equation}
\label{eq:x'bound-2}
|x_1'(t)| + |x_2'(t)| \lesssim \|\partial_t \phi\|_{L^2}.
\end{equation}
Observe that \eqref{eq:lyap-schmidt-dtphi} and \eqref{eq:x'bound-2} yield
\begin{equation}
\label{eq:dtg-bound}
\|\partial_t g(t)\|_{L^2} \lesssim \|\partial_t \phi(t)\|_{L^2}.
\end{equation}

In order to prove \eqref{eq:en-bound}, we observe that \eqref{eq:g-to-0} and \eqref{eq:dist-to-infty} imply
\begin{equation}
\label{eq:thresh-energy}
E(\phi(t), \partial_t \phi(t)) = 2E_p(H).
\end{equation}
Indeed, applying Cauchy-Schwarz we have
\begin{equation}
\begin{aligned}
\bigg|\int_\bR \big(\partial_x (1 - H_1(t) + H_2(t) + g(t))\big)^2 \ud x - \int_\bR \big(\partial_x (1 - H_1(t) + H_2(t))\big)^2 \ud x\bigg| \\ \lesssim \|\partial_x g(t)\|_{L^2}\|\partial_x (1 - H_1(t) + H_2(t))\|_{L^2} + \|\partial_x g(t)\|_{L^2}^2 \to 0\ \text{as }t\ \to \infty.
\end{aligned}
\end{equation}
Plugging $w = 1 - H_1 + H_2$ in \eqref{eq:U-taylor} and integrating in $x$ we obtain
\begin{equation}
\begin{aligned}
\bigg|\int_\bR U(1 - H_1(t) + H_2(t) + g(t))\ud x - \int_{\bR}U(1 - H_1(t) + H_2(t)) \ud x\bigg| \\
\lesssim \|g\|_{L^2}\|U'(1 - H_1(t) + H_2(t))\|_{L^2} + \|g(t)\|_{L^2}^2 \to 0\ \text{as }t\ \to \infty,
\end{aligned}
\end{equation}
where in the last step we use boundedness of $\|U'(1 - H_1(t) + H_2(t))\|_{L^2}$, easy to justify by \eqref{eq:inter-bound-1}.

Since $\|\partial_t \phi(t)\|_{L^2}^2 \to 0$ as $t \to \infty$, from the last two estimates and \eqref{eq:Ep-limit} we deduce
\begin{equation}
E(\phi(t), \partial_t \phi(t)) = \lim_{t \to \infty} E(1 - H_1(t) + H_2(t) + g(t), \partial_t \phi(t)) = \lim_{t\to \infty}E_p(1 - H_1(t) + H_2(t)) = 2E_p(H).
\end{equation}

Plugging $w = 1 - H_1(t) + H_2(t)$ in \eqref{eq:U-taylor-2} we obtain
\begin{equation}
\label{eq:Ep-phi-expansion}
\begin{aligned}
E_p(\phi(t)) &= E_p(1 - H_1(t) + H_2(t)) + \la \vD E_p(1 - H_1(t) + H_2(t)), g\ra \\ &+ \frac 12 \la g(t), \vD^2 E_p(1 - H_1(t) + H_2(t))g(t)\ra
+ O(\|g\|_{L^2}^3).
\end{aligned}
\end{equation}
Applying Lemma~\ref{lem:D2H} we get
\begin{equation}
\label{energy aaa}
\begin{aligned}
\|g\|_{H^1}^2 + \|\partial_t \phi\|_{L^2}^2 &\lesssim E(\phi(t), \partial_t \phi(t)) - E_p(1 - H_1(t) + H_2(t)) \\&+ \|\vD E_p(1 - H_1(t) + H_2(t))\|_{L^2}\|g\|_{L^2}.
\end{aligned}
\end{equation}
By Lemma~\ref{lem:V-two-kink} and \eqref{eq:thresh-energy}, the right hand side is bounded up to a constant by
\begin{equation}
\eee^{-(x_2 - x_1)} + \|\vD E_p(1 - H_1(t) + H_2(t))\|_{L^2}\|g\|_{L^2}.
\end{equation}
From \eqref{eq:inter-bound-1} and Lemma~\ref{lem:exp-cross-term} we have $\|\vD E_p(1 - H_1(t) + H_2(t))\|_{L^2} \lesssim \sqrt{x_2 - x_1} \eee^{-(x_2 - x_1)}$, hence \eqref{eq:en-bound} follows.

Bound \eqref{eq:x'bound} follows from \eqref{eq:x'bound-2} and \eqref{eq:en-bound}.

In order to prove \eqref{eq:x''bound}, we differentiate \eqref{eq:x'bound1} and~\eqref{eq:x'bound2} in time.
For example, from~\eqref{eq:x'bound} we obtain,  
\EQ{
 0 & =  x_1''(t) \big({-}\|\partial_x H\|_{L^2}^2 - \la \partial_x^2 H_1(t), g(t)\ra\big)  + (x_1'(t))^2 \ang{ \p_x^3 H_1(t) , \, g(t)}  - x_1'(t) \ang{ \p_x^2 H_1(t), \, \p_t g(t) }  \\
 & \quad + x_2''(t)\ang{ \p_x H_1(t) , \, \p_x H_2(t)} - x_1'(t) x_2'(t) \ang{ \p_x^2 H_1(t) , \, \p_x H_2(t)}  - (x_2'(t))^2 \ang{ \p_x H_1(t) , \, \p_x^2 H_2(t)} \\
 & \quad - x_1'(t) \ang{ \p_x H_1(t), \, \p_t \phi(t)} + \ang{ \p_x H_1(t), \, \p_t^2 \phi(t)} 
}
Rearranging, and using~\eqref{eq:lyap-schmidt-dtphi} we obtain, 
\EQ{
x_1''(t) &\big({-}\|\partial_x H\|_{L^2}^2 - \la \partial_x^2 H_1(t), g(t)\ra\big) + x_2''(t)\ang{ \p_x H_1(t) , \, \p_x H_2(t)}  \\
& = (x_1'(t))^2 \ang{ \p_x^3 H_1(t) , \, g(t)} - x_1'(t)  \ang{ \p_x^2 H_1(t), \, \p_t \phi(t)  }  + (x_1'(t))^2 \ang{ \p_x^2 H_1(t) , \p_x H_1(t)}  \\
& \quad - 2x_1'(t) x_2'(t) \ang{ \p_x^2 H_1(t) , \p_x H_2(t)}    - (x_2'(t))^2 \ang{ \p_x H_1(t) , \, \p_x^2 H_2(t)}  - x_1'(t) \ang{ \p_x H_1(t), \, \p_t \phi(t)}  \\
&\quad + \ang{ \p_x H_1(t), \, \p_t^2 \phi(t)} 
}
After similarly differentiating~\eqref{eq:x'bound2} it is clear that we obtain again a diagonally dominant linear system for $x_1''(t)$ and $x_2''(t)$.
Almost all of  the terms on the right-hand side are easily seen to be $\lesssim \eee^{-(x_2(t) - x_1(t))}$,
because they are at least quadratic with respect to $(g, \partial_t \phi, x_1', x_2')$.
The only potentially problematic term is
\begin{equation}
\label{eq:problematic}
\la \partial_x H_1(t), \partial_t^2 \phi(t)\ra = \la \partial_x H_1(t), \partial_x^2 \phi(t) - U'(\phi(t))\ra.
\end{equation}
However, this term is also $\lesssim \eee^{-(x_2(t) - x_1(t))}$, due to the fact that
$\partial_x H \in \ker(-\partial_x^2 + U''(H))$.
Indeed, we have
\begin{equation}
\partial_x^2 \phi(t) - U'(\phi(t)) = \partial_x^2 g - U''(H_1)g + \Phi(x_1, x_2, \cdot) - \big(U'(1 - H_1 + H_2 + g)-U'(1 - H_1 + H_2) - U''(H_1)g\big),
\end{equation}
and we observe that
\begin{gather}
\la \partial_x H_1, \partial_x^2 g - U''(H_1)g\ra = 0, \\
\int_{\bR}|\partial_x H_1(x)\Phi(x_1, x_2, x)| \ud x \lesssim \int_\bR \eee^{{-|x-x_1|}}\eee^{-(x - x_1)_+}\eee^{-(x_2 - x)_+}\ud x \lesssim \eee^{-(x_2 - x_1)}, \\
\end{gather}
so we are left with the last term. From \eqref{eq:U'-taylor} we have
\begin{equation}
\|U'(1 - H_1 + H_2 + g)-U'(1 - H_1 + H_2) - U''(1 - H_1 + H_2)g\|_{L^2} \lesssim \|g\|_{L^2}^2 \lesssim \eee^{-(x_2 - x_1)}.
\end{equation}
From \eqref{eq:inter-bound-2} it follows that
\begin{equation}
\label{eq:proj-on-dxH1}
\begin{aligned}
\int_{\bR} |\partial_x H_1||U''(1 - H_1 + H_2) - U''(H_1)||g|\ud x \lesssim \|g\|_{L^2}\|(\partial_x H_1)(U''(1 - H_1 + H_2) - U''(H_1))\|_{L^2} \\
\lesssim \|g\|_{L^2}\sqrt{x_2 - x_1}\eee^{-(x_2 - x_1)} \ll \eee^{-(x_2 - x_1)}.
\end{aligned}
\end{equation}
Combining the two estimates yields the conclusion.
\end{proof}
The rest of this section closely follows the corresponding arguments in \cite{jendrej2018dynamics}. 

\begin{lemma}
\label{lem:virial}
For any $M > 0$ there exists a constant $C > 0$ such that for any functions $\wt \chi$, $w$ and
$g$ such that $\|w\|_{H^1} \leq M$, $\|\wt \chi\|_{{W^{1,\infty}}} < \infty$
and $\|g\|_{H^1} \leq 1$ the following inequality is true:
\begin{equation}
\label{eq:virial}
\begin{aligned}
\bigg|\int_\bR \wt \chi\,\partial_x g\big(U'(w + g) - U'(w)\big)\ud x
 + \int_\bR \wt \chi\,\partial_x w\big(U'(w+ g) - U'(w) - U''(w) g\big)\ud x \bigg| \\
 \leq C\|\partial_x \wt \chi\|_{L^\infty}\|g\|_{H^1}^2.
\end{aligned}
\end{equation}
\end{lemma}
\begin{proof}
By the standard approximation procedure,
we can assume that $w, g \in C_0^\infty(\bR)$.

Consider the first line in \eqref{eq:virial}. Rearranging the terms, we obtain
\begin{equation}
\begin{aligned}
&\int_\bR \wt \chi\Big(\partial_x g\big(U'(w+ g) - U'(w)\big)
+ \partial_x w\big(U'(w+ g) - U'(w) - U''(w) g\big)\Big)\ud x \\
= &\int_\bR \wt\chi\Big(\partial_x (w + g)U'(w + g) - \partial_x w\, U'(w)
 - \big(\partial_x g\, U'(w) + g\, \partial_x w\, U''(w)\big)\Big)\ud x \\
 =& \int_\bR \wt\chi\, \partial_x\big(U(w + g) - U(w) - U'(w) g\big)\ud x,
\end{aligned}
\end{equation}
and we can integrate by parts.
\end{proof}

%
Recall that $\chi \in C^\infty$ is a decreasing function such that $\chi(x) = 1$ for $x \leq \frac 13$
and $\chi(x) = 0$ for $x \geq \frac 23$. We define
\begin{equation}
\label{eq:mod-cutoff-def}
\chi_1(t, x) := \chi\Big(\frac{x - x_1(t)}{x_2(t) - x_1(t)}\Big),\quad \chi_2 := 1 - \chi_1.
\end{equation}
Now the analysis is based on an ad-hoc change of unknowns in the modulation equations in order to remove some terms of low order.
We consider the following continuous real-valued functions:
\begin{equation}
\label{eq:impulsion-def}
\begin{aligned}
p_1(t) &:= \|\partial_x H\|_{L^2}^{-2}\la \partial_x(H_1(t) - \chi_1(t) g(t)), \partial_t \phi\ra, \\
p_2(t) &:= \|\partial_x H\|_{L^2}^{-2}\la -\partial_x(H_2(t) + \chi_2(t) g(t)), \partial_t \phi(t)\ra.
\end{aligned}
\end{equation}
\begin{lemma}
\label{lem:norm-form}
If $\phi$ is a strongly interacting kink-antikink pair, then
there exist $C, T_0 > 0$ such that $p_j \in C^1([T_0, \infty))$ and for all $t \geq T_0$
\begin{align}
\big|x_j'(t) - p_j(t)\big| &\leq C\eee^{-(x_2(t) - x_1(t))}, \label{eq:x'-p} \\
\big|p_j'(t) + (-1)^j F(x_2(t) - x_1(t))\big| &\leq C(x_2(t) - x_1(t))^{-1}\eee^{-(x_2(t) - x_1(t))}, \label{eq:p'-force}
\end{align}
where $F$ is defined by \eqref{eq:force-def}.
\end{lemma}
\begin{proof}
We can assume that $\phi$ is smooth, so that $p_j$ are of class $C^1$.
We obtain the general case by approximating a finite energy solution with smooth ones and passing to the limit.

We prove the inequalities for $j = 1$, the arguments for $j = 2$ being analogous.

Using the estimates obtained in Lemma~\ref{lem:en-bound}, \eqref{eq:x'bound1} yields
\begin{equation}
\big| x_1'(t) - \|\partial_x H\|_{L^2}^{-2}\la \partial_x H_1(t), \partial_t \phi(t)\ra\big| \lesssim \eee^{-(x_2(t) - x_1(t))}.
\end{equation}
We also have
\begin{equation}
|\la \partial_x(\chi_1(t) g(t)), \partial_t \phi(t)\ra| \lesssim (\|g(t)\|_{L^2} + \|\partial_x g(t)\|_{L^2})\|\partial_t \phi(t)\|_{L^2}
\lesssim \eee^{-(x_2(t) - x_1(t))},
\end{equation}
so \eqref{eq:x'-p} is proved.

We have
\begin{equation}
\label{eq:p1'}
\begin{aligned}
\|\partial_x H\|_{L^2}^2 p_1'(t) &= \la \partial_t (\partial_x H_1(t)), \partial_t \phi(t)\ra - \la \partial_x(g(t)\partial_t \chi_1(t)), \partial_t \phi(t)\ra - \la \partial_x(\chi_1(t)\partial_t g(t)), \partial_t \phi(t)\ra \\
&+ \la \partial_x H_1(t), \partial_t^2 \phi(t)\ra - \la \partial_x \chi_1(t) g(t), \partial_t^2 \phi(t)\ra - \la\chi_1(t)\partial_x g(t), \partial_t^2 \phi(t)\ra \\
&= I + II + III + IV + V + VI,
\end{aligned}
\end{equation}
and we will estimate each term one by one. Until the end of this proof, we will say that some quantity is negligible
if it is $\lesssim (x_2 - x_1)^{-1}\eee^{-(x_2 - x_1)}$, and we use the symbol $\simeq$ for equalities up to negligible quantities.

By the Chain Rule we have
\begin{equation}
\label{eq:dt-chi1}
\partial_t \chi_1(t, x) = \frac{-x_1'(t)(x_2(t) - x_1(t)) - (x - x_1(t))(x_2'(t) - x_1'(t))}{(x_2(t) - x_1(t))^2}\partial_x
\chi\Big(\frac{x - x_1(t)}{x_2(t) - x_1(t)}\Big),
\end{equation}
which yields, using \eqref{eq:x'bound}, $\|\partial_t \chi_1(t)\|_{L^\infty} \lesssim (x_2(t) - x_1(t))^{-1}\eee^{-\frac 12(x_2(t) - x_1(t))}$.
Thus $II \simeq 0$. Using \eqref{eq:dtg-bound} and
\begin{equation}
\label{eq:dxchi1-bound}
\|\partial_x \chi_1(t)\|_{L^\infty} \lesssim (x_2(t) - x_1(t))^{-1}
\end{equation}
we obtain
\begin{equation}
\begin{aligned}
III &\simeq -\la \chi_1(t)\partial_x \partial_t g(t), \partial_t \phi(t)\ra \\
&= -\la \chi_1(t) \partial_x \partial_t \phi(t), \partial_t \phi(t)\ra
+ x_1'(t) \la \chi_1(t) \partial_x^2 H_1(t), \partial_t \phi(t)\ra - x_2'(t) \la \chi_1(t) \partial_x^2 H_2(t), \partial_t \phi(t)\ra
\end{aligned}
\end{equation}
Integrating by parts and using again \eqref{eq:dxchi1-bound}, we see that the first term of the second line is negligible.
The last term is negligible as well, because $\partial_x^2 H_2(t)$ is (exponentially) small on the support of $\chi_1(t)$.
For a similar reason, we can remove $\chi_1(t)$ from the second term, and obtain
\begin{equation}
III \simeq x_1'(t) \la \partial_x^2 H_1(t), \partial_t \phi(t)\ra = - I,
\end{equation}
in other words we have $I + II + III \simeq 0$.

In order to estimate the remaining three terms, we write
\begin{equation}
\partial_t^2\phi = \partial_x^2 \phi(t) - U'(\phi(t)) = \Phi(x_1, x_2, \cdot) +\partial_x^2 g- \big(U'(1 - H_1 + H_2 + g)-U'(1 - H_1 + H_2)\big).
\end{equation}
In particular, examining the contribution of each term on the right above, and using~\eqref{eq:inter-bound-1},  \eqref{eq:en-bound}, \eqref{eq:dxchi1-bound} and the fact that $U'$ is locally Lipschitz we see that
the term $V$ is negligible.
Consider the term $VI$. Integrating by parts, we see that $\la \chi_1\partial_x g, \partial_x^2 g\ra$ is negligible.
By \eqref{eq:inter-bound-1} and Cauchy-Schwarz, $\la \chi_1\partial_x g, \Phi(x_1, x_2, \cdot)\ra$ is negligible as well.
Hence, by Lemma~\ref{lem:virial} with $\wt\chi = \chi_1$ and $w = 1 - H_1 + H_2$, and using again that $\p_x H_1(t)$ is exponentially small outside the support of $\chi_1(t)$ we have
\begin{equation}
\label{eq:VI}
\begin{aligned}
VI &\simeq \int_{\bR}\chi_1(\partial_x H_1 - \partial_x H_2)\big(U'(1 - H_1 + H_2 + g) - U'(1 - H_1 + H_2) - U''(1 - H_1 + H_2)g\big)\ud x \\
&\simeq {\int_{\R} } \partial_x H_1\big(U'(1 - H_1 + H_2 + g) - U'(1 - H_1 + H_2) - U''(1 - H_1 + H_2)g\big)\ud x \\
&\simeq{\int_{\R}} \partial_x H_1\big(U'(1 - H_1 + H_2 + g) - U'(1 - H_1 + H_2) - U''(H_1)g\big)\ud x.
\end{aligned}
\end{equation}
We already encountered the term $IV$, see \eqref{eq:problematic}, where we obtained
\begin{equation}
IV = \la \partial_x H_1, \Phi(x_1, x_2, \cdot)\ra - \big\la \partial_x H_1, \big(U'(1 - H_1 + H_2 + g) - U'(1 - H_1 + H_2) - U''(H_1)g\big)\big\ra.
\end{equation}
The last term cancels with the term $VI$ and, recalling the definition of $F$, we get \eqref{eq:p'-force}.
\end{proof}
\begin{proposition}
\label{prop:regime}
Let  $A$ be the constant defined  by \eqref{def AAA}.
If $\phi$ is a strongly interacting kink-antikink pair,
then there exist $C, T_0 > 0$ (depending on $\phi$) such that for all $t \geq T_0$
\begin{gather}
\label{eq:regime-1} 2 t^{-1} - C(t\log t)^{-1} \leq x_2'(t) - x_1'(t) \leq 2t^{-1} + C(t\log t)^{-1}, \\
\label{eq:regime-2} 2\log(At) - C(\log t)^{-1} \leq x_2(t) - x_1(t) \leq 2\log(At) + C(\log t)^{-1}, \\
\label{eq:regime-3} \|g(t)\|_{H^1} + \|\partial_t g(t) \|_{L^2} \leq Ct^{-1}(\log t)^{-1/2}.
\end{gather}
\end{proposition}
\begin{remark}
The estimates given in Theorem~\ref{thm:main} are stronger. However, proving the preliminary bounds above
is crucial for our proof of Theorem~\ref{thm:main} given in the next section.
The fact that the distance between the kinks is estimated with precision $(\log t)^{-1}$ is not crucial.
In order for the arguments in the next section to work, this could be any function converging to $0$ as $t \to \infty$.

\end{remark}
\begin{proof}
Set $z(t) := x_2(t) - x_1(t)$ and $p(t) := p_2(t) - p_1(t)$.
Lemma~\ref{lem:norm-form} together with Lemma~\ref{lem:Fz} yield, 
\begin{equation}
\label{eq:mod-for-diff}
|z'(t) - p(t)| \lesssim \eee^{-z(t)}, \qquad \big|p'(t) + 2A^2\eee^{-z(t)}\big| \lesssim z(t)^{-1}\eee^{-z(t)}.
\end{equation}
By assumption \eqref{eq:dist-to-infty-0} $\lim_{t\to \infty} z(t)=\infty$. We claim that for $t$ large enough $z(t)$ is a strictly increasing function.
Let $t_1 \geq T_0$, where $T_0$ is large (chosen later in the proof).
We need to show that for all $t > t_1$ we have $z(t) > z(t_1)$.
Suppose this is not the case, and let
\begin{equation}
t_2 := \sup\big\{t: z(t) = \inf_{\tau \geq t_1}z(\tau)\big\}.
\end{equation}
Then $t_2 > t_1$ is finite, $z(t_2) = \inf_{t_1 \le \tau \leq t_2}z(\tau)$ and $z'(t_2) = 0$.

Let $z_0 := z(t_2)$, $t_3 := \inf\{t \geq t_2: z(t) = z_0 + 1\}$.
Since $\lim_{t \to \infty}z(t) = \infty$, $t_3$ is finite.
We will show that the inequalities \eqref{eq:mod-for-diff} imply
\begin{equation}
\label{eq:q-mono-0}
z(t_3) \leq z_0 + \frac 12,
\end{equation}
which is a contradiction. Note that $z_0 \leq z(t) \leq z_0 + 1$ for $t \in [t_2, t_3]$, in particular $\eee^{-z(t)} \simeq \eee^{-z_0}$.

From \eqref{eq:mod-for-diff} we have
\begin{equation}
\label{eq:z-mono-1}
D_+ p(t) \leq -A^2\eee^{-z(t)} \leq -A^2\eee^{-z_0 - 1} = -\frac{A^2}{\eee} \eee^{-z_0}, \qquad\text{for all }t \in [t_2, t_3].
\end{equation}
Since $p(t_2)=p(t_2)-z'(t_2)\leq C\,\eee^{-z(t_2)} = C\,\eee^{-z_0}$, we get
\begin{equation}
\label{eq:q-mono-2}
p(t) \leq C\,\eee^{-z_0} - \frac{A^2}{\eee}(t - t_2)\eee^{-z_0}, \qquad\text{for all }t\in [t_2, t_3].
\end{equation}
Using \eqref{eq:mod-for-diff} again we obtain
\begin{equation}
z'(t) \leq C\,\eee^{-z_0} - \frac{A^2}{\eee}(t - t_2)\eee^{-z_0}, \qquad\text{for all }t\in [t_2, t_3].
\end{equation}
We now integrate for $t$ between $t_2$ and $t_3$:
\begin{equation}
\begin{aligned}
z(t_3) - z(t_2) &\leq  \int_{t_2}^{t_3}\Big({-}\frac{A^2}{\eee}(t - t_2)\eee^{-z_0} + C\,\eee^{-z_0}\Big)\ud t \\
&=-\frac 12 \frac{A^2}{\eee}\eee^{-z_0}(t_3 - t_2)^2 + C\,\eee^{-z_0}(t_3 - t_2) \\
&\ \leq \eee^{-z_0}\sup_{s>0} ( -\frac 12 \frac{A^2}{\eee} s^2+Cs)\\
&\leq \frac{\eee\,C^2}{2 A^2}\eee^{-z_0},
\end{aligned}
\end{equation}
so that \eqref{eq:q-mono-0} follows if $T_0$ (hence also $z_0$) is large enough. 

Set $r(t) := p(t) - 2A\eee^{-\frac 12 z(t)}$. Using~\eqref{eq:mod-for-diff} we have
\begin{equation}
\label{eq:dtr}
\begin{aligned}
r'(t) = p'(t) + Az'(t)\eee^{-\frac 12 z(t)} = -2A^2\eee^{-z(t)} + Ap(t)\eee^{-\frac 12 z(t)} + O(z(t)^{-1}\eee^{-z(t)}) \\
=  {+} A \eee^{-\frac 12 z(t)}\big(r(t) + O(z(t)^{-1}\eee^{-\frac 12 z(t)})\big).
\end{aligned}
\end{equation}
This implies that there exists $C > 0$ such that 
\EQ{ \label{eq:r-est} 
|r(t)| \leq Cz(t)^{-1}\eee^{-\frac 12 z(t)}, 
}
 for $t$ large enough.
Indeed, suppose there exists $t_1$ arbitrarily large such that $r(t_1) > Cz({t_1})^{-1}\eee^{-\frac 12 z({t_1})}$
(the case $r(t_1) < -Cz(t_1)^{-1}\eee^{-\frac 12 z(t_1)}$ is similar).
Let $t_2 := \sup\{t: r(t) = Cz(t_1)^{-1}\eee^{-\frac 12 z(t_1)}\}$. Since $\lim_{t \to \infty} r(t) = 0$, we have $t_2 \in (t_1, \infty)$
and $r'(t_2) \leq 0$.
Since $z(t)$ is non-decreasing, we have $r(t_2) = Cz(t_1)^{-1}\eee^{-\frac{z(t_1)}{2}} \geq Cz(t_2)^{-1}\eee^{-\frac{z(t_2)}{2}}$.
Thus, if we choose $C$ large enough, \eqref{eq:dtr} yields $r'(t_2) > 0$, a contradiction.

We deduce  {from~\eqref{eq:mod-for-diff}, the definition of $r(t)$, and~\eqref{eq:r-est} } that for some $t_0 > 0$ and all $t \geq t_0$ we have
\begin{equation}
\label{eq:z'-ineq}
\big|z'(t) - 2A \eee^{-\frac 12 z(t)}\big| \leq 2C z(t)^{-1}\,\eee^{-\frac 12 z(t)}
\ \Leftrightarrow\ \big|\big(\eee^{\frac 12 z(t)}\big)' - A\big| \leq Cz(t)^{-1},
\end{equation}
which implies, after integrating,
\begin{equation}
\label{eq:z'-asympt}
\begin{gathered}
(A - o(1))t \leq \eee^{\frac{z(t)}{2}} \leq (A + o(1))t\ \Leftrightarrow \\
 2\log t + 2\log(A - o(1)) \leq z(t) \leq 2\log t + 2\log(A + o(1)),
\end{gathered}
\end{equation}
for $t$ large enough.
Once we know that $z(t) \simeq \log t$, \eqref{eq:regime-2} follows by integrating \eqref{eq:z'-ineq} and taking the logarithm,
similarly as in \eqref{eq:z'-asympt} but with $(\log t)^{-1}$ instead of $o(1)$.
The bound \eqref{eq:regime-1} follows by inserting \eqref{eq:regime-2} into \eqref{eq:z'-ineq}.

We are left with \eqref{eq:regime-3}.
We claim that
\begin{equation}
\label{eq:dtphi-expansion}
\|\partial_t \phi(t)\|_{L^2}^2 = \big((x_1'(t))^2 + (x_2'(t))^2\big)\|\partial_x H\|_{L^2}^2 + \|\partial_t g(t)\|_{L^2}^2 + O(t^{-3}).
\end{equation}
Indeed, differentiating \eqref{eq:g-orth} we obtain $|\la \partial_x H_1(t), \partial_t g(t)\ra| \lesssim t^{-2}$,
so \eqref{eq:dtphi-expansion} follows by squaring \eqref{eq:lyap-schmidt-dtphi} and using \eqref{eq:inter-bound-3}.
Now, we observe that
\begin{equation}
\label{eq:x1'x2'lbound}
(x_1'(t))^2 + (x_2'(t))^2 \geq \frac 12 (x_2'(t) - x_1'(t))^2 \geq 2t^{-2} - Ct^{-2}(\log t)^{-1},
\end{equation}
where the last inequality follows from \eqref{eq:regime-1}.

On the other hand, from \eqref{eq:regime-2} we deduce
\begin{equation}
2{\kappa}^2 \eee^{-(x_2(t) - x_1(t))} = 2{\kappa}^2 A^{-2}t^{-2} + O(t^{-2}(\log t)^{-1}) = t^{-2}\|\partial_x H\|_{L^2}^2 + O(t^{-2}(\log t)^{-1}).
\end{equation}
By \eqref{energy aaa} and  \eqref{eq:V-two-kink},
for some $c > 0$ and $t$ large enough we have
\begin{equation}
c\|g(t)\|_{H^1}^2 + \frac 12 \|\partial_t \phi(t)\|_{L^2}^2 \leq t^{-2}\|\partial_x H\|_{L^2}^2 + O(t^{-2}(\log t)^{-1}),
\end{equation}
so \eqref{eq:regime-3} follows from \eqref{eq:dtphi-expansion} and \eqref{eq:x1'x2'lbound}.
\end{proof}
\begin{remark}
As a by-product of the proof of \eqref{eq:regime-3}, we can deduce that
$-t^{-1} - Ct^{-1}(\log t)^{-1/2} \leq x_1'(t) \leq -t^{-1} + Ct^{-1}(\log t)^{-1/2}$ and
$t^{-1} - Ct^{-1}(\log t)^{-1/2} \leq x_2'(t) \leq t^{-1} + Ct^{-1}(\log t)^{-1/2}$.
However, at this stage it is not clear whether $x_1(t) + \log t$ and $x_2(t) - \log t$ converge as $t\to\infty$.
\end{remark}

\section{The existence and uniqueness of the strongly interacting kink-antikink pair}
\label{sec:nonlin-anal}

\subsection{An implementation of the Lyapunov-Schmidt reduction}

Our strategy can be summarized as follows. Our aim is to find the strongly interacting kink-antikink pair $\phi(t,x)$ as a solution of \eqref{eq:phi4}. 
We assume {\it a priori} that $\phi=1-H_1+H_2+g$ and  that \eqref{eq:g-to-0}--\eqref{eq:g-orth} hold.
Projecting the equation \eqref{eq:phi4} onto the space spanned by $\partial_x H_j$, $j = 1, 2$, and onto its orthogonal complement,
we are lead to solving the {\it projected equation}
\begin{equation}
\label{lsch 1}
\partial_{t}^2 \phi(t, x)=\partial_{x}^2\phi(t, x)+U'(\phi(t, x))+\lambda_1(t)\partial_x H_1(t, x)+\lambda_2(t)\partial_x H_2(t, x)
\end{equation}
and the \emph{bifurcation equations}
\begin{equation}
\label{lsch 2}
\lambda_1(t)=0, \qquad \lambda_2(t)=0.
\end{equation}

Written in terms of $g$, the equation  \eqref{lsch 1}  is
\begin{equation}
\label{eq:phi4-step1}
\begin{aligned}
\partial_t^2 g &+ x_1''\partial_x H_1 - (x_1')^2\partial_x^2 H_1 - x_2''\partial_x H_2 + (x_2')^2 \partial_x^2 H_2 \\
 &- \partial_x^2 g + U'(1 - H_1 + H_2 + g) + U'(H_1) - U'(H_2)
 = \lambda_1 \partial_x H_1
+ \lambda_2 \partial_x H_2.
\end{aligned}
\end{equation}

We will first study solutions $(g, \lambda_1, \lambda_2)$ of \eqref{eq:phi4-step1} for a \emph{given} pair of trajectories $(x_1, x_2)$ satisfying
\begin{gather}
\label{eq:x1-x2-cond-1}
2\log t - C_0 \leq x_2(t) - x_1(t) \leq 2\log t + C_0, \\
\label{eq:x1-x2-cond-2}
|x_1'(t)| + |x_2'(t)| \leq C_0 t^{-1}, \\
\label{eq:x1-x2-cond-3}
|x_1''(t)| + |x_2''(t)| \leq C_0 t^{-2}
\end{gather}
for some $T_0 > 0$ and all $t \geq T_0$; note that \eqref{eq:regime-1}, \eqref{eq:x'bound} and \eqref{eq:x''bound}
guarantee that any kink-antikink pair falls into this regime.

Let us specify what we mean by a ``solution''. For $(x_1, x_2)$ given,
we say that $(g, \lambda_1, \lambda_2)$ solves \eqref{eq:phi4-step1} on a compact time interval $[T_1, T_2]$
if $(g, \partial_t g) \in C([T_1, T_2], \cE)$, \eqref{eq:g-orth} holds for all $T_1 \leq t \leq T_2$,
$\lambda_j \in C([T_1, T_2])$ and $\partial_t^2 g - \partial_x^2 g = r$ in the weak sense, where
\begin{equation}
\begin{aligned}
r := &-x_1''\partial_x H_1 + (x_1')^2\partial_x^2 H_1 + x_2''\partial_x H_2 - (x_2')^2 \partial_x^2 H_2 \\
&- U'(1 - H_1 + H_2 + g) - U'(H_1) + U'(H_2)
 + \lambda_1 \partial_x H_1
+ \lambda_2 \partial_x H_2.
\end{aligned}
\end{equation}
We say that $(g, \lambda_1, \lambda_2)$ solves \eqref{eq:phi4-step1} on $[T_0, \infty)$
if it does so on any interval $[T_0, T]$ for all $T > T_0$.

We will prove that for any pair $(x_1, x_2)$, equation \eqref{lsch 1} has a unique solution $(g, \lambda_1, \lambda_2)$
such that $(g, \partial_t g)$ belongs to a sufficiently small ball of the space $N_1(\cE)$;
note that this last condition is guaranteed by \eqref{eq:regime-3} for any kink-antikink pair.
Then, we will find all the pairs $(x_1, x_2)$ such that \eqref{lsch 2} holds.
It turns out that  \eqref{lsch 2} yields a nonlocal and nonlinear system of second order ODEs for $x_1(t), x_2(t)$. 

The norms $N_\gamma, S_\gamma, W_{\alpha, \beta}$ defined by \eqref{eq:Wab-def} will play an important role in our proof,
as they will allow us to set up a fixed point scheme. 
The following lemma provides some basic facts about them.
\begin{lemma}
\label{lem:Wab-prop}
(i)
For all $\gamma > 0$ and $\alpha \in (-\infty, \gamma)$, the space $N_\gamma \cap W_{\alpha, \gamma} := \{z \in N_\gamma: \|z\|_{W_{\alpha, \gamma}} < \infty\}$  with the norm $\|\cdot\|_{N_\gamma \cap W_{\alpha, \gamma}} := \max(\|\cdot\|_{N_\gamma}, \|\cdot\|_{W_{\alpha, \gamma}})$ is a Banach space,
in which $N_{\gamma + 1}$ is continuously embedded.

(ii) For all $\gamma > 0$ and $\alpha \in (-\infty, \gamma)$ there exists $C = C(\gamma, \alpha)$ such that
for all $z \in C^1$
\begin{equation}
\|z'\|_{W_{\alpha, \gamma}} \leq C \|z\|_{N_\gamma}.
\end{equation}

(iii)
If $\mu \geq 0$ and $\gamma > \frac 12(\sqrt{1+4\mu} - 1)$, then for any $v \in N_{\gamma + 1} \cap W_{\mu^-,\gamma + 1} \cap W_{\mu^+,\gamma + 1}$, where $\mu^\pm := \frac 12 (1 \pm \sqrt{1 + 4\mu})$, the equation
\begin{equation}
\label{eq:euler}
z'' = \mu t^{-2}z + v
\end{equation}
has a unique solution $z \in S_\gamma$. The mapping $v \mapsto z$ is a bounded linear operator $
N_{\gamma + 1} \cap W_{\mu^-,\gamma + 1} \cap W_{\mu^+,\gamma + 1} \to S_\gamma$.
\end{lemma}
\begin{proof}
Point (i) is left to the reader. Point (ii) follows by integrating by parts in time.
We now prove (iii).
The definition of $W_{\mu^\pm, \gamma + 1}$ and the fact that $\gamma + 1 > \mu^\pm$ imply
\begin{equation}
\lim_{t \to \infty}\sup_{\tau \geq t}\bigg|\int_t^\tau s^{\mu^\pm}v(s)\ud s\bigg| = 0,
\end{equation}
thus for all $t \geq T_0$ the integral $\int_t^\infty s^{\mu^\pm}v(s)\ud s$ exists as an improper Riemann integral.
Moreover, we have
\begin{equation}
\label{eq:euler-bound-forcing}
\sup_{t \geq T_0} t^{\gamma+1-\mu^\pm}\left|\int_t^\infty s^{\mu^\pm} v(s)\ud s\right| \leq \|v\|_{W_{\mu^\pm, \gamma + 1}}.
\end{equation}
We observe that \eqref{eq:euler} is a standard Euler differential equation. A particular solution is given by
\begin{equation}
z(t) :=\frac{1}{\sqrt{1+4\mu}} \left(t^{\mu^-}\int_t^\infty s^{\mu^+} v(s)\ud s-t^{\mu^+}\int_t^\infty s^{\mu^-} v(s)\ud s\right).
\end{equation}
From \eqref{eq:euler-bound-forcing} we easily deduce that
\begin{equation}
\sup_{t \geq T_0}t^\gamma |z(t)| \leq \frac{1}{\sqrt{1+4\mu}}\big(\|v\|_{W_{\mu^+, \gamma + 1}}+\|v\|_{W_{\mu^-, \gamma + 1}}\big)
\lesssim \|v\|_{W_{\mu^+, \gamma + 1}\cap W_{\mu^-, \gamma + 1}}.
\end{equation}
We have
\begin{equation}
z'(t) :=\frac{1}{\sqrt{1+4\mu}} \left(\mu^-t^{\mu^- - 1}\int_t^\infty s^{\mu^+} v(s)\ud s-\mu_+t^{\mu^+ - 1}\int_t^\infty s^{\mu^-} v(s)\ud s\right),
\end{equation}
thus, analogously,
\begin{equation}
\sup_{t \geq T_0}t^{\gamma+1} |z'(t)| \lesssim \|v\|_{W_{\mu^+, \gamma + 1}\cap W_{\mu^-, \gamma + 1}}.
\end{equation}
The fact that $\|z''\|_{N_{\gamma+1}} \lesssim \|v\|_{N_{\gamma + 1}\cap W_{\mu^+, \gamma + 1}\cap W_{\mu^-, \gamma + 1}}$
follows from the equation \eqref{eq:euler}. This finishes the proof that
\begin{equation}
\|z\|_{S_\gamma} \lesssim \|v\|_{N_{\gamma + 1}\cap W_{\mu^+, \gamma + 1}\cap W_{\mu^-, \gamma + 1}}.
\end{equation}

Regarding uniqueness, the general solution of \eqref{eq:euler} is
\begin{equation}
z_g(t) = z(t) + c^+t^{\mu^+} + c^- t^{\mu^-}.
\end{equation}
Since $\gamma > -\mu^- > -\mu^+$, it is clear that $c^+t^{\mu^+} + c^- t^{\mu^-} \notin N_\gamma$ unless $c^+ = c^- = 0$.

\end{proof}
For $\gamma > 2$ we set { 
\EQ{ \label{eq:Wdef} 
{\mathbf W}_{\gamma} := \bigcap_{\alpha \in \{{-}1, 0, 1, 2\}}W_{\alpha, \gamma}.
}}
As we will see, the four indices $\alpha$ correspond to the characteristic exponents of certain differential equations
of the form \eqref{eq:euler} appearing in the proof when we solve the bifurcation equation.

\subsection{The linear equation associated with \eqref{eq:phi4-step1}} \label{s:lin} 
In this subsection we treat the linear equation associated to \eqref{eq:phi4-step1} for given trajectories $(x_1, x_2)$ satisfying~\eqref{eq:x1-x2-cond-1},
\eqref{eq:x1-x2-cond-2} and \eqref{eq:x1-x2-cond-3}. We also compare solutions associated to two different sets of trajectories $(x_1, x_2)$ and $(\sh x_1,\sh  x_2)$ in preparation for the contraction mapping argument performed in Section \ref{s:contraction1}.

In the next lemma, we solve the linear problem corresponding to \eqref{eq:phi4-step1}.
\begin{lemma}
\label{lem:phi4-step1-lin}
For any $\gamma > 1$ and $\beta \in (2, \gamma + 1)$ there exists $C = C(\beta, \gamma) > 0$
and $T_0 = T_0(\beta, \gamma)$ such that the following holds.
For all $(x_1, x_2)$ satisfying \eqref{eq:x1-x2-cond-1},
\eqref{eq:x1-x2-cond-2} and \eqref{eq:x1-x2-cond-3}, and all $f \in N_{\gamma + 1}(L^2)$, the system
\begin{gather}
\label{eq:phi4-lind}
\partial_t^2 h - \partial_x^2 h + U''(1-H_1 + H_2)h 
= f +\lambda_1 \partial_x H_1 + \lambda_2 \partial_x H_2, \\
\label{eq:phi4-lind-orth}
\la \partial_x H_1, h\ra = \la \partial_x H_2, h\ra = 0
\end{gather}
has a unique solution $(h, \lambda_1, \lambda_2)$ such that $(h, \partial_t h) \in N_\gamma(\cE)$.
Moreover, this solution satisfies
\begin{align}
\label{eq:phi4-step1-lin}
\|(h, \partial_t h)\|_{N_{\gamma}(\cE)} + \sum_{j=1}^2
\big\|\lambda_j + \|\partial_x H\|_{L^2}^{-2}\la \partial_x H_j, f\ra \big\|_{{\mathbf W}_{\beta} \cap N_{\gamma+1}} \leq C \|f\|_{N_{\gamma+1}(L^2)}.
\end{align}
If $\gamma =1$, the same result holds without the inclusion of the ${\mathbf W}_{\beta}$ norm on the left-hand side of~\eqref{eq:phi4-step1-lin}.
\end{lemma}
\begin{remark}
Similarly as in the case of \eqref{eq:phi4-step1}, we say that $(h, \lambda_1, \lambda_2)$ is a solution of \eqref{eq:phi4-lind}--\eqref{eq:phi4-lind-orth} on a compact time interval $[T_1, T_2]$ if $(h, \partial_t h) \in C([T_1, T_2], \cE)$,
$\lambda_j \in C([T_1, T_2])$ and $\partial_t^2 h - \partial_x^2 h = r$ in the weak sense, where
\begin{equation}
\label{eq:r-def}
r := -U''(1-H_1 + H_2)h + f +\lambda_1 \partial_x H_1 + \lambda_2 \partial_x H_2,
\end{equation}
with the usual extension to the case of $[T_0, \infty)$.
\end{remark}
\begin{proof}

\noindent
\textbf{Step 1.} (Computation of the Lagrange multipliers.)
Assume that $(h, \lambda_1, \lambda_2)$ solves \eqref{eq:phi4-lind}--\eqref{eq:phi4-lind-orth}
on some time interval. Differentiating in time the orthogonality conditions, we obtain
\begin{equation}
\label{eq:deriv-orth-1}
0 = \dd t\la \partial_x H_j, h\ra = \la \partial_x H_j, \partial_t h\ra - x_j'\la \partial_x^2 H_j, h\ra.
\end{equation}
Differentiating again we get
\begin{equation}
\la \partial_x H_j, \partial_t^2 h\ra = 2x_j'\la \partial_x^2 H_j, \partial_t h\ra + x_j''\la \partial_x^2 H_j, h\ra - (x_j')^2\la \partial_x^3 H_j, h\ra.
\end{equation}
Multiplying \eqref{eq:phi4-lind} by $\partial_x H_1$ and integrating in $x$ we get
\begin{equation}
\label{eq:proj-H1}
\begin{aligned}
2x_1'\la \partial_x^2 H_1, \partial_t h\ra + x_1''\la \partial_x^2 H_1, h\ra - (x_1')^2\la \partial_x^3 H_1, h\ra+ \la (U''(1-H_1 + H_2) - U''(H_1))\partial_x H_1, h\ra \\
= \la \partial_x H_1, f\ra + \lambda_1 \|\partial_x H\|_{L^2}^2 + \lambda_2\la \partial_x H_1, \partial_x H_2\ra.
\end{aligned}
\end{equation}
Multiplying \eqref{eq:phi4-lind} by $\partial_x H_2$ and integrating in $x$ we get
\begin{equation}
\label{eq:proj-H2}
\begin{aligned}
2x_2'\la \partial_x^2 H_2, \partial_t h\ra + x_2''\la \partial_x^2 H_2, h\ra - (x_2')^2\la \partial_x^3 H_2, h\ra+ \la (U''(1-H_1 + H_2) - U''(H_2))\partial_x H_2, h\ra \\
= \la \partial_x H_2, f\ra + \lambda_2 \|\partial_x H\|_{L^2}^2 + \lambda_1\la \partial_x H_1, \partial_x H_2\ra.
\end{aligned}
\end{equation}
These two equalities form a linear system for $\lambda_1$ and $\lambda_2$.
By \eqref{eq:inter-bound-3}, its matrix is strictly diagonally dominant.
By \eqref{eq:x1-x2-cond-1} and \eqref{eq:inter-bound-2}, we know that
\begin{equation}
\big|\la (U''(1-H_1 + H_2) - U''(H_1))\partial_x H_1, h\ra\big| \ll t^{-1}\|h\|_{L^2},
\end{equation}
hence we obtain
\begin{equation}
\label{eq:l1-l2-bound}
|\lambda_1(t)| + |\lambda_2(t)| \lesssim \|f(t)\|_{L^2} + t^{-1} \|(h(t), \partial_t h(t))\|_\cE.
\end{equation}

\textbf{Step 2.} (Local existence.)
Let $[T_1, T_2] \owns t_0 \in \bR$, $f \in C([T_1, T_2], L^2)$ and $(h_0, h_1) \in \cE$ satisfying the orthogonality and compatibility conditions
\begin{equation}
\la \partial_x H_j(t_0), h_0\ra = 0, \qquad \la \partial_x H_j(t_0), h_1\ra = x_j'(t_0)\la \partial_x^2 H_j(t_0), h_0\ra.
\end{equation}
We prove that there exists a unique solution of \eqref{eq:phi4-lind}--\eqref{eq:phi4-lind-orth}
for the initial data $(h(t_0), \partial_t h(t_0)) = (h_0, h_1)$.
For given $(h, \partial_t h) \in C([T_1, T_2], \cE)$, we define $\lambda_j$ as the solution of the system \eqref{eq:proj-H2}
and $\wt h$ as the solution of $\partial_t^2 \wt h - \partial_x^2 \wt h  = r$,
with the given initial data, where $r$ is defined by \eqref{eq:r-def}.
The mapping $(h,\partial_t h) \mapsto (\wt h, \partial_t \wt h)$ is a contraction in $C([T_1, T_2], \cE)$ if $T_2 - T_1$ is sufficiently small.
The fixed point is the solution.

We obtain well-posedness on any compact time interval by dividing it into smaller ones.
Note that if $(h_0, h_1)$ and $f$ are smooth, then so is the solution, which justifies arguments
involving approximating a given finite-energy solution by smooth solutions.
Observe finally that if $f = 0$,
then for any $t$ the linear map $(h_0, h_1) \mapsto (h(t), \partial_t h(t))$ is bounded for the norm $\cE$,
in particular $(h_0, h_1) \wto 0$ implies $(h(t), \partial_t h(t)) \wto 0$. We will use this fact in the Step 5 below.

\textbf{Step 3.} (Energy estimate and uniqueness.)
Assume $(h, \lambda_1, \lambda_2)$ solves \eqref{eq:phi4-lind}--\eqref{eq:phi4-lind-orth} on $[T_0, \infty)$ and $(h, \partial_t h) \in N_\gamma(\cE)$.
We will prove that
\begin{equation}
\label{eq:h-estim-ref}
\|(h, \partial_t h)\|_{N_{\gamma}(\cE)} \leq C \|f\|_{N_{\gamma+1}(L^2)}.
\end{equation}
In particular, this proves uniqueness.

Like in the previous section, we set
\begin{equation}
\chi_1(t, x) := \chi\Big(\frac{x - x_1(t)}{x_2(t) - x_1(t)}\Big),\quad \chi_2 := 1 - \chi_1.
\end{equation}
We introduce a \emph{modified energy functional}
\begin{equation}
\begin{aligned}
I(t) := \int_\bR \Big(\frac 12 (\partial_t h(t))^2 + \frac 12 (\partial_x h(t))^2 + \frac 12 U''(1 - H_1(t) + H_2(t))h(t)^2 \\
-\sum_{j=1}^2 x_j'(t)\chi_j(t)\partial_t h(t)\,\partial_x h(t)\Big) \ud x.
\end{aligned}
\end{equation}
The last term is sometimes called a \emph{correction term}, because its size is negligible
as compared to the other terms. However, we will see that, once we take the time derivative, this term is not negligible anymore.
Denote $J_j(t) := \int_\bR \chi_j(t)\,\partial_t h(t)\,\partial_x h(t)\ud x$.

By coercivity, see Lemma~\ref{lem:D2H}, and \eqref{eq:phi4-lind-orth}, we have
\begin{equation}
\label{eq:I-coercive}
I(t) \gtrsim \|(h(t), \partial_t h(t))\|_{\cE}^2.
\end{equation}
After standard cancellations we obtain
\begin{equation}
\label{eq:dtI-1}
\begin{aligned}
I'(t) &= \big\la \partial_t h(t), f(t) +\lambda_1(t) \partial_x H_1(t) + \lambda_2(t) \partial_x H_2(t)\big\ra \\
&+ \frac 12 \int_\bR U'''(1-H_1(t) + H_2(t))(x_1'(t)\partial_x H_1(t) - x_2'(t)\partial_x H_2(t))h(t)^2\ud x \\
&- x_1''(t)J_1(t) - x_1'(t)J_1'(t)- x_2''(t)J_2(t) - x_2'(t)J_2'(t).
\end{aligned}
\end{equation}

Observe that \eqref{eq:deriv-orth-1} yields $\la \partial_x H_j, \partial_t h\ra = x_j'\la \partial_x^2 H_j, h\ra$,
so~\eqref{eq:x1-x2-cond-2} and \eqref{eq:l1-l2-bound} yield
\EQ{
\label{eq:first-line}
|\la \partial_t h, \lambda_1 \partial_x H_1 + \lambda_2 \partial_x H_2\ra| & \lesssim  {t^{-1} \big(\|f\|_{L^2} + t^{-1} \|(h, \partial_t h)\|_\cE\big) \| (h, \p_t h)\|_{\E} } \\
& \lesssim {t^{-1} \| f \|_{L^2}\|(h, \partial_t h)\|_{\cE} }+   t^{-2}\|(h, \partial_t h)\|_{\cE}^2. 
}
The first and third term of the last line of~{\eqref{eq:dtI-1}} are negligible. The second and fourth term are not, and we will see that they
cancel (up to negligible terms) the second line above.

We compute $ J_1'(t)$. The symbol ``$\simeq$'' means ``up to terms $\leq c \|(h, \partial_t h)\|_{\cE}^2$
for an arbitrarily small constant $c > 0$''. We have
\begin{equation}
\label{eq:dtJ1-0}
J_1' = -x_1'\int_{\bR}\partial_x \chi_1\,\partial_t h\,\partial_x h\ud x + \int_\bR \chi_1\,\partial_t^2 h\,\partial_x h\ud x
+ \int_{\bR}\chi_1\,\partial_t h\,\partial_{t}\partial_x h\ud x.
\end{equation}
Since $|x_1'| \lesssim t^{-1}$ and $|\partial_x \chi_1| \ll 1$, the first term is negligible.
The third term is also negligible,
since $\int_\bR \chi_1 \partial_x(\partial_t h)^2\ud x = -\int_{\bR}(\partial_x\chi_1)(\partial_t h)^2\ud x$.
We compute the second term using \eqref{eq:phi4-lind}:
\begin{equation}
\begin{aligned}
\int_\bR \chi_1 \partial_t^2 h \partial_x h\ud x = \int_\bR \chi_1\big(\partial_x^2 h - U''(1 - H_1 + H_2)h + f + \lambda_1 \partial_x H_1 + \lambda_2\partial_x H_2\big)\partial_x h\ud x.
\end{aligned}
\end{equation}
We have $\int_\bR \chi_1 \partial_x^2 h\partial_x h\ud x = -\frac 12 \int_{\bR}\partial_x\chi_1(\partial_x h)^2\ud x$,
which is negligible. Using \eqref{eq:l1-l2-bound},
\begin{equation}
\bigg|\int_\bR \chi_1 \big(f + \lambda_1 \partial_x H_1 + \lambda_2\partial_x H_2\big)\partial_x h\ud x\bigg| \leq C\|f\|_{L^2}\|h\|_{H^1} + c\|h\|_{H^1}^2.
\end{equation}
Finally,
\begin{equation}
\begin{aligned}
-\int_\bR \chi_1 U''(1-H_1 + H_2)h\,\partial_x h\ud x &= \frac 12 \int_\bR \chi_1 ({-}\partial_x H_1 + \partial_x H_2)U'''(1 - H_1 + H_2)h^2\ud x \\
&+ \frac 12 \int_\bR \partial_x \chi_1 U''(1-H_1 + H_2)h^2\ud x.
\end{aligned}
\end{equation}
The second term is negligible. Since $\partial_x H_2$ is small on the support of $\chi_1$, we conclude that
\begin{equation}
\bigg|J_1' + \frac 12 \int_\bR \chi_1\partial_x H_1 U'''(1-H_1 + H_2)h^2\ud x\bigg| \leq C\|f\|_{L^2}\|h\|_{H^1} + c\|h\|_{H^1}^2.
\end{equation}
In a similar way, 
\begin{equation}
\bigg|J_2' - \frac 12 \int_\bR {\chi_2}\partial_x H_2 U'''(1-H_1 + H_2)h^2\ud x\bigg| \leq C\|f\|_{L^2}\|h\|_{H^1} + c\|h\|_{H^1}^2.
\end{equation}
Combining these estimates with \eqref{eq:dtI-1}, we obtain
\begin{equation}
\label{eq:dtI-2}
\big|I' - \big\la \partial_t h, f +\lambda_1 \partial_x H_1 + \lambda_2 \partial_x H_2\big\ra\big| \leq Ct^{-1}\|f\|_{L^2}\|(h, \partial_t h)\|_{\cE} + ct^{-1}\|(h, \partial_t h)\|_{\cE}^2.
\end{equation}
In particular, using \eqref{eq:first-line},
\begin{equation}
\label{eq:dtI-4}
\big|I'\big| \leq C\|f\|_{L^2}\|(h, \partial_t h)\|_{\cE} + ct^{-1}\|(h, \partial_t h)\|_{\cE}^2.
\end{equation}
Integrating in time and using \eqref{eq:I-coercive}, we obtain
\begin{equation}
\label{eq:en-est-h}
\|(h(t), \partial_t h(t))\|_{\cE}^2 \leq C\int_t^\infty \|(h, \partial_t h)\|_{\cE}\|f\|_{L^2}\ud s
+ c\int_t^\infty s^{-1}\|(h, \partial_t h)\|_{\cE}^2\ud s,
\end{equation}
where $c$ can be made as small as we wish {choosing $T_0$ large}.
Invoking the definition of the norm $N_\gamma$, we get
\begin{equation}
\|(h(t), \partial_t h(t))\|_{\cE}^2 \leq C\|(h, \partial_t h)\|_{{N}_{\gamma}(\cE)}\|f\|_{N_{\gamma+1}(L^2)}\int_t^\infty s^{-\gamma}s^{-1-\gamma}\ud s
+ c\|(h, \partial_t h)\|_{N_\gamma(\cE)}^2 \int_t^\infty s^{-1}s^{-2\gamma}\ud s,
\end{equation}
thus $\|(h, \partial_t h)\|_{N_\gamma(\cE)} \lesssim \|f\|_{N_{\gamma+1}(L^2)}$, which is the required bound for the first term in \eqref{eq:phi4-step1-lin}.

\textbf{Step 4.} (Refined estimate of Lagrange multipliers.)
Regarding the second term in \eqref{eq:phi4-step1-lin}, it is clear from \eqref{eq:l1-l2-bound}
and the bound on $\|(h, \partial_t h)\|_{N_\gamma(\cE)}$ which we just proved that
\begin{equation}
\|\lambda_j\|_{N_{\gamma+1}} \lesssim \|f\|_{N_{\gamma + 1}(L^2)} \quad \Rightarrow\quad \big\|\lambda_j + \|\partial_x H\|_{L^2}^{-2}\la \partial_x H_j, f\ra \big\|_{N_{\gamma+1}} \lesssim \|f\|_{N_{\gamma + 1}(L^2)}.
\end{equation}
In order to obtain the bound on the $W_{\alpha, \beta}$ norm for $\alpha \in \{-1, 0, 1, 2\}$,
we multiply \eqref{eq:proj-H1} by $t^\alpha$ and integrate.
Since $|\la \partial_x H_1, \partial_x H_2\ra| \lesssim t^{-1}$, we get
\begin{equation}
\sup_{t \geq T_0} t^{\gamma + 2}|\lambda_2(t)\la \partial_x H_1, \partial_x H_2\ra| \leq C\|f\|_{N_{\gamma + 1}(L^2)}\quad\Rightarrow\quad
\|\lambda_2(t)\la \partial_x H_1, \partial_x H_2\ra\|_{W_{\alpha, \gamma + 1}} \leq C\|f\|_{N_{\gamma + 1}(L^2)}
\end{equation}
Similarly, we have
\begin{equation}
\sup_{t \geq T_0} t^{\beta + 1}|\la (U''(1-H_1 + H_2) - U''(H_1))\partial_x H_1, h\ra| \leq C\|f\|_{N_{\gamma + 1}(L^2)},
\end{equation}
which allows to bound the last term of the first line of \eqref{eq:proj-H1}.
The second and the third term are estimated in an analogous way (even a bound in $W_{\alpha, \gamma + 1}$
would be possible for these terms). The first term needs to be integrated by parts as follows:
\begin{equation}
\begin{aligned}
&\int_t^\tau \la s^\alpha x_1'(s)\partial_x^2 H_1(s), \partial_t h(s)\ra\ud s = \tau^\alpha x_1'(\tau)\la \partial_x^2 H_1(\tau), h(\tau)\ra - t^\alpha x_1'(t)\la \partial_x^2 H_1(t), h(t)\ra
\\ &- \int_t^\tau \la \alpha s^{\alpha - 1}x_1'(s)\partial_x^2 H_1(s) + s^\alpha x_1''(s) \partial_x^2 H_1(s)
- s^\alpha (x_1'(s))^2 \partial_x^3 H_1(s), h(s)\ra\ud s,
\end{aligned}
\end{equation}
and the last integral is estimated using the triangle inequality.

\textbf{Step 5.} (Existence.)
It remains to show that there exists a solution $(h, \lambda_1, \lambda_2)$ such that $(h, \partial_t h) \in N_\gamma(\cE)$.
Fix any sequence $t_n \to \infty$ and denote by 
$(h_n, \p_t h_n)$ the unique solution to~\eqref{eq:phi4-lind} on the time interval $[T_0, t_n]$
with the initial data $(h, \p_t h)(t_n) = (0, 0)$. We set $h(t) = 0$ for $t \geq t_n$.
The proof of \eqref{eq:h-estim-ref} applies almost without changes and we obtain
\begin{equation}
\|(h_n, \partial_t h_n)\|_{N_\gamma(\cE)} \leq C \|f\|_{N_{\gamma+1}(L^2)}.
\end{equation}
Let $(h_0, h_1)$ be a weak limit of $(h_n(T_0), \partial_t h_n(T_0))$ in $\cE$.
Clearly, it satisfies the orthogonality and compatibility conditions.
Let $(h, \partial_t h)$ be the solution for the initial data $(h(T_0), \partial_t h(T_0)) = (h_0, h_1)$.
We see that for all $t \geq T_0$, $(h(t), \partial_t h(t))$ is the weak limit of $(h_n(t), \partial_t h_n(t))$ in $\cE$,
thus by the Fatou property $(h, \p_t h)$ is a solution of \eqref{eq:phi4-lind} on $[T_0, \infty)$ 
belonging to the space $N_{\gamma}(\cE)$.
\end{proof} 


It turns out that if the time derivative of the forcing term decays, then we can substantially
improve the bounds provided by the last lemma.
This was pointed out to us by Y. Martel.
\begin{lemma}
\label{lem:phi4-lin-dt}
For any {$\gamma > 1$} and $\beta \in (2, \gamma + 1)$ there exists
$C = C(\beta, \gamma) > 0$ and $T_0 = T_0(\beta, \gamma)$ such that for all $(x_1, x_2)$
satisfying \eqref{eq:x1-x2-cond-1}, \eqref{eq:x1-x2-cond-2} and \eqref{eq:x1-x2-cond-3},
and for all $f \in N_{\gamma}(L^2)$ such that $\partial_t f \in N_{\gamma + 1}(L^2)$,
the system \eqref{eq:phi4-lind}--\eqref{eq:phi4-lind-orth} has a unique solution $(h, \lambda_1, \lambda_2)$ and
\begin{align}
\label{eq:phi4-step1-lin-dt}
\|(h, \partial_t h)\|_{N_{\gamma}(\cE)} + \sum_{j=1}^2
\big\|\lambda_j + \|\partial_x H\|_{L^2}^{-2}\la \partial_x H_j, f\ra \big\|_{{\mathbf W}_{\beta} \cap N_{\gamma+1}} \leq C \big(\|f\|_{N_{\gamma}(L^2)} + \|\partial_t f\|_{N_{\gamma + 1}(L^2)}\big).
\end{align}
{If $\gamma =1$, the same result holds without the inclusion of the ${\mathbf W}_{\beta}$ norm on the left-hand side of~\eqref{eq:phi4-step1-lin}.} 
\end{lemma}
\begin{proof}
\textbf{Step 1.} (First estimate of Lagrange multipliers.)
As in the proof of the previous lemma, we arrive at \eqref{eq:l1-l2-bound}.

\textbf{Step 2.} (Energy estimate.) We prove the bound on $(h, \partial_t h)$. We consider an energy functional
slightly different than in the proof of the previous lemma:
\begin{equation}
\label{eq:Itil-def}
\wt I(t) := I(t)  {-} \la h(t), f(t) \ra.
\end{equation}
By coercivity $\|(h, \partial_t h)\|_{\cE}^2 \lesssim \wt I + \|h\|_{L^2}\|f\|_{L^2}$.
Observe that
\begin{equation}
\Big|\dd t\la h, f \ra - \la \partial_t h, f \ra \Big| \leq \|h\|_{L^2}\|\partial_t f\|_{L^2}.
\end{equation}
From \eqref{eq:dtI-2} we have
\begin{equation}
\begin{aligned}
\big|\wt I' - \la \partial_t h, \lambda_1\partial_x H_1+\lambda_2\partial_x H_2 \ra\big|
\leq \|h\|_{{L^2}}\|\partial_t f\|_{{L^2}}+ Ct^{-1}\|f\|_{L^2}\|(h, \partial_t h)\|_{\cE} + ct^{-1}\|(h, \partial_t h)\|_{\cE}^2,
\end{aligned}
\end{equation}
so \eqref{eq:first-line} yields
\begin{equation}
\begin{aligned}
\big| \wt I'\big|
\leq \|h\|_{{L^2}}\|\partial_t f\|_{{L^2}}+ Ct^{-1}\|f\|_{L^2}\|(h, \partial_t h)\|_{\cE} + ct^{-1}\|(h, \partial_t h)\|_{\cE}^2,
\end{aligned}
\end{equation}
and we conclude as in Lemma~\ref{lem:phi4-step1-lin}.

\textbf{Step 3.} (Refined estimate of Lagrange multipliers.) This can be done similarly as in the proof of Lemma~\ref{lem:phi4-step1-lin}.
\end{proof}


{In the next lemma, we compare solutions $(h, \lambda_1, \lambda_2)$ of~\eqref{eq:phi4-lind},~\eqref{eq:phi4-lind-orth} as in Lemma~\ref{lem:phi4-step1-lin} associated to different trajectories $(x_1, x_2)$ and different forcing $f$. We first introduce some notation. Given trajectories $(x_1(t), x_2(t))$ satisfying~\eqref{eq:x1-x2-cond-1}, \eqref{eq:x1-x2-cond-2}, \eqref{eq:x1-x2-cond-3}, we define  $y(t) = (y_1(t), y_2(t))$ by 
\EQ{ \label{eq:y-def} 
y_1(t) := x_1(t) + \log(A t), \quad y_2(t) := x_2(t) - \log(At) 
}
where we remind the reader that $A= \sqrt{2} \| \p_x H \|_{L^2}^{-1} \kappa$, {see \eqref{def AAA}} {and ~\eqref{eq:k-def}}. 
}

\begin{lemma}
\label{lem:path-dep}
For any $\nu, \gamma > 1$ and $\beta \in (2, \nu + \gamma)$ there exist $C = C(\gamma, \nu, \beta)$ and $T_0 = T_0(\gamma, \nu, \beta)$ such that the following holds.
Let $(x_1, x_2)$ and $(\sh x_1, \sh x_2)$ be two pairs of trajectories satisfying
\eqref{eq:x1-x2-cond-1}, \eqref{eq:x1-x2-cond-2}, \eqref{eq:x1-x2-cond-3} and $\|\sh y - y\|_{S_\nu} \leq 1$, {for $y, \sh y$ as in~\eqref{eq:y-def}}.
Let $(h, \lambda_1, \lambda_2)$ be the solution of \eqref{eq:phi4-lind}
and $(\sh{h}, \sh{\lambda_1}, \sh{\lambda_2})$ the solution of \eqref{eq:phi4-lind}--\eqref{eq:phi4-lind-orth}
with $(\sh x_1, \sh x_2)$ instead of $(x_1, x_2)$ and $\sh f$ instead of $f$. Then
\begin{equation}
\label{eq:path-dep}
\begin{aligned}
&\big\|\big(\sh \lambda_j + \|\partial_x H\|_{L^2}^{-2}\la \partial_x \sh H_j, \sh f\ra \big) - \big(\lambda_j + \|\partial_x H\|_{L^2}^{-2}\la \partial_x H_j, f\ra \big)\big\|_{N_{\gamma + \nu}\cap {\mathbf W}_\beta} \\
&\qquad+\|(\sh h - h, \partial_t(\sh h - h))\|_{N_{\gamma + \nu - 1}(\cE)} \\ 
&\qquad\leq C\big(\|\sh y - y\|_{S_\nu}\big(\|f\|_{N_{\gamma + 1}(L^2)}+\|\sh f\|_{N_{\gamma + 1}(L^2)}\big) + \|\sh f - f\|_{N_{\gamma + \nu}(L^2)}\big).
\end{aligned}
\end{equation}
\end{lemma}
\begin{proof}
Let
\begin{equation}
\label{eq:nat-h-def}
\nt h = \sh h + a_1 \partial_x H_1 + a_2 \partial_x H_2
\end{equation}
be the projection of $\sh h$ on the subspace orthogonal to $\partial_x H_1$ and $\partial_x H_2$.
The idea is to apply Lemma~\ref{lem:phi4-step1-lin} in order to obtain an estimate of $\nt h - h$.

\noindent
\textbf{Step 1.} We prove that for $j \in \{1, 2\}$
\begin{equation}
\label{eq:a-coeff-bound}
\begin{aligned}
&\big\| \|\partial_x H\|_{L^2}^2 a_j + (\sh y_j - y_j)\la \partial_x^2 H_j, \sh h\ra\big\|_{N_{\gamma + \nu+1}} \\
+&\big\| \|\partial_x H\|_{L^2}^2 a_j' + (\sh y_j - y_j)\la \partial_x^2 H_j, \partial_t \sh h\ra\big\|_{N_{\gamma + \nu+1}} \\
+&\big\| \|\partial_x H\|_{L^2}^2 a_j'' + (\sh y_j - y_j)\la \partial_x^2 H_j, \partial_t^2 \sh h\ra\big\|_{N_{\gamma + \nu+1}} \\
&\quad \lesssim \|(\sh h, \partial_t \sh h)\|_{N_\gamma(\cE)}\|\sh y - y\|_{S_\nu}
\lesssim \|\sh f\|_{N_{\gamma + 1}(L^2)}\|\sh y - y\|_{S_\nu}.
\end{aligned}
\end{equation}
Note that \eqref{eq:a-coeff-bound} in particular implies
\begin{equation}
 \|a_j \|_{N_{\gamma + \nu}} +\| a_j' \|_{N_{\gamma + \nu}} +\|a_j'' \|_{N_{\gamma + \nu}}
  \lesssim \|(\sh h, \partial_t \sh h)\|_{N_\gamma(\cE)}\|\sh y - y\|_{S_\nu}
\lesssim \|\sh f\|_{N_{\gamma + 1}(L^2)}\|\sh y - y\|_{S_\nu}.
\end{equation}

In order to prove \eqref{eq:a-coeff-bound}, we multiply \eqref{eq:nat-h-def} by $\partial_x H_j$ and integrate:
\begin{equation}
\begin{aligned}
a_1\|\partial_x H\|_{L^2}^2 + a_2\la \partial_x H_1, \partial_x H_2\ra + \la \partial_x H_1, \sh h\ra &= 0, \\
a_1 \la \partial_x H_1, \partial_x H_2\ra + a_2\|\partial_x H\|_{L^2}^2 + \la \partial_x H_2, \sh h\ra &= 0.
\end{aligned}
\end{equation}
Observing that $\la \partial_x H_j, \sh h\ra = \la \partial_x H_j - \partial_x \sh H_j, \sh h\ra$, we obtain
\begin{equation}
\label{eq:aj}
\begin{aligned}
a_1\|\partial_x H\|_{L^2}^2 + a_2\la \partial_x H_1, \partial_x H_2\ra &= \la \partial_x \sh H_1- \partial_x H_1, \sh h\ra, \\
a_1 \la \partial_x H_1, \partial_x H_2\ra + a_2\|\partial_x H\|_{L^2}^2  &= \la \partial_x \sh H_2- \partial_x H_2, \sh h\ra.
\end{aligned}
\end{equation}
This is a strictly diagonally dominant linear system for $a_1$ and $a_2$.
Since
\begin{equation}
|\la \partial_x H_j - \partial_x \sh H_j, \sh h\ra| \leq \|\partial_x H_j - \partial_x \sh H_j\|_{L^2}\|\sh h\|_{L^2} \lesssim |\sh y - y|\|\sh h\|_{L^2},
\end{equation}
we obtain $|a_j| \lesssim |\sh y - y|\|\sh h\|_{L^2}$.
Observe that
\begin{equation}
|\la \partial_x \sh H_1 - \partial_x H_1, \sh h\ra + (\sh y_1 - y_1)\la \partial_x^2 H_1, \sh h\ra| \lesssim |\sh y_1 - y_1|^2 \|\sh h\|_{L^2}.
\end{equation}
From the first line in \eqref{eq:aj} we thus get
\begin{equation}
\big|\|\partial_x H\|_{L^2}^2 a_1 + (\sh y_1 - y_1)\la \partial_x^2 H_1, \sh h\ra\big| \lesssim |\sh y_1 - y_1| \|\sh h\|_{L^2}
\big(|\sh y_1 - y_1| + |\la \partial_x H_1, \partial_x H_2\ra|\big),
\end{equation}
which implies the required estimate for $a_1$. The estimate for $a_2$ is obtained analogously.

In order to prove the bound on $|a_j'|$, we differentiate in time \eqref{eq:aj}.
We obtain
\begin{equation}
\label{eq:dtaj}
\begin{aligned}
a_1'\|\partial_x H\|_{L^2}^2 + a_2'\la \partial_x H_1, \partial_x H_2\ra &= \dd t\la \partial_x \sh H_1- \partial_x H_1, \sh h\ra
- a_2 \dd t\la \partial_x H_1, \partial_x H_2\ra, \\
a_1' \la \partial_x H_1, \partial_x H_2\ra + a_2'\|\partial_x H\|_{L^2}^2  &= \dd t\la \partial_x \sh H_2- \partial_x H_2, \sh h\ra
- a_1 \dd t\la \partial_x H_1, \partial_x H_2\ra.
\end{aligned}
\end{equation}
Note that $|\la \partial_t(\partial_x \sh H_1- \partial_x H_1), \sh h\ra| \lesssim t^{-1}|\sh y_1 - y_1|\|\sh h\|_{L^2}$.
With this observation, a similar computation as above leads to the required bounds on $a_1'$ and $a_2'$.

Differentiating again, we obtain
\begin{equation}
\label{eq:dtdtaj}
\begin{aligned}
a_1''\|\partial_x H\|_{L^2}^2 + a_2''\la \partial_x H_1, \partial_x H_2\ra &= \frac{\ud^2}{\vd t^2}\la \partial_x \sh H_1- \partial_x H_1, \sh h\ra
- a_2 \frac{\ud^2}{\vd t^2}\la \partial_x H_1, \partial_x H_2\ra-2a_2' \frac{\ud}{\vd t}\la \partial_x H_1, \partial_x H_2\ra, \\
a_1'' \la \partial_x H_1, \partial_x H_2\ra + a_2''\|\partial_x H\|_{L^2}^2  &= \frac{\ud^2}{\vd t^2}\la \partial_x \sh H_2- \partial_x H_2, \sh h\ra
- a_1 \frac{\ud^2}{\vd t^2}\la \partial_x H_1, \partial_x H_2\ra-2a_1' \frac{\ud}{\vd t}\la \partial_x H_1, \partial_x H_2\ra.
\end{aligned}
\end{equation}
All the terms of the right hand side of the first line are treated similarly as before, except for
$\la \partial_x \sh H_j - \partial_x H_j, \partial_t^2 \sh h\ra$.
For this term, we write
\begin{equation}
\big|\la \partial_x \sh H_j - \partial_x H_j, \partial_t^2 \sh h\ra + (\sh y_j -y_j)\la \partial_x^2 H_j, \partial_t^2\sh h\ra\big|
\lesssim |\sh y_j - y_j|^2\|\partial_t^2 \sh h\|_{H^{-1}}.
\end{equation}
Invoking the equation~\eqref{eq:phi4-lind} satisfied by $\sh h$ and making use
of the bounds~\eqref{eq:l1-l2-bound} for $\sh \lambda_1, \sh \lambda_2$, we obtain
\begin{equation}
\big\|\la \partial_x \sh H_j - \partial_x H_j, \partial_t^2 \sh h\ra + (\sh y_j -y_j)\la \partial_x^2 H_j, \partial_t^2\sh h\ra\big\|_{N_{\gamma + \nu + 1}} \lesssim \|\sh f\|_{N_{\gamma + 1}}\|\sh y - y\|_{S_\nu}.
\end{equation}


\noindent
\textbf{Step 2.}
We write the equation satisfied by $\nt h - h$:
\begin{equation}
\label{eq:nth}
\partial_t^2 (\nt h - h) - \partial_x^2 (\nt h - h) + U''(1 - H_1 + H_2)(\nt h - h) = \nt f + \nt \lambda_1 \partial_x H_1
+ \nt \lambda_2 \partial_x H_2,
\end{equation}
where
\begin{equation}
\label{eq:diff-expansion}
\begin{aligned}
\nt f &:=
(\partial_t^2 \sh h - \partial_x^2 \sh h + U''(1 - \sh H_1 + \sh H_2)\sh h - \sh \lambda_1\partial_x H_1 - \sh \lambda_2\partial_x H_2) \\
& \quad - (\partial_t^2 h - \partial_x^2 h + U''(1 - H_1 + H_2) h - \lambda_1\partial_x H_1 - \lambda_2\partial_x H_2) \\
&\quad + (\partial_t^2 \nt h - \partial_t^2 \sh h - a_1''\partial_x H_1 - a_2''\partial_x H_2) - (\partial_x^2 \nt h - \partial_x^2 \sh h) + U''(1 - \sh H_1 + \sh H_2)(\nt h - \sh h) \\
& \quad + (U''(1 - H_1 + H_2) - U''(1 - \sh H_1 + \sh H_2))\nt h + a_1''\partial_x H_1 + a_2''\partial_x H_2,
\end{aligned}
\end{equation}
{and}
\begin{equation}
\label{eq:ntlambda}
\nt \lambda_j := \sh \lambda_j - \lambda_j.
\end{equation}
According to Lemma~\ref{lem:phi4-step1-lin}, we should
\begin{enumerate}[(i)]
\item estimate $\nt f$ in $N_{\gamma + \nu}(L^2)$,
\item estimate $\la \partial_x H_j, \nt f\ra$ in ${\mathbf W}_\beta$.
\end{enumerate}
%
We expand each line of \eqref{eq:diff-expansion}. Most terms can be estimated in $N_{\gamma + \nu + 1}(L^2)$,
which takes care of (i) and (ii) above, so we will call such terms ``negligible''. The first line equals
\begin{equation}
\sh f + \sh\lambda_1(\partial_x \sh H_1 - \partial_x H_1) + \sh\lambda_2(\partial_x \sh H_2 - \partial_x H_2).
\end{equation}
The second and third terms are negligible since
\begin{equation}
\|\sh\lambda_j(\partial_x \sh H_j - \partial_x H_j)\|_{N_{{\nu}  + \gamma + 1}} \lesssim \|\sh \lambda_j\|_{N_{\gamma + 1}}\|\sh y_j - y_j\|_{S_\nu}
\lesssim \|\sh y_j - y_j\|_{S_\nu}\|\sh f\|_{N_{\gamma + 1}(L^2)}.
\end{equation}
The second line of \eqref{eq:diff-expansion} equals $-f$.
Now we expand the third line.
We have
\begin{equation}
\begin{aligned}
\partial_t \nt h& = \partial_t \sh h + \sum_{j=1}^2 \big(a_j'\partial_x H_j - a_j x_j'\partial_x^2 H_j\big), \\
\partial_t^2 \nt h &= \partial_t^2 \sh h + \sum_{j=1}^2 \big(a_j''\partial_x H_j - 2a_j'x_j'\partial_x^2 H_j - a_jx_j''\partial_x^2 H_j + a_j(x_j')^2 \partial_x^3 H_j\big).
\end{aligned}
\end{equation}
Thus, it follows from \eqref{eq:a-coeff-bound} that the term $\partial_t^2 \nt h - \partial_t^2 \sh h - a_1''\partial_x H_1 - a_2''\partial_x H_2$ is negligible.
The remaining part of the {third} line equals
\begin{equation}
\begin{aligned}
\sum_{j=1}^2 a_j\big({-}\partial_x^2 + U''(1-\sh H_1+\sh H_2)\big)\partial_x H_j &= \sum_{j=1}^2 a_j\big(U''(1-\sh H_1+\sh H_2) - U''(H_j)\big)\partial_x H_j \\
&= \sum_{j=1}^2 a_j\big(U''(1-\sh H_1+\sh H_2) - U''(1- H_1 + H_2)\big)\partial_x H_j \\
&+\sum_{j=1}^2 a_j\big(U''(1-H_1+H_2) - U''(H_j)\big)\partial_x H_j.
\end{aligned}
\end{equation}
The last line is negligible by \eqref{eq:inter-bound-2} and \eqref{eq:a-coeff-bound}.
The other line can be estimated in $N_{\gamma + 2\nu}$, hence is negligible as well.

Finally, consider the fourth line of \eqref{eq:diff-expansion}. We have
\begin{equation}
\begin{aligned}
\|(U''(1 - H_1 + H_2) - U''(1 - \sh H_1 + \sh H_2))\nt h\|_{N_{\gamma + \nu}(L^2)} &\lesssim
\|\sh y - y\|_{S_{\nu}}(\|\sh h\|_{N_\gamma(L^2)} + \|a_1\|_{N_\gamma} + \|a_2\|_{N_\gamma}) \\
&\lesssim \|\sh y - y\|_{S_{\nu}}\|\sh f\|_{N_{\gamma + 1}(L^2)}.
\end{aligned}
\end{equation}
However, it turns out that this term is not negligible, and we have to estimate carefully
its projection on $\partial_x H_j$ {for this term's contribution to (ii)}. In what follows, ``$\simeq$'' means ``up to terms bounded in $N_{\gamma + \nu + 1}(L^2)$
by the right hand side of \eqref{eq:path-dep}''. By \eqref{eq:U''-taylor} we have
\begin{equation}
\begin{aligned}
\|U''(1 - H_1 + H_2) - U''(1 - \sh H_1 + \sh H_2) - ((\sh H_1 - H_1) - (\sh H_2 - H_2))U'''(1 - H_1 + H_2)\|_{N_{2\nu}(L^\infty)} \\
\lesssim \|\sh y - y\|_{S_\nu}^2 \lesssim \|\sh y - y\|_{S_\nu}.
\end{aligned}
\end{equation}
We also have, using Taylor expansions,
\begin{equation}
\|\sh H_j - H_j + (\sh y_j - y_j)\partial_x H_j\|_{N_{2\nu}(L^\infty)} \lesssim \|\sh y - y\|_{S_\nu}^2 \lesssim \|\sh y - y\|_{S_\nu}.
\end{equation}
From these two inequalities and $\|\nt h - \sh h\|_{N_{\gamma + \nu}(L^2)} \lesssim \|\sh y - y\|_{S_\nu}\|\sh f\|_{N_{ \gamma {+1}} (L^2)}$
we deduce that, up to negligible terms, the fourth line of \eqref{eq:diff-expansion} equals
\begin{equation}
\begin{aligned}
\sh h U'''(1 - H_1 + H_2)(-(\sh y_1 - y_1)\partial_x H_1 + (\sh y_2 - y_2)\partial_x H_2) \simeq \\
\simeq \sh h\big((\sh y_1 - y_1)U'''(H_1)\partial_x H_1 + (\sh y_2 - y_2)U'''(H_2)\partial_x H_2\big),
\end{aligned}
\end{equation}
where the last approximate equality follows from the fact that $U'''$ is locally Lipschitz, thus for instance
\begin{equation}
\label{eq:replace-potential}
|(U'''(1 - H_1 + H_2) + U'''(H_1))\partial_x H_1| \lesssim |1 + H_2||\partial_x H_1| \ll t^{-1}.
\end{equation}
Up to negligible terms, the scalar product with $\partial_x H_1$ is
\begin{equation}
\int_{\bR} \sh h(\sh y_1 - y_1)U'''(H_1)(\partial_x H_1)^2\ud x = (\sh y_1 - y_1)\la \partial_x^2 H_1, (\partial_x^2 - U''(H_1))\sh h\ra,
\end{equation}
where for the last equality we use \eqref{eq:U'''-identity}. By an estimate analogous to \eqref{eq:replace-potential} but
with $\partial_x^2 H_1$ instead of $\partial_x H_1$ and $U''$ instead of $U'''$, the right-hand side is approximately equal to
\begin{equation}
\begin{aligned}
&(\sh y_1 - y_1)\la (\partial_x^2 - U''(1 - H_1 + H_2))\sh h, \partial_x^2 H_1\ra = \\
&=(\sh y_1 - y_1)\la \partial_x^2 H_1, \partial_t^2 \sh h - \sh f - \sh \lambda_1 \partial_x \sh H_1 - \sh \lambda_2 \partial_x \sh H_2\ra.
\end{aligned}
\end{equation}
We observe that, by \eqref{eq:a-coeff-bound}, the term $(\sh y_1 - y_1)\la \partial_x^2 H_1,\partial_t^2 \sh h\ra$ is cancelled,
up to negligible terms, by the scalar product of the term $a_1'' \partial_x H_1$ with $\partial_x H_1$.
The term $(\sh y_1 - y_1)\la \partial_x^2 H_1, \sh \lambda_1 \partial_x \sh H_1\ra$ is negligible due to ${\la}\partial_x^2 H, \partial_x H\ra = 0$. The term $(\sh y_1 - y_1)\la \partial_x^2 H_1, \sh \lambda_2 \partial_x \sh H_2\ra$ is negligible as well,
and so is the term $(\sh y_1 - y_1)\la \partial_x^2 H_1, \sh f\ra$.
In conclusion,
\begin{equation}
\la \partial_x H_1, \nt f\ra \simeq \la \partial_x H_1, \sh f - f\ra
\simeq \la \partial_x \sh H_1, \sh f\ra - \la \partial_x H_1, f\ra,
\end{equation}
and similarly for the projection on $\partial_x H_2$.
\end{proof}
\begin{lemma}
\label{lem:path-dep-dt}
For any $\gamma, \nu > 1$ and $\beta \in (2, \gamma + \nu)$
there exist $C = C(\gamma, \nu, \beta)$ and $T_0 = T_0(\gamma, \nu, \beta)$ such that the following holds.
Let $(x_1, x_2)$ and $(\sh x_1, \sh x_2)$ be two pairs of trajectories satisfying
\eqref{eq:x1-x2-cond-1}, \eqref{eq:x1-x2-cond-2} and \eqref{eq:x1-x2-cond-3}.
Let $(h, \lambda_1, \lambda_2)$ be the solution of \eqref{eq:phi4-lind}
and $(\sh{h}, \sh{\lambda_1}, \sh{\lambda_2})$ the solution of \eqref{eq:phi4-lind}--\eqref{eq:phi4-lind-orth}
with $(\sh x_1, \sh x_2)$ instead of $(x_1, x_2)$ and $\sh f$ instead of $f$. Then
\begin{align}
&\big\|\big(\sh \lambda_j + \|\partial_x H\|_{L^2}^{-2}\la \partial_x \sh H_j, \sh f\ra \big) - \big(\lambda_j + \|\partial_x H\|_{L^2}^{-2}\la \partial_x H_j, f\ra \big)\big\|_{N_{\gamma + \nu}\cap {\mathbf W}_{\beta}}\\
&\ + \|(\sh h - h, \partial_t(\sh h - h))\|_{N_{\gamma + \nu - 1}(\cE)}\\
&\ \leq C\big(\|\sh y - y\|_{S_\nu}\big(\|f\|_{N_{\gamma}(L^2)} + \|\partial_t f\|_{N_{\gamma + 1}(L^2)}
+\|\sh f\|_{N_{\gamma}(L^2)} + \|\partial_t \sh f\|_{N_{\gamma + 1}(L^2)}\big)
+\|\sh f - f\|_{N_{\gamma + \nu}(L^2)}\big).
\end{align}
\end{lemma}
\begin{proof}
The proof is similar to the previous one, but using Lemma~\ref{lem:phi4-lin-dt} instead of Lemma~\ref{lem:phi4-step1-lin}.
\end{proof}

\subsection{{Solving~\eqref{eq:phi4-step1} for given trajectories $(x_1, x_2)$}}  \label{s:contraction1} 
Let $\gamma \geq 1$. Given a pair of trajectories $(x_1, x_2)$ and $(g, \partial_t g) \in N_\gamma(\cE)$, we define
\begin{equation}
(\Lambda_1(x_1, x_2, g), \Lambda_{{2}}(x_1, x_2, g), \Psi(x_1, x_2, g)) = (\lambda_1, \lambda_2, h)
\end{equation}
as the solution of the equation
\begin{equation}
\label{eq:Phi_g}
\begin{aligned}
&\partial_t^2 h - \partial_x^2 h + U''(1 - H_1 + H_2)h \\
&= (\lambda_1-x_1'') \partial_x H_1 +(x_1')^2\partial_x^2 H_1+ (\lambda_2+x_2'') \partial_x H_2 - (x_2')^2\partial_x^2 H_2 \\
&- U'(1-H_1 + H_2) - U'(H_1) + U'(H_2) \\
&- U'(1-H_1 + H_2 + g) + U'(1-H_1 + H_2) + U''(1 - H_1 + H_2)g
\end{aligned}
\end{equation}
satisfying the orthogonality conditions $\la \partial_x H_1, h\ra = \la \partial_x H_2, h\ra = 0$.
Note that we do not require the argument $g$ to satisfy any orthogonality conditions.

\begin{proposition}
\label{prop:prop-mapping}
The mapping $(\Lambda_1, \Lambda_2, \Psi)$ has the following properties.
\begin{enumerate}[(i)]
\item
For any {$\gamma \in (1, 2)$} and $\beta \in (2, \gamma + 1)$ there exist
$C_{{1}} = C_{{1}}(\beta, \gamma) > 0$ and $T_0 = T_0(\beta, \gamma)$ such that for all $(x_1, x_2)$
satisfying \eqref{eq:x1-x2-cond-1}, \eqref{eq:x1-x2-cond-2} and \eqref{eq:x1-x2-cond-3}
\begin{equation}
\label{eq:prop-mapping-1}
\begin{aligned}
\|\Lambda_1(x_1, x_2, 0) - x_1'' + F(x_2 - x_1)\|_{{\mathbf W}_\beta \cap N_{\gamma + 1}} + 
\|\Lambda_2(x_1, x_2, 0) + x_2'' + F(x_2 - x_1)\|_{{\mathbf W}_\beta \cap N_{\gamma + 1}} \\
+ \|(\Psi(x_1, x_2, 0), \partial_t \Psi(x_1, x_2, 0))\|_{N_\gamma(\cE)} \leq C_1, 
\end{aligned}
\end{equation}
{where $F$ is the normalized attraction force defined in~\eqref{eq:force-def}. 
For $\gamma =1$ the same conclusion holds without the inclusion of the ${\mathbf W}_{\beta}$ bound.}
\item For any $\gamma_1, \gamma_2 \geq 1$ there exist $C = C(\gamma_1, \gamma_2) > 0$
and $T_0 = T_0(\gamma_1, \gamma_2)$ such that for all $(x_1, x_2)$
satisfying \eqref{eq:x1-x2-cond-1}, \eqref{eq:x1-x2-cond-2} and \eqref{eq:x1-x2-cond-3} {and all $g,  \sh g \in N_{\gamma_1}(L^2)  \cap N_{\gamma_2}(L^2)$}
\begin{equation}
\label{eq:prop-mapping-2}
\begin{aligned}
&\|\Lambda_1(x_1, x_2, \sh g) - \Lambda_1(x_1, x_2, g)\|_{N_{\gamma_1 + \gamma_2}}
+\|\Lambda_2(x_1, x_2, \sh g) - \Lambda_2(x_1, x_2, g)\|_{N_{\gamma_1 + \gamma_2}} \\
&\qquad+ \|(\Psi(x_1, x_2, \sh g) - \Psi(x_1, x_2, g), \partial_t(\Psi(x_1, x_2, \sh g) - \Psi(x_1, x_2, g)))\|_{N_{\gamma_1 + \gamma_2 - 1}(\cE)} \\
&\qquad \leq C\big(\|\sh g\|_{N_{\gamma_1}(L^2)} + \|g\|_{N_{\gamma_1}(L^2)}\big)\|\sh g - g\|_{N_{\gamma_2}(L^2)}.
\end{aligned}
\end{equation}
\item For any $\gamma \in (1, 2)$, $\nu > 1$ and $\beta \in (2, \min(\nu + 2, \nu + 2\gamma - 1))$
there exist $C = C(\gamma, \nu, \beta) > 0$ and $T_0 = T_0(\gamma, \nu, \beta)$ such that for all $(x_1, x_2)$ and $(\sh x_1, \sh x_2)$
satisfying \eqref{eq:x1-x2-cond-1}, \eqref{eq:x1-x2-cond-2} and \eqref{eq:x1-x2-cond-3}, and all $g$ 
such that $\|(g, \partial_t g)\|_{N_\gamma(\cE)} \leq 1$
\begin{equation}
\label{eq:prop-mapping-3}
\begin{aligned}
&\|(\Lambda_1(\sh x_1, \sh x_2, g)- \Lambda_1(x_1, x_2, g)) - ((\sh x_1)'' - x_1'') + (F(\sh x_2 - \sh x_1) - F(x_2 - x_1))\|_{N_{\beta}\cap {\mathbf W}_\beta} \\
+& \|(\Lambda_2(\sh x_1, \sh x_2, g)- \Lambda_2(x_1, x_2, g)) + ((\sh x_2)'' - x_2'') + (F(\sh x_2 - \sh x_1) - F(x_2 - x_1))\|_{N_{\beta}\cap {\mathbf W}_\beta} \\
+& \|(\Psi(\sh x_1, \sh x_2, g) - \Psi(x_1, x_2, g), \partial_t (\Psi(\sh x_1, \sh x_2, g) - \Psi(x_1, x_2, g)))\|_{N_{\beta - 1}(\cE)} \\
&\qquad \leq C\|(\sh x_1, \sh x_2) - (x_1, x_2)\|_{S_\nu}.
\end{aligned}
\end{equation}
\end{enumerate}
\end{proposition}
\begin{proof}
If $g = 0$, then the last line in \eqref{eq:Phi_g} vanishes.  Note also that the second line of the right hand side of \eqref{eq:Phi_g} equals $\Phi(x_1, x_2, \cdot)$, {which is defined in~\eqref{eq:Phi-def}. Note that in all estimates in Section~\ref{s:lin} we are free to replace $\lambda_j \p_x H_j$ on the right-hand side of~\eqref{eq:phi4-lind} with $(\lambda_j  + (-1)^j x_j'') \p_x H_j$.} {Thus, in the context of Lemma~\ref{lem:phi4-lin-dt}, we can take forcing term $f$ in \eqref{eq:Phi_g} with $g =0$ to be 
\EQ{
f &= (x_1')^2\partial_x^2 H_1- (x_2')^2\partial_x^2 H_2  +\Phi(x_1, x_2, \cdot). 
}
Hence, after noting that $\ang{\p_x^2 H_j, \, \p_x H_j} = 0$, we see that up to negligible terms we have 
\EQ{
\lambda_j + (-1)^j x_j'' +  \| \p_x H \|_{L^2}^{-2} \ang{ \p_x H_j, \, f} =   \lambda_j + (-1)^j x_j'' + F(x_2-x_1)
} 
with $F$ defined in~\eqref{eq:force-def}.}
In order to apply Lemma~\ref{lem:phi4-lin-dt},
we need to bound the terms in $f$ in  $N_\gamma$ and their time derivatives in $N_{\gamma + 1}$ for all $\gamma < 2$.
It is clear that $\big\|\partial_t\big((x_j')^2\partial_x^2 H_j\big)\big\|_{L^2}  \lesssim t^{-3}$.
The second line of the right hand side of \eqref{eq:Phi_g} equals $\Phi(x_1, x_2, \cdot)$.
The Chain Rule and Lemma~\ref{lem:inter-bound} yield
\begin{equation}
\|\partial_t \Phi(x_1(t), x_2(t), \cdot)\|_{L^2} \leq \sum_{j=1}^2|x_j'(t)|\|\partial_{x_j}\Phi(x_1(t), x_2(t), \cdot)\|_{L^2}
\lesssim t^{-\gamma - 1},\quad\forall \gamma < 2,
\end{equation}
thus we have proved \eqref{eq:prop-mapping-1}.

In order to prove \eqref{eq:prop-mapping-2},
we observe that $\lambda_j := \Lambda_j(x_1, x_2, \sh g) - \Lambda_j(x_1, x_2, g)$
and $h := \Psi(x_1, x_2, \sh g) - \Psi(x_1, x_2, g)$ solve the equation
\begin{equation}
\begin{aligned}
&\partial_t^2 h - \partial_x^2 h + U''(1 - H_1 + H_2)h = \lambda_1 \partial_x H_1 + \lambda_2 \partial_x H_2 \\
&- U'(1-H_1 + H_2 + \sh g) + U'(1-H_1 + H_2) + U''(1 - H_1 + H_2)\sh g \\
&+ U'(1 - H_1 + H_2 + g) - U'(1 - H_1 + H_2) - U''(1 - H_1 + H_2)g
\end{aligned}
\end{equation}
and $\la \partial_x H_j, h\ra = 0$.
The second and third line constitute the forcing term. By \eqref{eq:U'-taylor-diff},
its $L^2$ norm is bounded up to a constant by $(\|\sh g\|_{H^1} + \|g\|_{H^1})\|\sh g - g\|_{H^1}$.
Applying Lemma~\ref{lem:phi4-step1-lin}, we get \eqref{eq:prop-mapping-2}.

We are left with \eqref{eq:prop-mapping-3}. Let
\begin{align}
f &:= (x_1')^2\partial_x^2 H_1 - (x_2')^2\partial_x^2 H_2
- U'(1-H_1 + H_2) + U'(H_1) - U'(H_2), \\
\wt f &:= - U'(1-H_1 + H_2 + g) + U'(1-H_1 + H_2) + U''(1 - H_1 + H_2)g, \\
\sh f &:= ((\sh x_1)')^2\partial_x^2 \sh H_1 - ((\sh x_2)')^2\partial_x^2 \sh H_2
- U'(1-\sh H_1 + \sh H_2) + U'(\sh H_1) - U'(\sh H_2), \\
\sh{\wt f} &:= - U'(1-\sh H_1 + \sh H_2 + g) + U'(1-\sh H_1 + \sh H_2) + U''(1 - \sh H_1 + \sh H_2)g.
\end{align}
Let $(h, \lambda_1, \lambda_2)$ solve \eqref{eq:phi4-lind}--\eqref{eq:phi4-lind-orth},
$(\wt h, \wt \lambda_1, \wt \lambda_2)$ solve \eqref{eq:phi4-lind}--\eqref{eq:phi4-lind-orth} with $\wt f$ instead of $f$,
$(\sh h, \sh \lambda_1, \sh \lambda_2)$ solve \eqref{eq:phi4-lind}--\eqref{eq:phi4-lind-orth} with $\sh f$ instead of $f$
and $\sh H_j$ instead of $H_j$, and $(\sh {\wt h}, \sh {\wt\lambda_1}, \sh {\wt\lambda_2})$
solve \eqref{eq:phi4-lind}--\eqref{eq:phi4-lind-orth} with $\sh {\wt f}$ instead of $f$ and $\sh H_j$ instead of $H_j$. Then
\begin{align}
\Lambda_j(x_1, x_2, g) &= (-1)^{j+1} x_j'' + \lambda_j + \wt \lambda_j, \\
\Lambda_j(\sh x_1, \sh x_1, g) &= (-1)^{j+1} (\sh x_j)'' + \sh\lambda_j + \sh{\wt \lambda_j}, \\
\Psi(x_1, x_2, g) &= h + \wt h, \\
\Psi(\sh{x_1}, \sh{x_2}, g) &= \sh h + \sh{\wt h}.
\end{align}
In order to estimate $\sh h - h$ and $\sh \lambda_j - \lambda_j$, take $\gamma \in (\max(1, \beta - \nu), 2)$ in Lemma~\ref{lem:path-dep-dt}.
We have already seen while proving \eqref{eq:prop-mapping-1} that
\begin{equation}
\|f\|_{N_\gamma(L^2)} + \|\partial_t f\|_{N_{\gamma + 1}(L^2)} + \|\sh f\|_{N_\gamma(L^2)} + \|\partial_t \sh f\|_{N_{\gamma + 1}(L^2)} \lesssim 1,
\end{equation}
so we only need to show that $\|\sh f - f\|_{N_{\gamma + \nu}(L^2)} \lesssim \|\sh x - x\|_{S_\nu}$.
Estimating $((\sh x_j)')^2\partial_x^2 \sh H_j - (x_j')^2\partial_x^2 H_j$ in $N_{\nu + 2}$ 
is straightforward, and the remaining part is estimated using \eqref{eq:inter-bound-1}.
The projections of $f$ and $\sh f$ on $\partial_x H_j$ and $\partial_x \sh H_j$
yield the terms $F(x_2 - x_1)$ and $F(\sh x_2 - \sh x_1)$ in \eqref{eq:prop-mapping-3}.

Concerning the estimates of $\sh{\wt h} - \wt h$ and $\sh{\wt\lambda_j} - \wt \lambda_j$,
we have, with the same choice of $\gamma$,
\begin{align}
\|\wt f\|_{N_{\gamma+1}(L^2)} + \|\wt{\sh f}\|_{N_{\gamma + 1}(L^2)} &\lesssim 1, \\
\|\sh {\wt f} - \wt f\|_{N_{\gamma + \nu + 1}(L^2)} &\lesssim \|\sh x - x\|_{S_\nu},
\end{align}
where the last estimate follows from the point-wise bound
\begin{equation}
|\sh {\wt f} - \wt f| \lesssim (|\sh x_1 - x_1| + |\sh x_2 - x_2|)|g|^2,
\end{equation}
see \eqref{eq:U'-taylor-diff}.
We obtain the conclusion using Lemma~\ref{lem:path-dep} and noting again that $N_{\gamma + \nu + 1} \subset \mathbf W_\beta$.
\end{proof}
\begin{remark}
Since the distance between the kink and the antikink is $2\log t + O(1)$,
the forcing term coming from the interaction of the kink and the anti-kink
is of size $t^{-2}$, hence too large to be handled directly by Lemma~\ref{lem:phi4-step1-lin}; {see~\eqref{eq:inter-bound-1} and Lemma~\ref{lem:exp-cross-term} for this computation}.
The same remark applies to the term $x_k'(t)^2 \partial_x^2 H_k$.
However, taking the time derivative of these forcing terms, we gain one power of $t$,
and we can use Lemma~\ref{lem:phi4-lin-dt}, which is what happens in the proof above.
\end{remark}
\begin{proposition}
\label{prop:lyap-schmidt}
For any $C_0 > 0$ there exist $T_0 > 0$ and $\delta > 0$ such that the following is true.
For any $x_1, x_2: [T_0, \infty) \to \bR$ satisfying \eqref{eq:x1-x2-cond-1}, \eqref{eq:x1-x2-cond-2} and \eqref{eq:x1-x2-cond-3},
the equation \eqref{eq:phi4-step1} has a unique solution $(\lambda_1, \lambda_2, g) = \big(\lambda_1(x_1, x_2), \lambda_2(x_1, x_2), g(x_1, x_2)\big)$ such that $\|(g, \partial_t g)\|_{N_1(\cE)} \leq \delta$.
For all $\gamma \in [1, 2)$ there exist $C = C(\gamma)$ and $T_0 = T_0(\gamma)$ such that this solution satisfies
\begin{align}
\sum_{j=1}^2\|\lambda_j +(-1)^j x_j'' + F(x_2 - x_1)\|_{N_{\gamma + 1}\cap {\mathbf W_{\gamma + 1}}} + \|(g, \partial_t g)\|_{N_{\gamma}(\cE)} \leq 1. \label{eq:g-l1-l2-est}
\end{align}
Moreover, for all $\nu > 1$ and $\beta \in (2, \nu + 2)$ there exist
$C = C(\nu, \beta) > 0$ and $T_0 = T_0(\nu, \beta)$ such that
\begin{equation}
\label{eq:lyap-schmidt-lip-l}
\begin{aligned}
\big\|\lambda_j(\sh x_1, \sh x_2) -\lambda_j(x_1, x_2) + (-1)^j\big((\sh x_j)'' - x_j''\big)
+\big(F(\sh x_2 - \sh x_1) - F(x_2 - x_1)\big)\big\|_{N_{\beta}\cap {\mathbf W_{\beta}}} \\
+\|(g(\sh x_1, \sh x_2) - g(x_1, x_2), \partial_t (g(\sh x_1, \sh x_2) - g(x_1, x_2)))\|_{N_{\beta - 1}(\cE)} \leq C\|(\sh x_1, \sh x_2) - (x_1, x_2)\|_{S_\nu}.
\end{aligned}
\end{equation}
{where $(x_1, x_2)$ and $(\sh x_1, \sh x_2)$ are any two pairs of admissible trajectories satisfying $x - \sh x \in S_{\nu}$.} 
\end{proposition}
\begin{proof}
We notice that for given $(x_1, x_2)$, $g$ solves \eqref{eq:phi4-step1} if and only if $g$ is a fixed point
of the mapping $\Psi(x_1, x_2, \cdot)$. By \eqref{eq:prop-mapping-2}, this mapping is a contraction
in a sufficiently small ball of $N_1(\cE)$. Also, \eqref{eq:prop-mapping-1} implies that this ball is invariant,
hence by the Contraction Principle there exists a unique fixed point $g = g(x_1, x_2)$.
By a similar argument, there exists a unique fixed point in the unit ball of $N_\gamma(\cE)$ for any $\gamma \in (1, 2)$.
Since $N_\gamma(\cE) \subset N_1(\cE)$, we deduce that the unique fixed point in the small ball of $N_1(\cE)$
in fact belongs to $N_\gamma(\cE)$ for all $\gamma \in (1, 2)$.

Let $(x_1, x_2)$ and $(\sh x_1, \sh x_2)$ be two pairs of trajectories such that $\sh x_j - x_j \in S_\nu$ for some $\nu > 1$,
$g := g(x_1, x_2)$ and $\sh g := g(\sh x_1, \sh x_2)$. Let $\beta \in (2, \nu + 2)$. We have
\begin{equation}
\begin{aligned}
&\|(\sh g - g, \partial_t(\sh g - g))\|_{N_{\beta - 1}(\cE)}\\
 &\quad= \|(\Psi(\sh x_1, \sh x_2, \sh g) - \Psi(x_1, x_2, g), \partial_t(\Psi(\sh x_1, \sh x_2, \sh g) - \Psi(x_1, x_2, g)))\|_{N_{\beta - 1}(\cE)} \\
&\quad \leq  \|(\Psi(\sh x_1, \sh x_2, \sh g) - \Psi(\sh x_1, \sh x_2, g), \partial_t (\Psi(\sh x_1, \sh x_2, \sh g) - \Psi(\sh x_1, \sh x_2, g)))\|_{N_{\beta - 1}(\cE)} \\
&\quad +\|(\Psi(\sh x_1, \sh x_2, g) - \Psi(x_1, x_2, g), \partial_t (\Psi(\sh x_1, \sh x_2, g) - \Psi(x_1, x_2, g)))\|_{N_{\beta - 1}(\cE)}.
\end{aligned}
\end{equation}
Bound \eqref{eq:prop-mapping-2} implies that the first term is $\ll \|(\sh g - g, \partial_t (\sh g - g))\|_{N_{\beta - 1}(\cE)}$.
Taking $\gamma = \frac 32$ in Proposition~\ref{prop:prop-mapping} (iii), which is allowed by \eqref{eq:g-l1-l2-est},
we obtain that the second term is $\lesssim \|\sh x - x\|_{S_\nu}$. This proves the second bound in \eqref{eq:lyap-schmidt-lip-l}.

Take $\gamma_1 \in (\max(1, \beta - \nu), 2)$ and $\gamma_2 := \beta + 1 - \gamma_1 < \nu + 1$.
As we have just proved, $\|(\sh g - g, \partial_t (\sh g - g))\|_{N_{\gamma_2}(\cE)} \lesssim \|\sh x - x\|_{S_\nu}$,
thus \eqref{eq:prop-mapping-2} yields
\begin{equation}
\|\Lambda_j(\sh x_1, \sh x_2, \sh g) - \Lambda_j(\sh x_1, \sh x_2, g)\|_{N_\beta \cap{\mathbf  W}_\beta}
\lesssim \|\Lambda_j(\sh x_1, \sh x_2, \sh g) - \Lambda_j(\sh x_1, \sh x_2, g)\|_{N_{\beta+1}} \lesssim \|\sh x - x\|_{S_\nu}.
\end{equation}
Now the bound on the first line in \eqref{eq:lyap-schmidt-lip-l} follows from \eqref{eq:prop-mapping-3}.
\end{proof}

\subsection{{Completion of the proof of Theorem~\ref{thm:main}} }

{First, we establish two preparatory lemmas.}

\begin{lemma}
\label{lem:traj-asym}
Let $1 < \gamma < 2$. Suppose that ${v} \in N_{\gamma + 1} \cap {\mathbf W}_{\gamma+1}$
and let $z(t)$ be a solution of
\begin{equation}
z''(t) = -2F(z(t)) + {v(t)}
\end{equation}
such that $|z(t) - 2\log t| \lesssim 1$ and $0 \leq z'(t) \lesssim t^{-1}$.
Then there exists $t_0 \in \bR$ such that
\begin{equation}
\label{eq:traj-asym}
\begin{aligned}
\big|z(t) - 2\log(A(t-t_0))\big| \lesssim t^{-\gamma}, \ \Big|z'(t) - \frac{2}{t-t_0}\Big| \lesssim t^{-\gamma - 1},
\ \Big|z''(t) + \frac{2}{(t-t_0)^2}\Big| \lesssim t^{-\gamma - 1}.
\end{aligned}
\end{equation}
\end{lemma}
\begin{proof} Since $z(t)\eee^{-2z(t)} \in N_{\gamma + 2} \subset {\mathbf W}_{\gamma + 1}$ for all $\gamma < 2$,
Lemma~\ref{lem:Fz} allows to replace $F(z)$ by $A^2\eee^{-z}$.

\textbf{Step 1.} We prove that
\begin{equation}
\label{eq:traj-asym-step1}
z(t) = 2\log(A t) + O(t^{1-\gamma}).
\end{equation}
Multiplying the equation by $z'(t)$ and integrating we obtain
\begin{equation}
\frac 12 (z'(t))^2 = 2A^2\eee^{-z(t)} - \int_t^\infty z'({\tau} ){v(\tau)}\ud {\tau}  = 2A^2\eee^{-z(t)} + O(t^{-\gamma - 1}),
\end{equation}
thus
\begin{equation}
(z'(t) - 2A\eee^{-\frac 12 z(t)})(z'(t) + 2A\eee^{-\frac 12 z(t)}) \lesssim t^{-\gamma - 1} \ \Rightarrow\ z'(t) = 2A\eee^{-\frac 12 z(t)} + O(t^{-\gamma}),
\end{equation}
which in turn implies
\begin{equation}
\big(\eee^{\frac 12 z(t)}\big)' = A + O(t^{1-\gamma})\ \Rightarrow\ \eee^{\frac 12 z(t)} = t(A + O(t^{1-\gamma})).
\end{equation}
Taking the logarithm, we obtain \eqref{eq:traj-asym-step1}.

\textbf{Step 2.} We improve. We set $z(t) = 2\log(A t) + u(t)$. In this step we prove $|u(t)| \lesssim t^{-1 + \epsilon}$, {for a fixed small number $\eps$}. From the equation we obtain
\begin{equation}
u''(t) = 2t^{-2}u(t) + {v(t)} + O(t^{-2}|u(t)|^2).
\end{equation}
Set
\begin{equation}
\wt u(s) := u(\eee^s).
\end{equation}
The Chain Rule yields
\begin{equation}
\wt u''(s) = 2\wt u(s) + \wt u'(s) + \eee^{2s}{v(\eee^s)} + O(|\wt u(s)|^2).
\end{equation}
This system diagonalises as follows:
\begin{align}
(\wt u(s) + \wt u'(s))' &= 2(\wt u(s) + \wt u'(s)) + \eee^{2s}{v(\eee^s)} + O(|\wt u(s)|^2), \label{eq:traj-asym-syst1} \\
(2\wt u(s) - \wt u'(s))' &= -(2\wt u(s) - \wt u'(s)) - \eee^{2s}{v(\eee^s)} + O(|\wt u(s)|^2). \label{eq:traj-asym-syst2}
\end{align}
Suppose $|\wt u(s)| \lesssim \eee^{-\beta s}$ for some $\beta > 0$. We begin with $\beta = \gamma - 1$ and in a finite
number of steps we will bootstrap this to $\beta = 1 - \epsilon$.
The assumption ${v}  \in W_{2, \gamma + 1}$ implies in particular that
\begin{equation}
\lim_{s\to \infty} \int^{\eee^{s}} \eee^{3\sigma}{v(\eee^\sigma)}\ud \sigma = \lim_{T\to \infty}\int^T \tau^2 {v(\tau)}\ud\tau
\end{equation}
exists and is finite. Thus \eqref{eq:traj-asym-syst2} yields
\begin{equation}
|\eee^s (2\wt u(s) - \wt u'(s))| \lesssim  \eee^{(1-2\beta)s}\ \Rightarrow\ |2\wt u(s) - \wt u'(s)| \lesssim  \eee^{-2\beta s}.
\end{equation}
From the assumption ${v} \in W_{-1, \gamma + 1}$ we have
\begin{equation}
\Big|\int_{s}^\infty {v(\eee^\sigma)} \ud \sigma\Big| = \Big|\int_{\eee^s}^\infty \tau^{-1}{ v(\tau) } \ud\tau\Big| \lesssim \eee^{-(\gamma + 2)s},
\end{equation}
so integrating \eqref{eq:traj-asym-syst1} we obtain
\begin{equation}
|\wt u(s) + \wt u'(s)| \lesssim \eee^{-2\beta s}\ \Rightarrow \ |\wt u(s)| \lesssim \eee^{-2\beta s}.
\end{equation}
Thus we can double the value of $\beta$, which concludes Step 2.

\textbf{Step 3.} We improve again, using the information obtained in Step 2. We obtain from \eqref{eq:traj-asym-syst1}
and \eqref{eq:traj-asym-syst2}
\begin{align}
(\wt u(s) + \wt u'(s))' &= 2(\wt u(s) + \wt u'(s)) + \eee^{2s}{v(\eee^s)} + O(\eee^{(-2 + \epsilon)s}), \label{eq:traj-asym-syst1bis} \\
(2\wt u(s) - \wt u'(s))' &= -(2\wt u(s) - \wt u'(s)) - \eee^{2s}{v(\eee^s)} + O(\eee^{(-2 + \epsilon)s}). \label{eq:traj-asym-syst2bis}
\end{align}
Equation \eqref{eq:traj-asym-syst2bis} is equivalent to
\begin{equation}
\big(\eee^s(2\wt u(s) - \wt u'(s))\big)' = -\eee^{3s}{v(\eee^s)} + O(\eee^{(-1 + \epsilon)s}).
\end{equation}
From the assumption ${v}  \in W_{2, \gamma + 1}$ we deduce that the limit $b = \lim_{s \to \infty} \big(\eee^s(2\wt u(s) - \wt u'(s))\big)$
exists and
\begin{equation}
2\wt u(s) - \wt u'(s) = b\eee^{-s} + O\big(\eee^{-\gamma s}\big).
\end{equation}
Equation \eqref{eq:traj-asym-syst1bis} is equivalent to
\begin{equation}
\big(\eee^{-2s}(\wt u(s) + \wt u'(s))\big)' = {v(\eee^s)}+ O(\eee^{(-4 + \epsilon)s}) \ \Rightarrow \ \wt u(s) + \wt u'(s) = O(\eee^{-\gamma s}).
\end{equation}
We conclude that
\begin{equation}
\begin{aligned}
\wt u(s) &= \frac b3 \eee^{-s} + O(\eee^{-\gamma s}), \\
\wt u'(s) &= - \frac b3\eee^{-s} + O(\eee^{-\gamma s}),
\end{aligned}
\end{equation}
which after a straightforward transformation yield the bounds for $z$ and $z'$ in \eqref{eq:traj-asym} with {$t_0:= \frac{b}{6}$}.
The bound on $z''$ follows from the bound on $z$, \eqref{eq:Fz} and the differential equation.
\end{proof}
\begin{lemma}
\label{lem:lin-eq-for-traj}
Let $\gamma > 1$. For any $f_1, f_2 \in N_{\gamma + 1} \cap {\mathbf W}_{\gamma + 1}$
there exists a unique solution $(y_1, y_2) \in S_\gamma$ of the system
\begin{equation}
\begin{aligned}
y_1'' &= t^{-2}(-(y_2 - y_1)) + f_1, \\
y_2'' &= t^{-2}(y_2 - y_1) + f_2,
\end{aligned}
\end{equation}
and this defines a bounded operator
$N_{\gamma + 1} \cap {\mathbf W}_{\gamma + 1} \owns (f_1, f_2) \mapsto (y_1, y_2) \in S_\gamma$.
\end{lemma}
\begin{proof}
By setting $z_1 := y_2 + y_1$ and $z_2 := y_2 - y_1$, we transform the system to two decoupled second order equations:
\begin{equation}
\begin{aligned}
z_1'' &= g_1, \qquad &g_1 := f_2 + f_1, \\
z_2'' &= 2t^{-2}z_2 + g_2, \qquad &g_2 := f_2 - f_1.
\end{aligned}
\end{equation}
Since $\gamma > 1 = \frac 12(\sqrt{1 + 4\times 2} - 1)$, Lemma~\ref{lem:Wab-prop} (iii) applies.
\end{proof}

\begin{proof}[Proof of Theorem~\ref{thm:main}]
Observe that $(x_1, x_2, g)$ solves \eqref{eq:dt2g} if and only if
\begin{equation}
\label{eq:dt2g-equiv}
g = g(x_1, x_2), \qquad \lambda_1(x_1, x_2) = \lambda_2(x_1, x_2) = 0.
\end{equation}
\textbf{Step 1.}
Fix $\gamma \in (1, 2)$. We will prove that there exists a unique solution $(x_1, x_2, g)$ of \eqref{eq:dt2g-equiv} such that
$\|(x_1(t), x_2(t)) - ({-}\log(At), \log(At))\|_{S_\gamma} \leq 1$ and $\|(g, \partial_t g)\|_{N_\gamma(\cE)} \leq 1$.
In particular, we obtain the same solution for all the values $\gamma \in (1, 2)$.

We define $(y_1, y_2)$ by $(x_1(t), x_2(t)) = ({-}\log(At) + y_1(t), \log(At) + y_2(t))$
and we set up a fixed point problem for $(y_1, y_2) \in S_\gamma$.
Given $(y_1, y_2) \in S_\gamma$, we define $(\wt y_1, \wt y_2) = \Theta(y_1, y_2)$ as the solution of the following system of differential equations:
\begin{align}
\wt y_1'' &= -t^{-2}(\wt y_2 - \wt y_1) - \lambda_1({-}\log(At) + y_1, \log(At) + y_2) + y_1'' + t^{-2}(y_2 - y_1), \\
\wt y_2'' &= t^{-2}(\wt y_2 - \wt y_1) + \lambda_2({-}\log(At) + y_1, \log(At) + y_2) + y_2'' - t^{-2}(y_2 - y_1).
\end{align}
We see that $\lambda_j({-}\log(At) + y_1, \log(At) + y_2) = 0$ for $j \in \{1, 2\}$
is equivalent to $(y_1, y_2)$ being a fixed point of $\Theta$.
In this proof, we denote
\begin{equation}
f_j(y_1, y_2, t) := (-1)^j \lambda_j({-}\log(At) + y_1(t), \log(At) + y_2(t)) + y_j''(t) - (-1)^j t^{-2}(y_2(t) - y_1(t)).
\end{equation}

We first check that $\Theta(0, 0) \in S_\gamma$. By Lemma~\ref{lem:lin-eq-for-traj}, it suffices to verify that
\begin{equation}
\label{eq:approx-in-Sg}
\lambda_j({-}\log(At), \log(At)) \in N_{\gamma + 1}\cap {\mathbf W}_{\gamma + 1}.
\end{equation}
By Lemma~\ref{lem:Fz} {and recalling the definition of $A$ in Proposition~\ref{prop:regime}}, we have $\log(At)'' + F(2\log(At)) \in N_{\gamma + 2} \subset N_{\gamma + 1} \cap {\mathbf W}_{\gamma + 1}$,
so \eqref{eq:approx-in-Sg} follows from Proposition~\ref{prop:lyap-schmidt}.

We now prove that $\Theta$ is a contraction in $S_\gamma$. Again by Lemma~\ref{lem:lin-eq-for-traj},
it suffices to verify that for any $c > 0$
\begin{equation}
\label{eq:rhs-traj-lip}
\|f_j(\sh y_1, \sh y_2, \cdot) - f_j(y_1, y_2, \cdot)\|_{N_{\gamma + 1}\cap {\mathbf W}_{\gamma + 1}} \leq c \|(\sh y_1, \sh y_2) - (y_1, y_2)\|_{S_\gamma},
\end{equation}
provided we take $T_0$ sufficiently large.
Let $z := x_2 - x_1 = 2\log(A t) + y_2 - y_1$ and $\sh z := \sh x_2 - \sh x_1 = 2\log(A t) + \sh y_2 - \sh y_1$.
For $z \gg 1$ and $|\sh z - z| \ll 1$ we have, by \eqref{eq:dFz},
\begin{equation}
\label{eq:force-diff}
|(F(\sh z) - F(z)) - 2A^2 (\eee^{-\sh z} - \eee^{-z})| \lesssim \bigg|\int_z^{\sh z}w\eee^{-2w}\ud w\bigg| \ll t^{-3}|\sh z - z|
\end{equation}
and
\begin{equation}
\begin{aligned}
\eee^{-\sh z} - \eee^{-z} &= (A t)^{-2}\big(\eee^{-(\sh y_2 - \sh y_1)} - \eee^{-(y_2 - y_1)}\big)
= -(A t)^{-2}\int_{y_2 - y_1}^{\sh y_2 - \sh y_1}(1 + O(w))\ud w \\
&= -(A t)^{-2}((\sh y_2 - \sh y_1)- (y_2 - y_1)) + o(t^{-3}|\sh y - y|),
\end{aligned}
\end{equation}
where the last inequality follows from $|\sh y_2 - \sh y_1| + |y_2 - y_1| \lesssim t^{-\gamma} \ll t^{-1}$.
Plugging this into \eqref{eq:force-diff} we obtain
\begin{equation}
\big\|\big(F(\sh x_2 - \sh x_1) - F(x_2 - x_1)\big) + t^{-2}\big((\sh y_2 - \sh y_1) - (y_2 - y_1)\big)\big\|_{N_{\gamma + 2}}
\ll \|\sh y - y\|_{N_\gamma}.
\end{equation}
Comparing this bound with \eqref{eq:lyap-schmidt-lip-l}, we get \eqref{eq:rhs-traj-lip}.

Invoking the Contraction Principle we obtain the unique solution $(x_1, x_2, g)$.
Set $\wt x_1(t) := -x_2(t)$, $\wt x_2(t) := -x_1(t)$ and $\wt g(t, x) := g(t, -x)$.
Observe that, by the symmetry of the problem, $(\wt x_1, \wt x_2, \wt g)$ also satisfies the requirements
stated at the beginning of Step 1. hence, by uniqueness, $g(t, -x) = g(t, x)$ and $x_2(t) = -x_1(t)$ for all $t$ and $x$.
We see that $\Phi := 1 - H_1 + H_2 + g$ is a strongly interacting kink-antikink pair satisfying
\eqref{eq:x-asympt} and \eqref{eq:error-est}, with $x(t) = x_2(t)$.

\noindent
\textbf{Step 2.} Let $(x_1, x_2, g)$ be a solution of \eqref{eq:dt2g}.
Then \eqref{eq:regime-3} implies that for arbitrary $\delta > 0$ there exists $T_0 > 0$ such that
\begin{equation}
\|(g, \partial_t g)\|_{N_1(\cE)} \leq \delta.
\end{equation}
Moreover, \eqref{eq:regime-1}, \eqref{eq:x'bound} and \eqref{eq:x''bound} imply
\eqref{eq:x1-x2-cond-1}, \eqref{eq:x1-x2-cond-2} and \eqref{eq:x1-x2-cond-3}.
Thus, Proposition~\ref{prop:lyap-schmidt} yields
\begin{equation}
(g, \partial_t g) \in N_\gamma(\cE),\qquad \text{for all }\gamma \in [1, 2),
\end{equation}
and, since $\lambda_j(x_1, x_2) = 0$,
\begin{equation}
(-1)^j x_j''(t) + F(x_2(t) - x_1(t)) \in N_{\gamma + 1} \cap {\mathbf W}_{\gamma + 1},\qquad\text{for all }\gamma \in [1, 2).
\end{equation}
Let $m(t) := x_2(t) + x_1(t)$, $z(t) := x_2(t) - x_1(t)$.
The function $z(t)$ satisfies
\begin{equation}
z''(t) + 2F(z(t)) \in {\mathbf W}_{1 + \gamma}\quad\text{for all }\gamma \in [1, 2).
\end{equation}
Thus, using \eqref{eq:regime-1} and \eqref{eq:regime-2}, the assumptions of
Lemma~\ref{lem:traj-asym} are satisfied, hence there exists $t_0$ such that for all $\gamma \in [1, 2)$,
\eqref{eq:traj-asym} holds.

The function $m(t)$ satisfies
\begin{equation}
m''(t) \in {\mathbf W}_{1 + \gamma}\quad \text{for all }\gamma \in [1, 2).
\end{equation}
Integrating in time and using $f \in W_{0, \gamma + 1}$,
$m'(t) \to 0$ we get $|m'(t)| \lesssim t^{-\gamma-1}$, in particular $x_0 := \frac 12\lim_{t \to \infty}m(t)$ is well-defined,
and $|m(t) - 2x_0| \lesssim t^{-\gamma}$.

We obtain
\begin{equation}
\begin{aligned}
&(x_1, x_2) - \big(x_0 - \log(A(t-t_0)), x_0 + \log(A(t-t_0))\big) \\
&= \Big(\frac{m - z}{2}, \frac{m + z}{2}\Big) - \big(x_0 - \log(A(t-t_0)), x_0 + \log(A(t-t_0))\big) \in S_\gamma,
\end{aligned}
\end{equation}
for all $\gamma \in (1, 2)$.
We deduce that, after translating in time by $t_0$ and in space by $x_0$, the triple $(x_1, x_2, g)$
satisfies the requirements of Step 1, and the conclusion follows.
\end{proof}
\begin{remark}
Our existence proof is constructive, as we can in principle obtain better and better approximate solutions
by the usual iteration scheme.
It can be seen {from} the proof that we obtain functions approximating the fixed point at arbitrary polynomial order in time.
Indeed, our proof in fact yields
\begin{equation}
\label{eq:power-improvement}
\|\Theta(\sh x) - \Theta(x)\|_{S_\beta} \lesssim \|\sh x - x\|_{S_\gamma},\quad\text{for any }\beta < \gamma + 1.
\end{equation}
\end{remark}

\bibliography{researchbib}

\bibliographystyle{plain}

\end{document}